\documentclass[11pt]{article}
\usepackage[margin=0.75in]{geometry}

\usepackage{amssymb,amsmath,mathtools,amsthm,thmtools,amsfonts}

\usepackage{comment}
\usepackage{thm-restate}

\usepackage{url} 
\usepackage{hyperref}
\hypersetup{
colorlinks,
linkcolor={black},
citecolor={black},
urlcolor={blue!60!black},
pdftitle={A Combinatorial Proof of Hilton’s Conjecture and Beyond},
pdfauthor={Thomas Lesgourgues, Luke Postle}}

\usepackage[utf8]{inputenc}
\usepackage{color}

\usepackage{enumitem}
\setlist[enumerate,1]{label={(\roman*)},noitemsep, topsep=0pt}
\setlist[enumerate,2]{label={(\alph*)},noitemsep, topsep=0pt}
\setlist[enumerate,3]{label={(\arabic*)},noitemsep, topsep=0pt}

\setlist[itemize]{nolistsep,noitemsep, topsep=0pt}

\newcommand{\bb}[1]{\mathbb{#1}}

\newcommand{\cA}{\ensuremath{\mathcal A}}
\newcommand{\cB}{\ensuremath{\mathcal B}}

\newcommand{\cE}{\ensuremath{\mathcal E}}
\newcommand{\cF}{\ensuremath{\mathcal F}}
\newcommand{\cG}{\ensuremath{\mathcal G}}
\newcommand{\cH}{\ensuremath{\mathcal H}}

\newcommand{\cK}{\ensuremath{\mathcal K}}
\newcommand{\cL}{\ensuremath{\mathcal L}}
\newcommand{\cM}{\ensuremath{\mathcal M}}

\newcommand{\cP}{\ensuremath{\mathcal P}}
\newcommand{\cQ}{\ensuremath{\mathcal Q}}
\newcommand{\cR}{\ensuremath{\mathcal R}}
\newcommand{\cS}{\ensuremath{\mathcal S}}
\newcommand{\cT}{\ensuremath{\mathcal T}}

\newcommand{\cV}{\ensuremath{\mathcal V}}
\newcommand{\cW}{\ensuremath{\mathcal W}}

\newcommand{\Expect}[1]{\ensuremath{\mathbb E\left[#1\right]}}
\newcommand{\Prob}[1]{\ensuremath{\mathbb P\left[#1\right]}}

\DeclarePairedDelimiter{\parens}{(}{)}
\DeclarePairedDelimiter{\set}{\{}{\}}
\DeclarePairedDelimiter{\brackets}{[}{]}
\DeclarePairedDelimiter{\floor}{\lfloor}{\rfloor}
\DeclarePairedDelimiter{\ceil}{\lceil}{\rceil}
\DeclarePairedDelimiter{\abs}{\lvert}{\rvert}   
\DeclarePairedDelimiter{\size}{\lvert}{\rvert}   

\newcommand{\cBoff}{\ensuremath{\cB_{{\rm off}}}}
\newcommand{\cBon}{\ensuremath{\cB_{{\rm on}}}}

\renewcommand{\div}[1]{\ensuremath{\mathrm{Div}\parens*{#1}}}

\renewcommand{\tilde}[1]{\widetilde{#1}}

\usepackage[noabbrev,capitalise]{cleveref}

\theoremstyle{plain}
\newtheorem{thm}{Theorem}[section]
\newtheorem{lem}[thm]{Lemma}
\newtheorem{proposition}[thm]{Proposition}

\newtheorem{cor}[thm]{Corollary}

\newtheorem{conj}[thm]{Conjecture}

\newenvironment{proofclaim}[1][\proofname]
{\proof[#1]}
{\endproof}

\declaretheorem[
  style=plain,
  name=Claim,
  within=theorem,
]{claim}

\newenvironment{lateproof}[1]
 {%
  \renewcommand{\theclaim}{\ref{#1}.\arabic{claim}}%
  \begin{proof}[Proof of~\cref{#1}]%
 }
 {\end{proof}}

\newenvironment{nonlateproof}[1]
 {%
  \renewcommand{\theclaim}{\ref{#1}.\arabic{claim}}%
  \begin{proof}%
 }
 {\end{proof}}

\usepackage{etoolbox}
 \AtEndEnvironment{proof}{\setcounter{claim}{0}}

\theoremstyle{plain} 
\newcommand{\thistheoremname}{}
\newtheorem{genericthm}{\thistheoremname}

\theoremstyle{definition}
\newtheorem{definition}[thm]{Definition}

\newtheorem{remark}[thm]{Remark}

\crefname{equation}{}{}
\crefname{lem}{Lemma}{Lemmas}
\crefname{claim}{Claim}{Claims}
\crefname{thm}{Theorem}{Theorems}
\crefname{enumi}{}{}

\numberwithin{equation}{section}

\usepackage{thm-restate}
\usepackage{soul}
\setstcolor{red}

\usepackage[dvipsnames, table]{xcolor}

\newcommand{\Erdos}{Erd\H{o}s}
\newcommand{\Kuhn}{K\"{u}hn}
\newcommand{\Lovasz}{Lov{\'a}sz}
\newcommand{\Renyi}{R\'{e}nyi}
\newcommand{\Sos}{S\'os}

\title{A Combinatorial Proof of Hilton's Conjecture and Beyond}

\author{
Thomas Lesgourgues
\thanks{School of Mathematics and Statistics, University of New South Wales, Australia. {\tt t.lesgourgues@unsw.edu.au}. Most of this work has been done while TL was at the University of Waterloo.}
\and
Luke Postle
\thanks{Combinatorics and Optimization Department,
University of Waterloo, Waterloo, Ontario N2L 3G1, Canada {\tt lpostle@uwaterloo.ca}. Partially supported by NSERC
under Discovery Grant No. 2019-04304.}}
\date{\today}

\usepackage{ifthen}

\usepackage{tikz}
\usepackage{calc}
\usetikzlibrary{calc,patterns,decorations.pathreplacing,backgrounds,shapes}

\tikzstyle{vertex}=[circle, draw, fill=black, inner sep=0pt, minimum size=4pt] 
\newcommand{\vertex}{\node[vertex]} 

\newcommand{\TriangleBooster}[1]
{
\begin{tikzpicture}[x=#1 cm, y=#1 cm]

    \draw [draw=black, fill=yellow!40] (0,0) -- (5,5) -- (5,3) -- cycle;
    \draw [draw=black, fill=yellow!40] (10,0) -- (5,3) -- (5,1) -- cycle;
    \draw [draw=black, fill=yellow!40] (0,0) -- (5,1) -- (5,-1) -- cycle;
    \draw [draw=black, fill=yellow!40] (10,0) -- (5,-1) -- (5,-3) -- cycle;
    \draw [draw=black, fill=yellow!40] (0,0) -- (5,-3) -- (5,-5) -- cycle;

    \draw [draw=black, fill=red!40] (10,0) -- (5,5) -- (5,3) -- cycle;
    \draw [draw=black, fill=red!40] (0,0) -- (5,3) -- (5,1) -- cycle;
    \draw [draw=black, fill=red!40] (10,0) -- (5,1) -- (5,-1) -- cycle;
    \draw [draw=black, fill=red!40] (0,0) -- (5,-1) -- (5,-3) -- cycle;
    \draw [draw=black, fill=red!40] (10,0) -- (5,-3) -- (5,-5) -- cycle;

    \vertex (u) at (0,0) [label=180:$u$]{};
    \vertex (v) at (10,0) [label=0: $v$]{};
    
    \foreach \i in {1, 2, 3, 4, 5, 6}{
        \pgfmathsetmacro\y{7-2*\i}
        \vertex (b\i) at (5,\y) [label=0:$b_{\i}$]{};
    }

    \begin{scope}[on background layer]
        \path (b1) edge[line width= 2pt, dashed, fill =red!40, out=180, in=180, looseness=2] (b6);
        \path (b1) edge[line width= 3pt, dashed] (v);
        \path (b6) edge[line width= 3pt, dashed] (v);
    \end{scope} 

    \draw [draw = black, fill = red!40] (0,-6) rectangle (1,-6.5);
    \node at (1,-6.25) [label=0:$\cBon$]{};

    \draw [draw = black, fill = yellow!40] (3,-6) rectangle (4,-6.5);
    \node at (4,-6.25) [label=0:$\cBoff$]{};

    \draw [draw = black] (6,-6.25) -- (7,-6.25);
    \node at (7,-6.25) [label=0:$B$]{};

    \draw [line width= 2pt, dashed, draw = black] (9,-6.25) -- (10,-6.25);
    \node at (10,-6.25) [label=0:$R$]{};

\end{tikzpicture}}

\begin{document}

\maketitle

\begin{abstract} 
Using refined absorption, we prove that for every integer $g\ge 1$ and real $\gamma > 0$, and for sufficiently large $n$, there exists an $n^{-1+\gamma}$-spread distribution on Latin squares of order $n$ and girth at least $g$ that have no proper subsquares.

This implies a combinatorial proof of Hilton's conjecture from the 1970s (recently proved algebraically by Allsop and Wanless) that for all sufficiently large $n$, there exists a subsquare-free Latin square of order $n$; indeed, it implies there exist at least $n^{(1-o(1))n^2}$ subsquare-free squares.

Simultaneously it also implies the existence of high girth Latin squares (recently proved by Kwan, Sah, Sawhney and Simkin) and even an $n^{-1+\gamma}$-spread distribution on high girth Latin squares.
\end{abstract}


\section{Introduction}

\subsection{Subsquare-Free Latin Squares of Arbitrarily High Girth}

A {\em Latin square} of order $n$ is an $n\times n$ matrix with entries from a set of $n$ symbols, such that each row and each column contains each symbol exactly once. These are a historically fundamental type of combinatorial designs, see for instance~\cite{laywine1998discrete,keedwell2015latin} for an account of their rich history, properties and applications. 

There exist a surprising number of open questions regarding some basic properties of Latin squares. Most notably, there is no known easily computable formula for the number of Latin squares of order $n$ (see e.g. Van Lint and Wilson~\cite[Chapter 17]{van2001course}). A fundamental area of research focuses on the existence of various substructures in Latin squares. Given a Latin square $L$ of order $n$, a {\em subsquare} of order~$k$ in~$L$ is a $k\times k$ submatrix of $L$ that is itself a Latin square, called {\em proper} if $2\leq k<n$. In particular, a subsquare of size $2$ is called an \emph{intercalate}. If $L$ has no proper subsquares, it is said to be {\em subsquare-free}. It is known since 1975 that there exist Latin squares with no intercalates for all orders (except for orders $2$ and $4$)~\cite{kotzig1975latin,kotzig1976certain,mcleish1975existence}. This question was in part motivated by a connection with disjoint Steiner triple systems~\cite{kotzig1975latin}. In 1974, Hilton conjectured that there should exist a subsquare-free Latin square of order $n$ for any $n$ large enough. This conjecture first appeared in~\cite[Problem 1.7]{keedwell2015latin}, although it is stated incorrectly there. Over the ensuing decades, constructions were developed for various values of $n$. In 1980, Heinrich~\cite{heinrich1980latin} constructed subsquare-free Latin squares of order $n = pq$ with $p,q$ distinct primes and $n\neq 6$; this was generalised in 1982 by Andersen and Mendelsohn~\cite{andersen1982direct} for $n$ divisible by a prime $p\geq 5$. In 2007, Maenhaut, Wanless, and Webb~\cite{maenhaut2007subsquare} proved their existence for odd $n$. This left values of $n$ of the form $n=2^x 3^y$ for some $x\geq1$, $y\geq0$, which was finally provided by Allsop and Wanless~\cite{allsop2023latin} in 2023. We also note that Allsop and Wanless~\cite{allsop2023latin} proved the existence of $d$-dimensional permutations with no sub-hypercubes for all $d\ge 2$. 

An intercalate is also an example of a {\em low girth} substructure. The \emph{girth} of a graph is the length of its shortest cycle. In particular, if the graph~$G$ is connected, the girth of~$G$ is the smallest~$k$ such that~$G$ has~$k$ vertices that span at least~$k$ edges. In \cite{linial2018challenges}, Linial defines a {\em cycle} in a Latin square $L$ to be a set of rows $A$, a set of columns $B$, and a set of symbols $C$, with $|A| + |B| + |C| > 3$, such that the $A \times B$ submatrix of~$L$ contains at least $|A| + |B| + |C|-2$ symbols from $C$. We define the {\em girth} of a Latin square $L$ to be the minimum of $|A| + |B| + |C| - 2$ over all such cycles in $L$\footnote{For consistency with the definition of girth in other design problems, this differs  from Linial's definition by an additive constant of $2$.}. It is not hard to see that no Latin square has girth smaller than $4$, and that a Latin square has girth $4$ if and only if it contains an intercalate. Linial then naturally asked if there exist Latin squares with arbitrarily high girth. 

These definitions and questions were in part motivated by the Brown-\Erdos-\Sos~\cite{BES73} problem in extremal hypergraph theory, and in particular by an old conjecture of \Erdos~\cite{E73} on the existence of high-girth Steiner triple systems, which was recently proved by Kwan, Sah, Sawhney, and Simkin~\cite{KSSS2022STS}. 
The same group answered Linial's question positively, showing that there exist Latin squares with arbitrarily high girth~\cite{KSSS2023substructures}. 

In this article, we start by extending Hilton's conjecture to the setting of high-girth Latin squares.
\begin{thm}\label{thm:ExistHighGirthSQLS}
For every integer $g\geq 1$, there exists an integer $n_0 >0 $ such that for all $n\ge n_0$, there exists a subsquare-free Latin square of order $n$ and girth at least $g$.
\end{thm}

It must be noted at that point that our proof of~\cref{thm:ExistHighGirthSQLS} is not a mere extension of the proof of Kwan, Sah, Sawhney, and Simkin~\cite{KSSS2023substructures} to the subsquare-free setting, or vice versa, of the proof of Hilton's conjecture to the high-girth setting. Thus we provide new independent proofs of both results. Indeed, we provide a new, purely combinatorial, proof of Hilton's conjecture, using tools recently developed by Delcourt and the second author~\cite{DPI}, called refined absorbers. These allow us to extend our results into many directions, including the high girth version from~\cref{thm:ExistHighGirthSQLS}, but also adding spread, counting and thresholds properties, as presented below.

For the rest of this article, we will view Latin squares as graph decompositions. For graphs $F$ and $G$, an {\em $F$-packing} of $G$ is a collection $S$ of pairwise edge-disjoint subgraphs of $G$, each being isomorphic to $F$, and an {\em $F$-decomposition} of $G$ is an $F$-packing $S$ of $G$ in which every edge of $G$ is in an element of $S$. We sometimes call $K_3$-decompositions (respectively, $K_3$-packings) by triangle-decompositions (respectively, triangle-packings). A Latin square $S$ of order $n$ then corresponds to a triangle decomposition of the complete tripartite graph $K_{n,n,n}$, and a subsquare of $S$ of order $k$ corresponds to a triangle packing in $K_{n,n,n}$ that is also a triangle decomposition of $K_{k,k,k}$. In that setting, a cycle of length $j$ in $L$, often called a {\em $(j+2,j)$-configuration}, is a set of $j$ edge-disjoint triangles spanning at most $j+2$ vertices, and the girth of a packing is naturally the length of its shortest cycle. In particular, a Latin square of order $n$ containing a subsquare of order $k$ is a triangle packing in $K_{n,n,n}$ with girth at most $3k-2$ (though the converse does not hold); and therefore a Latin square of order $n$ with girth at least $g$ contains no subsquare of order at most $\frac{g+1}{3}$. Since every triangle-packing of $K_{n,n,n}$ has girth at least $4$, the cases $g\in\{1,2,3\}$ of our final results follow from the case $g=4$. We therefore state those results for every $g\geq1$, but formulate the technical high-girth treasury results only for $g\geq 4$.

\subsection{A Highly Spread Existence Theorem}

Determining thresholds of combinatorial properties in random hypergraphs is a major research question in combinatorics, dating back to the seminal result of \Erdos{} and \Renyi{}~\cite{erdHos1966existence} for the appearance of perfect matchings in the binomial random graph $G(n,p)$. Shamir’s problem~\cite{schmidt1983threshold} asks for the threshold for the appearance of perfect matchings in $G^{(3)}(n,p)$, the random $n$-vertex $3$-uniform hypergraph in which each triangle is present with probability $p$. This question was solved (up to a constant factor) in the seminal paper by Johansson, Kahn, and Vu~\cite{johansson2008factors}, with sharp threshold and hitting time results obtained in later work of Kahn~\cite{kahn2022hitting,kahn2023asymptotics}. 
Building on recent progress on the Sunflower Conjecture by Alweiss, Lovett, Wu and Zhang~\cite{alweiss2020improved}, Park and Pham~\cite{PP22} proved the Kahn–Kalai conjecture, while the fractional version (conjectured by Talagrand~\cite{talagrand2010many}) was obtained earlier by Frankston, Kahn, Narayanan, and Park~\cite{frankston2021thresholds}. Using an observation by Talagrand~\cite{talagrand2010many}, these breakthrough results can be used to deduce upper bounds on thresholds for a wide variety of problems, in particular yielding a much simpler proof of the resolution of Shamir's problem by Johansson, Kahn, and Vu. One of the key concepts at the center of these questions concerns the notion of {\em spread} probability distributions, defined as follows for triangle-packings.

\begin{definition}\label{def:Spread}
    Let $\cF,\cS$ be families of triangle-packings of $K_{n,n,n}$. A probability distribution over $\cF$ is {\em $q$-spread} if for all triangle-packings $S$ of $K_{n,n,n}$ and a random triangle-packing $P\in\cF$ (chosen according to the distribution) we have
    \[\Prob{S\subseteq P}\leq q^{|S|}.\]
    We say that the probability distribution is {\em $q$-spread on $\cS$} if the above holds for all $S\in \cS$.
\end{definition}

A natural follow-up to Shamir's problem concerns $\cG^{(3)}(n,n,n,p)$, the random tripartite $3$-uniform hypergraph on parts of size $n$ in which each balanced $3$-set is included independently with probability $p$. The question of determining the threshold for when $\cG^{(3)}(n,n,n,p)$ contains a Latin square
is commonly attributed to an unpublished 2006 manuscript
of Johansson. A trivial necessary condition is that each edge must be contained in at least one triangle, which happens when $p=\Omega(\log n/n)$. Luria and Simkin~\cite{luria2019threshold} conjectured the threshold to be $\Theta(\log n/n)$. Jain and Pham~\cite{JP22}, and independently Keevash~\cite{K22}, solved this conjecture by constructing an $O(1/n)$-spread distribution on Latin squares in the complete tripartite $3$-uniform graph. Previous progress on this conjecture had been done by Sah, Sawhney, Simkin~\cite{SSS23} with an $O(n^{o(1)-1})$-spread distribution, and later by Khan, Kelly, K\"{u}hn, Methuku, and Osthus~\cite{KKKMO22}, improving the spread parameter to $O(\log n/n)$.  Both arguments used some form of iterative absorption. Kelly further conjectured (see the comment after \cite[Corollary~1.7]{divoux2026subsquares}) that the {\em uniform} distribution over Latin squares of order $n$ is $((e^2 +o(1))/n)$-spread. If true, this would imply an alternative proof of the recent results on threshold probabilities for Latin squares via Park–Pham's Theorem~\cite{PP22}. We note that, recently, Allsop and Morris~\cite{AM2026} proved that the uniform distribution over Latin squares of order $n$ is $O(1/n)$-spread on $\cS$, where $\cS$ is a given (large) family of sparse $K_3$-packings of $K_{n,n,n}$. We in fact prove a much stronger form of~\cref{thm:ExistHighGirthSQLS}, proving that there exists a highly spread distribution over high girth subsquare-free Latin squares as follows.

\begin{thm}\label{thm:SpreadSubsquare}
For every integer $g\geq 1$ and real $\gamma > 0$, there exists an integer $n_0 > 0$ such that for all integers $n\ge n_0$, there exists an $n^{-1+\gamma}$-spread distribution over subsquare-free Latin squares of order $n$ and girth at least $g$.
\end{thm}

There is famously no known easily computable formula for the number of Latin squares of order $n$, and the best known lower bound is $\parens*{e^{-2}n-o(n)}^{n^2}$ (see e.g. Van Lint and Wilson~\cite[Chapter 17]{van2001course}). Kwan, Sah, Sawhney, and Simkin~\cite{KSSS2023substructures} showed that the number of Latin squares of order $n$ with no intercalates is at least $\parens*{e^{-9/4}n-o(n)}^{n^2}$, and as observed by Allsop and Wanless~\cite{allsop2023latin}, now that the existence of subsquare-free Latin squares is settled, the most natural question asks for asymptotic estimates of their number.~\cref{thm:SpreadSubsquare} yields the following counting result as an immediate corollary.

\begin{cor}\label{cor:counting}
For every integer $g\geq 1$ and real $\gamma > 0$, there exists an integer $n_0 >0$ such that for all integers $n\ge n_0$, there are at least $n^{(1-\gamma)n^2}$ subsquare-free Latin squares of order $n$ and girth at least $g$.
\end{cor}

It must be noted here that spreadness is a key component of our proof of Hilton's conjecture, and not simply an additional optional property. We use spreadness to avoid subsquares, and our most general theorem could in fact be used to prove the existence of Latin squares avoiding any substructure that is avoided via spreadness with some positive probability. We also note that our proof does not require the Park-Pham Theorem~\cite{PP22} or its fractional version by Frankston, Kahn, Narayanan, and Park~\cite{frankston2021thresholds}. However, in conjunction with the existence of a well-spread distribution over subsquare-free Latin squares from~\cref{thm:SpreadSubsquare}, the Park-Pham Theorem implies the following upper bound on the threshold for the appearance of subsquare-free Latin squares (which is trivially at least $\Omega(\log n/n)$).

\begin{cor}\label{cor:threshold}
    For every integer $g\geq 1$ and real $\gamma > 0$, and every sufficiently large integer $n$, if $p\geq n^{\gamma-1}$ then $\cG^{(3)}(n,n,n,p)$ contains a subsquare-free Latin square of order $n$ and girth at least $g$ with high probability (i.e. with probability going to $1$ as $n\to\infty$).
\end{cor}

\subsection{A More General Theorem}

As explained, spreadness is the key strategy to prove the existence of subsquare-free Latin squares. Throughout this article, for a given Latin square of order $n$, we partition potential subsquares into three categories, depending on their size. We call subsquares {\em small} if they have bounded size, {\em large} if they have size at least $n^c$ for some fixed constant $c$, and {\em medium} otherwise. The rationale for this partition is that avoiding small subsquares is a consequence of being high girth, while randomness should be sufficient to avoid the appearance of all other subsquares. However an $n^{-1+\gamma}$-spread distribution is not spread-enough to avoid the large subsquares. We show that such a spread distribution avoids all medium subsquares, but avoiding large subsquares would require to be within a factor $2$ of an optimal spreadness of $e^2/n$. This level of spreadness is already an open problem for general Latin squares, while we would further need it for the family of high girth Latin squares.

The key idea then is to choose a set $Y\subseteq K_{n,n,n}$ with small maximum degree that intersects all large subsquares, and prove that we can achieve a form of ``super-spreadness'' for the triangles containing edges of $Y$, namely a $\frac{1+\varepsilon}{n}$-spread distribution, simultaneously with an $n^{-1+\gamma}$-spread distribution over high-girth Latin squares. We remark that we cannot use the same methodology directly for medium subsquares too, as we cannot guarantee the existence of a set $Y\subseteq K_{n,n,n}$ that intersects all medium subsquares and has small enough maximum degree. 

For a given set $Y\subseteq K_{n,n,n}$, we say that a $K_3$-packing $S$ of $K_{n,n,n}$ is {\em $Y$-disjoint} if $|E(T)\cap E(Y)|\le 1$ for all $T\in S$ and {\em $Y$-enclosed} if $|E(T)\cap E(Y)|=1$ for all $T\in S$. Here then is our main result.

\begin{thm}\label{thm:main}
    For every integer $g\geq 4$ and for every reals $\gamma,\varepsilon\in(0,1)$, there exist integers $a \ge 1$ and $n_0 > 0$ such that the following holds for all $n\ge n_0$: Let $Y \subseteq K_{n,n,n}$  with $\Delta(Y)\leq \frac{n}{\log^{ag} n}$, let $\cL_D$ be the family of $Y$-disjoint $K_3$-decompositions of $K_{n,n,n}$ with girth at least $g$, and let $\cS$ be the family of $Y$-enclosed $K_3$-packings of $K_{n,n,n}$. Then
    there exists a probability distribution $\nu$ over $\cL_D$ such that $\nu$ is $n^{-1+\gamma}$-spread and $\frac{1+\varepsilon}{n}$-spread on $\cS$.
\end{thm}

We show how Theorem~\ref{thm:SpreadSubsquare} follows from Theorem~\ref{thm:main} in~\cref{sec:ProofSqFree} by choosing $Y$ appropriately and then showing that the $n^{-1+\gamma}$-spreadness implies the existence of a medium subsquare happens with low probability while the $\frac{1+\varepsilon}{n}$-spreadness on $\cS$ implies the existence of a large subsquare happens with low probability.

We note then that the concept of medium subsquares could be replaced by any family of triangle-packings of $K_{n,n,n}$ which happen with low probability given $n^{-1+\gamma}$-spreadness. Similarly large subsquares can be replaced with any $Y$-enclosed family $\cS$ of packings which happen with low probability given $\frac{1+\varepsilon}{n}$-spreadness for some appropriately chosen $Y$.

\subsection{A High Girth Completion Theorem}

The celebrated Nash-Williams' Conjecture~\cite{nash1970unsolved} focuses on triangle decompositions in graphs with large minimum degree. It asserts that, subject to divisibility conditions, every large enough graph on $n$ vertices with minimum degree at least $3n/4$ has a triangle decomposition. Many constructions show that the ratio $3/4$ is tight, see for instance~\cite{DP2021progress}. After decades of partial progress, this conjecture was recently resolved in full by Delcourt and Postle~\cite{DP26}. Their proof represents a major breakthrough in graph decomposition theory, settling a problem that had guided much of the development of modern absorption methods for designs. The resolution of Nash-Williams' Conjecture also clarifies the role of earlier progress on triangle decompositions. Barber, \Kuhn{}, Lo, and Osthus~\cite{BKLO16} showed that the existence of $K_3$-decompositions is related to the existence of so-called fractional $K_3$-decompositions. Using this relation, it is known from Garaschuk's thesis~\cite{garaschuk2014linear} in 2014 that a ratio of 0.956 is sufficient. In 2015, Dross~\cite{dross2016fractional} improved it to 0.9, and the best result before the resolution of the problem was due to Delcourt and Postle~\cite{DP2021progress} who showed that a ratio of $0.82733$ is sufficient. Their recent work proves the conjectured threshold $3/4$ itself, first at the fractional level and then, via a stability analysis and absorption, for exact triangle decompositions.
These results can be viewed as {\em completion} theorems. Indeed, Nash-Williams' Conjecture implies that, for $n$ large enough, any $K_3$-packing of $K_n$ in which any vertex is in at most $n/8$ triangles can be extended into a $K_3$-decomposition of $K_n$. 

Similar problems have been naturally studied in the partite setting, i.e.~for Latin squares. A partial Latin square of order $n$ is a partially filled $n\times n$ grid of cells satisfying the conditions of a Latin square, or equivalently a $K_3$-packing of $K_{n,n,n}$. In 1981, Smetaniuk~\cite{smetaniuk1981new}, and independently Anderson and Hilton~\cite{andersen1983thank}, showed that if at most $n-1$ entries have been made, then a partial Latin square can always be completed into a Latin square (and this bound is best possible). In 1983, Daykin and H{\"a}ggkvist~\cite{daykin1983completion} conjectured that this is also true for any partial Latin square in which we have used each row, column and symbol at most $n/4$ times, yielding a transposition of Nash-Williams' Conjecture in the Latin square setting. After progress by Chetwynd and H\"{a}ggkvist~\cite{chetwynd1985completing}, 
Gustavsson~\cite{gustavsson1991decompositions} and Bartlett~\cite{bartlett2013completions}, Barber, K{\"u}hn, Lo, Osthus, and Taylor~\cite{BKLOT2017clique} proved that this is true for any partial Latin square in which we have used each row, column and symbol less than $n/25$ times. They used iterative absorption to reduce the completion problem to the corresponding fractional decomposition problem, and their result then follows from progress on the fractional problem by Bowditch and Dukes~\cite{BD2019}.  Yu and Feng~\cite{YF2026} recently improved the completion density from $1/25$ to $2/25$, adapting ideas from previous progress on the Nash-Williams' Conjecture by Delcourt and Postle~\cite{DP2021progress}.

One could then naturally conjecture a similar statement in high girth setting; ``Let $S$ be a high girth $K_3$-packing of $K_{n,n,n}$ and $G:=K_{n,n,n}\setminus E(S)$. If $G$ has high enough minimum degree, then there exists a high-girth $K_3$-decomposition of $G$''. However such a statement would not amount to a completion theorem. Indeed one needs to ensure that $S$ and the $K_3$-decomposition of $G$ have high girth collectively. It is not clear if a minimum degree condition on $G$ is sufficient to ensure that there exists a high girth $K_3$-decomposition of $K_{n,n,n}$, unless perhaps if $S$ is quasi-random. 

In order to prove~\cref{thm:main}, one could try to follow the outline of the proof of existence conjecture via refined absorption~\cite{DPI}, adapting it to the partite, high girth and spread settings, while also taking into account $Y$-enclosed packings at every step of the proof. It seemed more natural, and much easier, to deal with $Y$-enclosed packings in a first {\em pre-seeding} step, and then to adapt existing results to work for completion of these packings.

The following theorem is a transposition of Gustavsson's conjecture in a high-girth setting. The minimum degree condition is complemented by a regularity requirement on the {\em design-treasury} of the host graph. This concept, introduced by Delcourt and Postle~\cite{DPII}, captures the necessary conditions for the application of a nibble/forbidden-submatching approach to form a large packing, coupled with absorption techniques to complete it into a full decomposition. Informally, a treasury is composed of three hypergraphs: the first one (the design-hypergraph) encodes the family of cliques that are available to decompose most of the edges of $G$ via a probabilistic / nibble approach; the second one (the reserve-hypergraph), encodes additional cliques that might be used to augment this family, forcing the non-covered edges into a small set of preselected ones; and the third hypergraph (the configuration-hypergraph) encodes packings that have to be avoided when building a decomposition, e.g. packings with low girth. A perfect matching of this treasury is then a decomposition of the host graph avoiding undesirable configurations. Our two-step ``pre-seeding/completion'' approach imposes another new requirement on this design-treasury, called abundance; indeed the configuration-hypergraph additionally encodes packings that must be avoided, not because of their low girth, but because they would induce a low girth cycle when combined with the packing built during the pre-seeding step. We then need all our constructions to avoid these ``forbidden configurations''. In particular, the absorption techniques rely on the existence of ``sufficiently many'' small graphs (called boosters, see~\cref{def:booster}) presenting some useful decomposition properties. The treasury is then said to be abundant if most of these small graphs avoid the forbidden configurations. While the abundance of a treasury seems related to its regularity, it cannot be directly inferred by it. We refer the reader to~\cref{def:Treasury,def:DesignTreasury} for the definitions of treasuries, and to~\cref{sec:AbundantTreasuries,sec:RegularTreasuries} for a formal introduction of regularity and abundance.

\begin{restatable}[]{thm}{MinDegreeThm}\label{thm:MinDegree}
    For every integer $g\geq 4$, there exist an integer $n_0 > 0$ and reals~$\alpha,\beta,\sigma\in(0,1)$ such that the following holds for all $n\ge n_0$: Let $G \subseteq K_{n,n,n}$ be a $K_{1,1,1}$-divisible graph with $\delta(G)\geq (2-\sigma) n$ such that there exists a subtreasury $T$ of ${\rm Treasury}^g(G,G,K_3)$ that is $(n,\alpha n,\beta,\alpha)$-regular and $g$-abundant. Then there exists an $n^{-1+2/g}$-spread probability distribution over perfect matchings of~$T$.
\end{restatable}

In~\cref{sec:proof_main}, we prove that~\cref{thm:MinDegree} implies the following completion theorem.

\begin{cor}[Completion]\label{cor:completion}
    For every integer $g\geq 4$, there exist an integer $n_0 > 0$ and reals $\alpha,\beta\in(0,1)$ such that the following holds for all $n\ge n_0$:  Let $S$ be a $K_3$-packing of $K_{n,n,n}$ with girth at least $g$ and $\Delta(S) \le \frac{n}{\log^{g-1} n}$, such that there exists a subtreasury of ${\rm Proj}^g(K_{n,n,n},S,K_3)$ that is $(n,\alpha n,\beta,\alpha)$-regular and $g$-abundant. Then there exists a $K_3$-decomposition of $K_{n,n,n}$ with girth at least $g$ that contains $S$.
\end{cor}

For a given set $Y\subseteq K_{n,n,n}$, we say that a $K_3$-packing $S$ of $K_{n,n,n}$ is {\em $Y$-fully-enclosed} if $S$ is $Y$-enclosed and for every edge $e\in Y$, there exists $T\in S$ containing $e$. Finally, the pre-seeding step is as follows, finding a ``super-spread'' probability-distribution over $Y$-fully-enclosed packings. We will combine it with~\cref{thm:MinDegree}, applying the latter to the ``left-overs'' of the pre-seeding theorem; i.e.~to $G$ being formed by $K_{n,n,n}$ after removing the edges of the pre-seeding packing. The treasury ${\rm Proj}^g(K_{n,n,n},S,K_3)$, called a projection treasury (see~\cref{def:ProjectionPacking}), is similar to the treasury of~\cref{thm:MinDegree}, with a larger family of forbidden configurations, to ensure that the union of the two packings (coming from~\cref{thm:MinDegree,thm:preseeding}) is also not a forbidden configuration.

\begin{thm}[Pre-seeding]\label{thm:preseeding}
    For every integer $g\geq 4$ there exists $a \ge 1$ such that, for every reals $\alpha,\beta,\varepsilon\in(0,1)$, there exists $n_0 > 0$ such that the following holds for all $n\ge n_0$: Let $Y \subseteq K_{n,n,n}$  with $\Delta(Y)\leq \frac{n}{\log^{ag} n}$. Let $\cL_P$ be the family of $Y$-fully-enclosed $K_3$-packings $S$ of $K_{n,n,n}$ with girth at least $g$ and $\Delta(S) \le \frac{4n}{\log^{a(g-1)} n}$, such that there exists a subtreasury of ${\rm Proj}^g(K_{n,n,n},S,K_3)$ that is $(n,\alpha n,\beta,\alpha)$-regular and $g$-abundant. Then $\cL_P$ is non-empty, and there exists a $\frac{1+\varepsilon}{n}$-spread probability distribution over $\cL_P$.
\end{thm}

We demonstrate that combining~\cref{thm:MinDegree,thm:preseeding} yields~\cref{thm:main} in~\cref{sec:proof_main}.

\subsection{Notation}\label{ss:Notation}

A \emph{hypergraph} $H$ consists of a pair $(V,E)$ where $V=V(H)$ is a set whose elements are called \emph{vertices} and $E=E(H)$ is a set of subsets of $V$ called \emph{edges}; for brevity, we also write $H$ for $E(H)$, hence identifying a hypergraph with its edge-set. Similarly, we write $v(H)$ for the number of vertices of $H$ and either $e(H)$ or alternatively $|H|$ for the number of edges of $H$. 
For an integer $r\ge 1$, a hypergraph $H$ is said to be \emph{$r$-bounded} if every edge of $H$ has size at most $r$ and  \emph{$r$-uniform} if every edge has size exactly $r$; an \emph{$r$-uniform hypergraph} is called an \emph{$r$-graph} for short. For a hypergraph $G$ and subset $S\subseteq V(G)$, we let $G(S)$ denote the set $\{e\in G: S\subseteq e\}$. Note this differs from the standard notation $G(S)$ used to denote the related concept of a `link hypergraph' where the set $S$ is removed from every edge of $G(S)$.\smallskip

Let $G$ be a hypergraph. For an integer $k\ge 1$, we let $G^{(k)}$ denote the $k$-uniform subhypergraph of $G$ consisting of all edges of $G$ of size $k$. If $G$ is uniform, then for a vertex $v \in V(G)$, we let $d_G(v)$ denote the number of edges of $G$ containing $v$. If $G$ is not uniform, then for an integer $i\ge 1$, we thus write $d_{G^{(i)}}(v)$ for the \emph{$i$-degree of $v$ in $G$}, which is defined to be the number of edges of $G$ of size $i$ containing $v$. 
If $G$ is uniform, then \emph{maximum $i$-codegree of $G$}, denoted $\Delta_i(G)$ is the maximum of $|G(U)|$ for $U\subseteq V(G)$ with $|U|=i$. If $G$ is $k$-uniform, then the \emph{maximum degree of $G$}, denoted $\Delta(G)$ is $\Delta_{k-1}(G)$; similarly we write $\delta(G)$ for the \emph{minimum degree of $G$} which is $\delta_{k-1}(G)$. If $G$ is not uniform, then the \emph{maximum $(s,t)$-codegree of $G$}, denoted $\Delta_{t}\left(G^{(s)}\right)$ is the maximum of $|G^{(s)}(U)|$ for $U\in \binom{V(G)}{t}$. \smallskip

For a graph $G\subseteq K_{n,n,n}$ on vertex set $V(G)=V^1\sqcup V^2\sqcup V^3$, we write $G_{i,j} = G[V^i \cup V^j]$ for the graph induced by the edges between $V^i$ and $V^j$, and $e(V^i,V^j)$ for $e(G[V^i \cup V^j])$. Throughout, indices are taken modulo 3. We say that $G\subseteq K_{n,n,n}$ is {\em edge-balanced} if for some integer $m\geq 0$ and for all $i\in[3]$ we have $e(V^i,V^{i+1})=m$, and
that $G$ is {\em $K_{1,1,1}$-divisible} if for all $i\in[3]$ and all vertex $v\in V^i$, \[\size{\set{u\in V^{i-1},uv\in E(G)}} = \size{\set{u\in V^{i+1},uv\in E(G)}}.\] 
For a given graph $G\subseteq K_{n,n,n}$, we denote by $\div{G}$, the set of all $K_{1,1,1}$-divisible subgraphs of $G$. Given a family $S$ of graphs, we write $E(S)$ for $\bigcup_{H\in S}E(H)$, $V(S)$ for $\bigcup_{H\in S}V(H)$, and $\Delta(S)$ for the maximum degree of the graph induced by the edges in $E(S)$. 

\subsection{Outline of Paper}\label{ss:Outline}

\cref{sec:OurProofOverview} gives an overview of the proof strategy. In particular, we explain the division into small, medium, and large subsquares, describe the pre-seeding and completion steps, and prove that~\cref{thm:MinDegree,thm:preseeding} together imply~\cref{thm:main} as well as~\cref{cor:completion}.
In~\cref{sec:ProofSqFree}, we prove~\cref{thm:SpreadSubsquare} from~\cref{thm:main}. 
\cref{sec:PreSeeding} is devoted to the pre-seeding step. We introduce the required notions of rooted boosters, treasuries, abundance, projections, and quantum packings, and prove the quantum pre-seeding theorem,~\cref{thm:SpreadQuantumPreseeding}. We then deduce the pre-seeding theorem,~\cref{thm:preseeding}, from~\cref{thm:SpreadQuantumPreseeding}.
In~\cref{sec:AdditionalTools}, we collect the tools required for the completion step. These include the reserve lemma,~\cref{thm:reserve}, the refined Latin-omni-absorber theorem,~\cref{thm:LatinAbsorberThm}, the spread Latin-omni-absorber theorem,~\cref{thm:MainLatinAbsorbers}, the regularity-boosting lemma,~\cref{lem:RegBoost}, and the spread version of the forbidden-submatching theorem,~\cref{thm:ForbiddenSubmatchingReservesSpread}.
In~\cref{sec:ProofCompletion}, we combine the tools from~\cref{sec:AdditionalTools} to prove the spread minimum-degree completion theorem,~\cref{thm:MinDegree}.
\cref{sec:HighGirthAbsorption} is devoted to the spread absorption argument underlying~\cref{thm:MainLatinAbsorbers}. We introduce Latin omni-boosters and their projections, prove the conversion from vertex-spread to triangle-spread in~\cref{lem:VertexToTriangleSpread}, and prove the omni-booster theorem,~\cref{thm:HighGirthOmniBooster}. We then deduce~\cref{thm:MainLatinAbsorbers}.
Finally,~\cref{sec:ConcludingRemarks} discusses several possible extensions and open questions arising from our work.

\section{Proof Overview}\label{sec:OurProofOverview}

\subsection{Small, Medium, and Large Subsquares}

We recall that for a given Latin square of order $n$, we partition potential subsquares into three categories, depending on their size. We call subsquares {\em small} if they have size at most $k_0$ for some fixed constant $k_0$, {\em large} if they have size at least $n^c$ for some fixed constant $c$, and {\em medium} otherwise. Avoiding small subsquares is an immediate consequence of being high girth, while randomness should be sufficient to avoid the appearance of all other subsquares. Indeed, if we denote by $L(n)$ the number of Latin squares of order $n$, it is known (e.g.~see~\cite[Theorem 4.13]{ball2024course}) that 
\[\frac{(n!)^{2n}}{n^{n^2}}\leq L(n)\leq \prod_{i=1}^n(i!)^{n/i}.\]

Both upper and lower bounds can be obtained by studying the permanent of a matrix associated with a Latin square. The upper bound follows from the Bregman-Minc inequality~\cite{bregman1973some}, while the lower bound follows from a result by Egorychev~\cite{egorychev1981solution} on permanent of doubly stochastic matrices. Using Stirling's formula, it follows that the number of Latin squares of order $n$ is asymptotically equal to
\begin{equation}\label{eq:EnumLS}
L(n)=\exp\parens*{n^2(\log n-2+o(1))}.
\end{equation}
Recall that the existence of a $q$-spread distribution $\nu$ over high-girth Latin squares implies that the probability for a fixed $K_3$-packing $S$ to be contained in a Latin square (selected at random using $\nu$) is at most $q^{|S|}$. In particular, as there exist at most $\binom{n}{k}^3L(k)$ subsquares of order $k$ in $K_{n,n,n}$, the probability that such a random Latin square contains a subsquare is at most (by union bound) $\sum \binom{n}{k}^3L(k)\cdot q^{k^2}$. 

Avoiding all medium and large subsquares would follow from an ``optimal'' spreadness, namely a $\frac{(1+\varepsilon)e^2}{n}$-spread distribution for some $\varepsilon<1$, which is not even known to exist over Latin-squares (dropping the high girth requirement). An $n^{-1+\gamma}$-spread distribution over high-girth Latin squares is spread-enough to avoid all medium subsquares since $\sum \binom{n}{k}^3L(k)\cdot n^{(\gamma-1)k^2} \le \sum \parens*{\frac{ne}{k}}^{3k} k^{k^2} \cdot n^{(\gamma-1)k^2}$ tends to $0$ as $n$ goes to infinity provided $k_0\ge 4$ (for the small terms) and $c < 1-\gamma$ (for the large terms). However, it is not spread-enough to avoid all large subsquares. The key idea then is to choose a set $Y\subseteq K_{n,n}$ of cells, that is pairs (row, column) in a Latin square, with small maximum degree and such that $Y$ intersects each large subsquare of order $k$ in at least $k\log^2 n$ row-column pairs. We prove that we can achieve a form of ``super-spreadness'', namely a $\frac{(1+\varepsilon)}{n}$-spread distribution $\nu$ for some $\varepsilon<1$, on the family of $K_3$-packings of $K_{n,n,n}$ such that every triangle contains exactly one edge of $Y$, i.e. on the family of $Y$-enclosed $K_3$-packings of~$K_{n,n,n}$.

If a Latin square contains a large subsquare of size $k$, we know that the set of cells $R$ of this subsquare contains at least $k\log^2 n$ edges from $Y$. This induces an $Y$-enclosed $K_3$-packing of size $|Y\cap R|$. There are at most $\binom{n}{k}^2$ choices for $R$, and given that $R$ is contained in a Latin square of size $k$, at most $\binom{n}{k}k^{|Y\cap R|}$ ways to ``complete'' the $Y$-enclosed $K_3$-packing of size $|Y\cap A|$ using at most $k$ symbols.

Using the ``super-spreadness'' of the probability distribution $\nu$ on the family of $Y$-enclosed $K_3$-packings of $K_{n,n,n}$, we know that the probability that a random Latin square (selected at random using $\nu$) contains such a fixed packing $S$ is at most $\parens*{\frac{(1+\varepsilon)}{n}}^{|S|}$. By union bound, the probability that a random Latin-square contains a large subsquare is then at most $\sum \binom{n}{k}^3\parens*{\frac{k(1+\varepsilon)}{n}}^{k\log^2 n}$ $\le \sum \parens*{\frac{ne}{k}}^{3k} \parens*{\frac{k(1+\varepsilon)}{n}}^{k\log^2 n} \le \sum e^{3k} \parens*{\frac{k(1+\varepsilon)}{n}}^{k(\log^2 n-3)}$
, summing over all large sizes $k$ (up to $n/2$), which tends to $0$ as $n$ goes to infinity.

\subsection{Pre-seeding Step}\label{sec:PreSeedingIntro}

The existence of such a spread distribution is achieved in two steps.~\cref{thm:preseeding} is a {\em pre-seeding} step, followed by the min-degree/completion argument of~\cref{thm:MinDegree}. During the pre-seeding step, for a given set of edges $Y\subseteq K_{n,n,n}$, our goal is to find a set $S$ of edge-disjoint triangles in $K_{n,n,n}$ such that each triangle contains one edge from $Y$, and every edge of $Y$ is contained in a triangle of $S$ (a property called $Y$-fully-enclosed) with the additional properties that: a) $S$ has high girth and low maximum degree, b) there exists a regular enough and abundant subtreasury of the projection treasury of $S$, and c) a super-spread distribution over such $S$. 

The problem for (a) is similar to the ``bipartite matching" problem from~\cite{DP22} but as $Y$ is relatively small, the edge-disjointness and high-girth requirements could be handled using the \Lovasz{} Local Lemma~\cite{EL73}, while the required maximum-degree bound can be obtained using the slotting argument from~\cite[Lemma~4.5]{DKPIV}. To then also obtain (c) - existence of a spread-enough distribution over such packings - we could use a variant of the Local Lemma (see Haeupler, Saha, and Srinivasan~\cite{HSS11} and Jain and Pham~\cite{JP22}), but this still leaves the problem of (b) - to ensure that we will be able to extend any such packing into a high-girth $K_3$-decomposition of $K_{n,n,n}$, we want the ``left-overs'' to be regular-enough; formally we require the girth-$g$ projection treasury of $S$ to contain a regular enough subtreasury. This problem does not seem amenable to Local Lemma based techniques. Thus we employ the \emph{quantum} approach from~\cite{DPII} used to resolve a similar issue in the proof of the High-Girth Existence Conjecture. This allows us to forego using the Local Lemma altogether as follows.

Namely, each edge of $Y$ chooses not one triangle but some set of triangles (independently with some small probability $p$). Mentally, it is best to consider this as some \emph{quantum} choice where the triangle is not yet fixed. We show that if $p$ is small enough, then with high probability all possible choices of triangles from the resulting sets will satisfy (b); more formally, the common projection treasury of all the possible choices has a regular enough subtreasury (and hence any individual choice does as well). It then remains to show for each edge of $Y$ that almost all remaining choices for its triangle will be high girth with all other choices (the low maximum degree is ensured by all the quantum sets together with high probability). Picking the triangle randomly from this good set then yields the desired super-spread distribution. 

\subsection{A Spread Minimum Degree High Girth Theorem}\label{sec:OutlineCompletion}

\cref{thm:MinDegree} is then the min-degree/completion argument, proving that we can indeed extend any such packing into the desired $K_3$-decomposition. This is achieved more generally by showing that any $K_{1,1,1}$-divisible regular-enough (in the regular-treasury sense) graph $G\subseteq K_{n,n,n}$, with high-enough minimum degree, contains a suitable high-girth $K_3$-decomposition. We follow the strategy laid out by Delcourt and Postle~\cite{DPII} in their proof of the high girth existence conjecture. 
\begin{enumerate}[label=(\arabic*)]
    \item Reserve a random subset $X$ of $E(G)$.
    \item Construct a high girth omni-absorber $A$ of $X$ that {\em does not shrink too many configurations}.
    \item Regularity boost $G\setminus (A\cup X)$.
    \item Apply a ``Forbidden Submatchings with reserves''~\cite{DP22} theorem to find a $K_3$-packing of $G\setminus A$ covering $G\setminus (A\cup X)$ and then extend this to a $K_3$-decomposition of $G$ by definition of omni-absorber.  
\end{enumerate}

We follow this plan with the following notable modifications. The first obvious one is that each step has to be done in a tripartite setting. In particular, the existence of partite absorbers uses the refined efficient omni-absorber theorem for Latin squares recently announced by Delcourt and Postle~\cite{DP2026Partite}. Contrary to~\cite{DPII}, the host graph $G$ is not complete, and the absorption has to be achieved in a min-degree / regular-subtreasury setting. We show that this is not a huge gap, as most of the required ingredients from~\cite{DPII} already exist in a high minimum-degree setting. Passing to a regular subtreasury requires only minor changes: the required lower bounds on degree estimates follow from regularity, while the relevant upper bounds are preserved when available triangles are deleted. The most challenging aspect of the proof then is related to spreadness, and specifically showing that the triangles used in the decomposition family of our high girth omni-absorbers can be embedded in a spread way. To do so, we start with a refined omni-absorber $A_0$, and replace every triangle in its decomposition family by a private $g$-sphere. As in~\cite{DPII}, we first choose a small random set of candidate spheres for every root. This quantum-booster argument shows that every root retains many available spheres, that their union has low maximum degree, and that the common projection treasury remains regular. Choosing one sphere for each root then yields a vertex-spread distribution over omni-boosters. The new spread ingredient is~\cref{lem:VertexToTriangleSpread}, which converts this vertex-spreadness into spreadness of the triangles appearing in the on/off decompositions. Informally, prescribing several triangles from one sphere still determines sufficiently many of its non-root vertices. This yields the spread distribution over omni-absorbers required in~\cref{thm:MainLatinAbsorbers}.

\subsection{Proofs of Completion}\label{sec:proof_main}

The following two propositions relate the treasuries of~\cref{thm:MinDegree,thm:preseeding}, and show how they complement each other to build a high-girth $K_3$-decomposition of $K_{n,n,n}$. At this step, these are the only statements needed to prove~\cref{thm:main} from~\cref{thm:MinDegree,thm:preseeding}, relying on the formal definitions of treasuries. We therefore defer their proofs to~\cref{sec:GirthProjection}.  

\begin{restatable}[]{proposition}{ProjSubTreasury}
\label{prop:ProjSubTreasury}
    For any $g\geq 4$, let $S$ be a $K_3$-packing of $K_{n,n,n}$ and let $G:=K_{n,n,n}\setminus E(S)$. Then ${\rm Proj}^g(K_{n,n,n},S,K_3)$, the girth-$g$ projection of $S$, is a subtreasury of ${\rm Treasury}^g(G,G,K_3)$.
\end{restatable}

\begin{restatable}[]{proposition}{CompletionTreasury}
\label{prop:completion_treasury}
For a fixed integer $g\geq4$, let $S$ be a $K_3$-packing of $K_{n,n,n}$ with girth at least $g$. If there exists a perfect matching $M$ of the girth-$g$ projection treasury ${\rm Proj}^g(K_{n,n,n},S,K_3)$, then $S\cup M$ is a $K_3$-decomposition of $K_{n,n,n}$ with girth at least $g$.
\end{restatable}

Using these propositions, we show how the completion~\cref{cor:completion} follows from the ``minimum degree''~\cref{thm:MinDegree},  and how the main~\cref{thm:main} on high-girth $K_3$-decompositions follows using additionally the pre-seeding~\cref{thm:preseeding}. We start with the proof of the corollary.

\begin{proof}[Proof of Completion~\cref{cor:completion}]
    Let $\alpha,\beta,\sigma,n_0$ be such that~\cref{thm:MinDegree} holds. For a large enough integer $n\geq n_0$ (in particular such that $\sigma\cdot \log^{g-1}n>1$), let $S$ be a $K_3$-packing of $K_{n,n,n}$ with girth at least $g$ and $\Delta(S) \le \frac{n}{\log^{g-1} n}$, such that there exists a subtreasury $T$ of ${\rm Proj}^g(K_{n,n,n},S,K_3)$ that is $(n,\alpha n,\beta,\alpha)$-regular and $g$-abundant. Let $G:=K_{n,n,n}\setminus E(S)$, and observe that $G$ is $K_{1,1,1}$-divisible, and that, with $n$ large enough, $G$ has minimum degree $2n-\Delta(S)\geq (2-\sigma)n$.
    By~\cref{prop:ProjSubTreasury},~${\rm Proj}^g(K_{n,n,n},S,K_3)$ is a subtreasury of ${\rm Treasury}^g(G,G,K_3)$ and therefore so is $T$. By~\cref{thm:MinDegree}, as $T$ is a subtreasury of ${\rm Treasury}^g(G,G,K_3)$ that is $(n,\alpha n,\beta,\alpha)$-regular and $g$-abundant, there exists a perfect matching  $M$ of $T$. By~\cref{prop:SubTreasury} $M$ is also a perfect matching of ${\rm Proj}^g(K_{n,n,n},S,K_3)$, and by~\cref{prop:completion_treasury} $S\cup M$ is a $K_3$-decomposition of $K_{n,n,n}$ of girth at least $g$.
\end{proof}

We are now ready to show that~\cref{thm:main} follows from~\cref{thm:preseeding,thm:MinDegree}. 

\begin{proof}[Proof of~\cref{thm:main}]
    
    Fix an integer $g\geq 4$ and reals $\gamma,\varepsilon\in(0,1)$. Let $h$ be an integer such that $h\geq g$ and $h>2/\gamma$. Let $\alpha,\beta,\sigma,n_1$ be such that~\cref{thm:MinDegree} holds with parameter $h$. Let $a_0$ and $n_2$ such that~\cref{thm:preseeding} holds with parameters $h,\alpha,\beta,\varepsilon$. Let $n\geq \max\{n_1,n_2\}$ be sufficiently large. Let  $a:=\ceil*{\frac{a_0h}{g}}$, and observe that 
    $\Delta(Y)\leq \frac{n}{\log^{ag} n}\leq \frac{n}{\log^{a_0h} n}$.

    Let $\cL_P$ be the family of $Y$-fully-enclosed $K_3$-packings $P$ of $K_{n,n,n}$ with girth at least $h$ and $\Delta(P)\leq \frac{4n}{\log^{a_0(h-1)}n}$, such that there exists a subtreasury of ${\rm Proj}^h(K_{n,n,n},P,K_3)$ that is $(n,\alpha n,\beta,\alpha)$-regular and $h$-abundant. Let $\eta_1$ be the $\parens*{\frac{1+\varepsilon}{n}}$-spread probability distribution over $\cL_P$, as guaranteed by~\cref{thm:preseeding}. Let $P$ be a packing chosen at random in $\cL_P$ using the probability distribution $\eta_1$, with $T$ denoting the regular and abundant subtreasury of ${\rm Proj}^h(K_{n,n,n},P,K_3)$ also guaranteed by~\cref{thm:preseeding}. For any packing $S$ of $K_{n,n,n}$ we have
    \[\Prob{S\subseteq P}\leq\parens*{\frac{1+\varepsilon}{n}}^{|S|}.\]

    Let $G:=K_{n,n,n}\setminus E(P)$ and observe that $G$ is $K_{1,1,1}$-divisible, with minimum degree $\delta(G)=2n-\Delta(P)\geq(2-\sigma) n$. By~\cref{prop:ProjSubTreasury},~${\rm Proj}^h(K_{n,n,n},P,K_3)$ is a subtreasury of ${\rm Treasury}^h(G,G,K_3)$ and therefore so is $T$ by transitivity of the subtreasury relation (see~\cref{def:SubTreasury}).
    
    As $T$ is a subtreasury of ${\rm Treasury}^h(G,G,K_3)$ that is $(n,\alpha n,\beta,\alpha)$-regular and $h$-abundant, by~\cref{thm:MinDegree} there exists an $\parens*{n^{-1+2/h}}$-spread probability distribution $\eta_2$ over perfect matchings $M$ of $T$. For any packing $S$ of $K_{n,n,n}$ and any $P'\in \cL_P$ we have
    \[\Prob{~S\subseteq M\mid P=P'~}\leq\parens*{\frac{1}{n^{1-2/h}}}^{|S|}.\]
    
    Let $Z=P\cup M$. By~\cref{prop:SubTreasury}, $M$ is also a perfect matching of ${\rm Proj}^h(K_{n,n,n},P,K_3)$, and by~\cref{prop:completion_treasury} $Z$ is a $K_3$-decomposition of $K_{n,n,n}$ with girth at least $h\geq g$. Given that $P$ is $Y$-enclosed while $E(M)\cap E(Y)=\emptyset$, we obtain that $Z$ is $Y$-disjoint. We now show that this probability distribution is $n^{-1+\gamma}$-spread over $\cL_D$, and $(1+\varepsilon)/n$-spread on $\cS$. To that end, fix a packing $S$ of $K_{n,n,n}$, and let $S_1:=\{F\in S\colon E(F)\cap E(Y)\neq\emptyset\}$, and $S_2:=S\setminus S_1$. Recall that $P$ is $Y$-fully-enclosed, therefore $S\subseteq (P\cup M)$ if and only if $S_1\subseteq P$ and $S_2\subseteq M$. We obtain
        \begin{align*}
            \Prob{S\subseteq Z}&= \Prob{(S_1\subseteq P)\wedge (S_2\subseteq M)}\\    
            &= \Prob{S_2\subseteq M\mid S_1\subseteq P}\cdot \Prob{S_1\subseteq P}\\  
            &\leq \Prob{S_1\subseteq P}\cdot\parens*{\sum_{P'\in\cL_P}\Prob{S_2\subseteq M\mid P=P'}\cdot \Prob{P=P'\mid S_1\subseteq P}}\\  
            &\leq 
            \parens*{\frac{1+\varepsilon}{n}}^{|S_1|}
            \parens*{\frac{1}{n^{1-2/h}}}^{|S_2|}
            \brackets*{\sum_{P'\in\cL_P}\Prob{P=P'\mid S_1\subseteq P}}\\  
            &= \parens*{\frac{1+\varepsilon}{n}}^{|S_1|}
            \parens*{\frac{1}{n^{1-2/h}}}^{|S_2|}.
        \end{align*}

    If $S\in\cS$, then $S$ is $Y$-enclosed and every triangle in $S$ intersect $Y$ on one edge, therefore $S_1=S$, $S_2=\emptyset$ and
    \[\Prob{S\subseteq Z}\leq \parens*{\frac{1+\varepsilon}{n}}^{|S|}.\]
    Otherwise
    \[\Prob{S\subseteq Z}\leq \parens*{\frac{1+\varepsilon}{n}}^{|S_1|}\parens*{\frac{1}{n^{1-2/h}}}^{|S_2|}\leq \parens*{\frac{1}{n^{1-\gamma}}}^{|S|},\]    
    since $2/h< \gamma$.
\end{proof}

\section{Subsquare-free Latin Squares}\label{sec:ProofSqFree}

We now prove that~\cref{thm:main} implies~\cref{thm:SpreadSubsquare}, following the proof strategy laid out in~\cref{sec:OurProofOverview}. We start by showing the existence of a set $Y\subseteq K_{n,n}$, with small maximum degree and such that $Y$ intersects each large subsquare of order $k$ in many cells.

\begin{lem}\label{lem:ExistenceYIntersectingAllLarge}
For every $a \ge 1$ and $c\in (0,1)$, there exists an integer $n_0$ such that for all $n\ge n_0$ there exists a set $Y\subseteq K_{n,n}$ with $\Delta(Y)\leq n/\log^a n$ such that for all integers $k$ with $\frac{n}{2} \ge k \ge n^c$ and every $R\subseteq K_{n,n}$ with $R\cong K_{k,k}$ we have $|E(Y)\cap E(R)|\ge k \log^2 n$.
\end{lem}

\begin{proof}
    Let $p :=\frac{1}{2\log^a n}$, and $Y$ be a random subgraph of $K_{n,n}$ keeping each edge independently at random with probability $p$. With high probability we have $\Delta(Y)<n/\log^a n$. Let $n$ be large enough, and let $\cP$ be the probability that there exists a set $R\subseteq K_{n,n}$, with $R\cong K_{k,k}$, $\frac{n}{2} \ge k \ge n^c$  and $|E(Y)\cap E(R)|\leq k \log^2 n$. It suffices to show that $\cP<1$.

    For a fixed set of edges $R\subseteq K_{n,n}$ with $R\cong K_{k,k}$ and $n^c\leq k \leq \frac{n}{2}$, let $N=|E(Y)\cap E(R)|$. We have $N\sim{\rm Bin}(k^2,p)$, $\mathbb{E}N=pk^2$ and by the Chernoff bound  (e.g.~see Alon and Spencer~\cite{AS16}), for any $\delta\in(0,1)$,
    \[\Prob{N<k \log^2 n} < \Prob{N<(1-\delta)\mathbb{E}N} = \exp\brackets*{-\delta^2\frac{pk^2}{2}},\]
    where we used $k\geq n^c$ and therefore $k \log^2 n\ll pk^2$ for the first inequality. By union bound, we have
    \[\cP\leq \sum_{k=n^c}^{n/2}\binom{n}{k}^2\exp\brackets*{-\delta^2\frac{pk^2}{2}}\leq\sum_{k=n^c}^{n/2}\exp\brackets*{2k\log n-\delta^2\frac{pk^2}{2}}.\]
    With $k\geq n^c$, we have that $2k\log n-\delta^2\frac{pk^2}{2}$ decreases with $k$, and we conclude that
    \[\cP\leq n \exp\brackets*{2n^c\log n-\delta^2\frac{pn^{2c}}{2}}\leq 
    \exp\brackets*{\log n + 2n^c\log n-\delta^2\frac{n^{2c}}{4\log^a n}}=o(1).\qedhere\]
\end{proof}

We are now ready to prove~\cref{thm:SpreadSubsquare}, using spreadness over the family of high-girth $K_3$-packings of $K_{n,n,n}$, and ``super-spreadness'' on the family of $Y$-enclosed $K_3$-packings of $K_{n,n,n}$.

\begin{lateproof}{thm:SpreadSubsquare}
    For an integer $g\geq 1$ and reals $\gamma,\varepsilon\in(0,1)$, let $c,\theta>0$ be arbitrarily small reals, in particular with $\theta<\gamma$. Let $g_0:=\max\{g,11\}$, $k_0=\floor*{\frac{g_0+1}{3}}\geq 4$, and let $n$ be a large enough integer. We say that a Latin square of order $k$ is {\em small} if $k\leq k_0$, {\em large} if $k\geq k_1:=n^c$, and {\em medium} otherwise. 

    Let $a$ be an integer such that~\cref{thm:main} holds with parameters $g_0$, $\theta$, and $\varepsilon$. Let $Y\subseteq K_{n,n}$ be a set of cells in $K_{n,n,n}$ with $\Delta(Y)\leq \frac{n}{\log^{ag_0}n}$, as guaranteed by~\cref{lem:ExistenceYIntersectingAllLarge}. Let $\cL_D$ be the family of $Y$-disjoint $K_3$-decompositions of $K_{n,n,n}$ with girth at least $g_0$, and let $\cS$ be the family of $Y$-enclosed $K_3$-packings of $K_{n,n,n}$. \Cref{thm:main} yields a probability distribution $\nu$ over $\cL_D$ that is $n^{-1+\theta}$-spread and $\frac{1+\varepsilon}{n}$-spread on $\cS$.   

    It is not difficult to see that any Latin square of order $n$ contains no proper subsquare of order larger than $n/2$. Meanwhile, any Latin square from $\cL_D$ has girth at least $g_0\geq 3k_0-1$, and therefore contains no small subsquare (of order at most $k_0$). 

    \begin{claim}\label{claim:no_medium}
        Let $P$ be Latin square chosen at random from $\cL_D$ using the probability distribution $\nu$. Then 
        \[\mathbb{P}_\nu\brackets*{P\in\cL_D\text{ contains a medium subsquare}}\leq 1/3.\]        
    \end{claim}

    \begin{proofclaim}
        Let $P$ be a Latin square chosen at random using the probability distribution $\nu$, and let $q:=n^{-1+\theta}$. It follows from~\eqref{eq:EnumLS} that there exists a constant $C$ such that, for any $k$, there are at most $\exp\brackets*{k^2(\log k+C)}$ Latin squares of order $k$. Given that $\nu$ is $q$-spread over $\cL_D$, by the union bound,
        \begin{align*}
            \mathbb{P}_\nu\brackets*{P\in\cL_D\text{ contains a medium subsquare}} 
            &\leq \sum_{k=k_0}^{k_1} {\binom{n}{k}}^3 L(k)\cdot q^{k^2}\\
             &\leq \sum_{k=k_0}^{k_1}n^{3k}\cdot \exp\brackets*{k^2(\log k+C)}\cdot n^{(-1+\theta)k^2}\\
             &\leq \sum_{k=k_0}^{k_1}\exp\brackets*{k^2\cdot \parens*{\log k+C+\frac{3}{k}\log n - (1-\theta)\log n}}\\
        \end{align*}
        With $k\leq k_1 = n^c$, we have $\log k\leq c\log n$ and
        \begin{align*}
        \mathbb{P}_\nu\brackets*{P\in\cL_D\text{ contains a medium subsquare}} 
             &\leq \sum_{k=k_0}^{k_1}\exp\brackets*{k^2\cdot \log n\cdot \parens*{c + o(1) + \frac{3}{k}-1+\theta}}\\
             &\leq \sum_{k=k_0}^{k_1}\exp\brackets*{k^2\cdot \log n\cdot \parens*{c + o(1) + \frac{3}{k_0}-1+\theta}}
        \end{align*}
        With $n$ large enough, $k_0\geq 4$ and $c,\theta$ arbitrarily small, we have $c + o(1) + \frac{3}{k_0}-1+\theta<0$, and 
        \begin{align*}
        \mathbb{P}_\nu\brackets*{P\in\cL_D\text{ contains a medium subsquare}} 
             &\leq \sum_{k=k_0}^{k_1}\exp\brackets*{k_0^2\cdot \log n\cdot \parens*{c + o(1) + \frac{3}{k_0}-1+\theta}}\\
              &\leq n^c\exp\brackets*{k_0^2\cdot \log n\cdot \parens*{c + o(1) + \frac{3}{k_0}-1+\theta}}\\
             &\leq \exp\brackets*{k_0^2\cdot \log n\cdot \parens*{c + o(1) + \frac{3}{k_0}-1+\theta + \frac{c}{k_0^2}}}\\
             &\leq1/3,
        \end{align*}
        using again that $n$ is large enough, while $k_0\geq 4$ and $c,\theta$ are arbitrarily small.
    \end{proofclaim}

    \begin{claim}\label{claim:no_large}
        Let $P$ be a Latin square chosen at random from $\cL_D$ using the probability distribution $\nu$. Then 
        \[\mathbb{P}_\nu\brackets*{P\in\cL_D\text{ contains a large subsquare}}\leq 1/3.\]      
    \end{claim}
    
    \begin{proofclaim}
    
    With $q:=(1+\varepsilon)/n$, let $P$ be a Latin square of order $n$, chosen at random using the probability distribution $\nu$, and assume that $P$ contains a large subsquare $B$. For some integer $k\in[n^c,n/2]$, there exists $R\cong K_{k,k}$, representing the cells $B$, and by definition of $Y$ we have that $|E(Y)\cap E(R)|\geq k\log^2 n$.

    Consider the $K_3$-packing $S$ of $K_{n,n,n}$ induced by $B$, keeping only triangles intersecting the set $Y$,
    \[S:=\set*{T\in B\colon |E(Y)\cap E(T)|=1}.\]

    This is a $K_3$-packing of size $|E(Y)\cap E(R)|$ that is $Y$-enclosed, hence in $\cS$, and containing at most $k$ symbols. Therefore there are at most $\binom{n}{k}k^{|E(Y)\cap E(R)|}$ possible packings $S$ for any given $R$, and observe that there are at most $\binom{n}{k}^2$ choices for $R$. As $\nu$ is $q$-spread on $\cS$, we obtain by the union bound,
    
    \begin{align*}
    \mathbb{P}_\nu\brackets*{P\in\cL_D\text{ contains a large subsquare}}
    &\leq \mathbb{P}\brackets*{~\exists R \cong K_{k,k} \subseteq P\textrm{ s.t. $R$ induces a $Y$-enclosed packing $S$}}\\
             &\leq\sum_{k=k_1}^{n/2}\sum_{\substack{R\subseteq K_{n,n}\\R\cong K_{k,k}}} \binom{n}{k} k^{|E(Y)\cap E(R)|}\cdot q^{|S|}\\
            &=\sum_{k=k_1}^{n/2}\sum_{\substack{R\subseteq K_{n,n}\\R\cong K_{k,k}}} \binom{n}{k} (kq)^{|E(Y)\cap E(R)|}.
    \end{align*} 

    With $q=\frac{1+\varepsilon}{n}$ and $k\leq n/2$, we obtain that $kq<1$ for $n$ large enough, and by~\cref{lem:ExistenceYIntersectingAllLarge}, we have $|E(Y)\cap E(R)|\geq k\log^2 n$. Therefore
    \begin{align*}
    \mathbb{P}_\nu\brackets*{P\in\cL_D\text{ contains a large subsquare}}
             &\leq\sum_{k=k_1}^{n/2}\binom{n}{k}^3\cdot(kq)^{k\log^2 n}\\
             &=\sum_{k=k_1}^{n/2}\binom{n}{k}^3\cdot\parens*{\frac{(1+\varepsilon)k}{n}}^{k\log^2 n}\\
             &\leq\sum_{k=k_1}^{n/2} n^{3k}\cdot \exp\brackets*{k\cdot \log^2 n\cdot (\log k + \log(1+\varepsilon)-\log n)}\\
             &\leq\sum_{k=k_1}^{n/2} n^{3k}\cdot \exp\brackets*{k\cdot \log^2 n\cdot ( \log(1+\varepsilon)-\log 2)}.
    \end{align*}  
    Let $c' = (\log 2 - \log(1+\varepsilon))>0$, then 
    \begin{align*}
    \mathbb{P}_\nu\brackets*{P\in\cL_D\text{ contains a large subsquare}}
             &\leq\sum_{k=k_1}^{n/2} n^{3k}\cdot \exp\brackets*{-c'k\log^2 n}\\
             &\leq\sum_{k=k_1}^{n/2} \exp\brackets*{(3-c'\log n)\cdot k\cdot \log n}.
    \end{align*}  
    For $n$ large enough we have $3-c'\log n<-1$, therefore $(3-c'\log n)k\log n$ is decreasing with $k$ and 
    \begin{align*} 
    \mathbb{P}_\nu\brackets*{P\in\cL_D\text{ contains a large subsquare}}
             &\leq\sum_{k=k_1}^{n/2} \exp\brackets*{-n^c\log n }\\
             &\leq \exp\brackets*{\log n -n^c\log n }\\
             &\leq 1/3.\qedhere
    \end{align*} 
    \end{proofclaim}

Let $\cL$ be the family of subsquare-free Latin squares of girth at least $g$ and order $n$.
For a Latin square $P\in\cL_D$, let $\cE$ be the event ``$P$ contains no medium and no large subsquare". We define $\eta$ to be the probability distribution $\nu$ conditioning on $\cE$. 
Recall that every Latin square in $\cL_D$ has girth at least $g_0\geq g$, and contains no small subsquare. Therefore, if $P$ is a Latin square from $\cL_D$, selected at random using the probability distribution $\eta$, then $P$ is a random element from $\cL$. We obtain that for any non-empty $K_3$-packing $S$ of $K_{n,n,n}$, we have
\[\mathbb{P}_\eta\brackets*{S\subseteq P} = \mathbb{P}_\nu\brackets*{S\subseteq P\mid P\text{ is subsquare-free}} 
\leq \frac{\mathbb{P}_\nu\brackets*{S\subseteq P}}{\mathbb{P}_\nu\brackets*{P\text{ is subsquare-free}}}. \]
Recall that $\nu$ is $n^{-1+\theta}$-spread, while, by the above claims, we have $\mathbb{P}_\nu\brackets*{P\text{ is subsquare-free}}\geq \frac{1}{3}$. We obtain
\[\mathbb{P}_\eta\brackets*{S\subseteq P} \leq 3(n^{-1+\theta})^{|S|}\leq n^{(-1+\gamma)|S|},\]
with $n$ large enough and $\theta<\gamma$, while the case $S=\emptyset$ is trivially true. 
\end{lateproof}

We reiterate that the existence of a $\frac{(1+\varepsilon)e^2}{n}$-spread distribution over high-girth Latin squares would be sufficient to prove the existence of subsquare-free Latin squares, by simple modifications of the proof of~\cref{claim:no_medium}. In light of this, we conjecture the existence of the following distribution.

\begin{conj}
    For every integer $g\geq4$, there exists a $\frac{(1+o(1))e^2}{n}$-spread distribution over the family of Latin squares of order $n$ and girth at least $g$.
\end{conj}


\section{Pre-seeding}\label{sec:PreSeeding}

This section is dedicated to the proof of the pre-seeding theorem,~\cref{thm:preseeding}. 

\subsection{Probabilistic Tools - Kim-Vu Theorem}

We require the following corollary of the polynomial concentration theorem of Kim-Vu~\cite{KV00}, included in a more general form in~\cite{DKPIII}, to prove concentration results on degrees and codegrees of hypergraphs in various treasuries. 
\begin{cor}\label{cor:KimVu}
For all integers $k\ge 1$, there exists $n_0 \ge 1$ such that the following holds for all $p\in (0,1]$ and $n\ge n_0$. Let $H$ be a $k$-uniform multi-hypergraph on at most $n$ vertices, and $H_p$ be the subgraph of $H$ induced on the set $X$ where $X$ is obtained by choosing each vertex of $H$ independently with probability $p$. If $\kappa>0$ satisfies
\begin{equation}
    p^{i}\ge \frac{\Delta_i(H)}{\kappa}\cdot \log^{4k+2} n\tag{$*_i$}
\end{equation}
for all $i \in [k]$ and $e(H)\leq \kappa$, then with probability at least $1-e^{-\log^2 n}$,
\[e(H_p)\le p^k \cdot 2\kappa.\]
\end{cor}

Intuitively ~\cref{cor:KimVu} states that if $H$ has at most $\kappa$ edges and no small set of vertices controls too many of those edges, then random vertex sampling behaves as though the surviving-edge count were essentially concentrated around $p^k\kappa$.

\subsection{Rooted Booster}

A booster is a graph with the desirable property that it has two disjoint $K_3$-decompositions. Boosters are essential to the second completion step where refined absorbers are converted into high-girth absorbers by replacing triangles of their decomposition families with suitable girth-boosters (in our setting, the spheres from~\cref{def:spheres}). However, the pre-seeding step introduces forbidden configurations that must be avoided in this second completion step. In order for the boosting argument to remain viable, we need to ensure that these new constraints do not eliminate too many potential boosters. This motivates the early introduction of boosters, well before they are required in the absorption step.

\begin{definition}[Latin-Booster]\label{def:booster}
A \emph{Latin-booster} is a graph $B\subseteq K_{n,n,n}$ along with two disjoint $K_3$-decompositions $\cB_1,\cB_2$ of $B$ such that $\cB_1\cap\cB_2=\emptyset$.
\end{definition}

In practice, we use these gadgets to boost one fixed copy $R$ of $K_3$ at a time. We then define a rooted booster as follows.

\begin{definition}[Rooted Latin-Booster]\label{def:rootedbooster}
A \emph{rooted Latin-booster} rooted at $R\cong K_3$, which we denote by $(B,\cBon,\cBoff,R)$, is a graph $B\subseteq K_{n,n,n}$ that is edge-disjoint from $R$ along with two disjoint $K_3$ packings $\cBoff,\cBon$ where $\cBoff$ is a $K_3$-decomposition of $B$ and $\cBon$ is a $K_3$-decomposition of $B\cup R$ with $R\not\in \cBon$.
\end{definition}

It follows from this definition that  $\cBon$ and $\cBoff\cup \{R\}$ are two triangle-decompositions of $B\cup R$. When building a triangle-packing $P$ of a graph $G$, if the packing uses a fixed triangle $R$, we can instead replace $R$ by $\cBon$; alternatively, if $R$ is not in the packing $P$, we simply use $\cBoff$ to decompose the edges of $B$. Therefore, we can replace $R$ with the triangles of the two decompositions $\cBon$ and $\cBoff$. Our goal is then to build boosters whose decomposition families have desirable properties, in particular, they will have high girth. To that end, we define the notion of the rooted girth of a booster as follows.

\begin{definition}[Rooted Girth]
Let $\cB$ be a $K_3$-packing of $G=K_{n,n,n}$. For a vertex set $R\subseteq V(G)$, we define the \emph{rooted girth} of $\cB$ at $R$ as the smallest integer $g\geq 1$ such that there exists a subset $\cB'\subseteq \cB$  with $|\cB'|=g$ and $|V(\bigcup \cB')\setminus R| < g$.  

Similarly, if $B$ is a rooted Latin-booster rooted at $R$, then we define the \emph{rooted girth} of $B$ as the minimum of the girth of $\cBon$, the girth of $\cBoff\cup \{R\}$, and the rooted girth of $\cBon$ at $V(R)$.
\end{definition}

This definition has been introduced by Delcourt and Postle~\cite[Definition 1.15]{DPII}, in the more general setting of hypergraph decompositions into copies of $K_q^r$. We observe that in the present case (with $q=3$ and $r=2$), if $B$ is a rooted Latin-booster rooted at $R$ then the rooted girth of $\cBon$ at $V(R)$ will always be at least the girth of $\cBon$. However, this is only an artifact of the definition when $q=r+1$. For consistency with the existing literature, we decided to keep the definition as presented.

In their remarkable works on high-girth Steiner Triple System~\cite{KSSS2022STS} and on substructures in Latin squares~\cite{KSSS2023substructures}, Kwan, Sah, Sawhney, and Simkin built the following Rooted Latin-boosters, also known as {\em cycles of length $2g$} in design theory.

\begin{definition}[Spheres]\label{def:spheres}
    For $g\geq 2$, and a given copy $R$ of $K_3$, we define the $(R,g)$-sphere to be the following rooted Latin-booster $(B,\cBon,\cBoff,R)$. Denote $V(R)$ by $\{v,b_1,b_{2g}\}$. Add $2g-1$ vertices $u,b_2,b_3,\ldots,b_{2g-1}$, and let $B$ be the graph with vertex-set $V(B)=\{u,v,b_1,\ldots,b_{2g}\}$ adding the following edges,
    \[\set{vb_j\colon j\in\{2,\ldots,2g-1\}},\quad \set{ub_j\colon j\in\{1,\ldots,2g\}},\quad\set{b_jb_{j+1}\colon j\in\{1,\ldots,2g-1\}}.\]    
    Let $\cBon,\cBoff$ be the $K_3$-packing of $B$ defined by
    \begin{align*}
        \cBon  &= \{(vb_1b_2),(ub_2b_3),(vb_3b_4),\ldots,(vb_{2g-1}b_{2g}),(ub_{2g}b_{1})\},\\
        \cBoff &= \{(ub_1b_2),(vb_2b_3),(ub_3b_4),\ldots,(ub_{2g-1}b_{2g})\}.
    \end{align*}
\end{definition}
We refer the reader to~\cref{fig:TriangleBooster} for an illustration of an $(R,3)$-sphere. It is easy to check that if $(B,\cBon,\cBoff,R)$ is an $(R,g)$-sphere, then $B\cup R$ is a graph on $2g+2$ vertices and $6g$ edges, that $\cBoff$ is a $K_3$-decomposition of $B$, $\cBon$ is a $K_3$-decomposition of $B\cup R$ with $R\not\in \cBon$, and that $B\cup R$ is a $K_{1,1,1}$-divisible tripartite graph with vertex-partition $\{\{u,v\},\{b_1,b_3,\ldots,b_{2g-1}\},\{b_2,b_4,\ldots,b_{2g}\}\}$. We will often make use of the following facts.

\begin{remark}\label{rem:GirthSpheres}
    For every integer $g\geq 2$, it follows from~\cite[Lemmas~4.7~and~4.8]{KSSS2022STS} that the rooted-girth of an $(R,g)$-sphere is $2g$.
\end{remark}

\begin{remark}\label{rem:countingSpheres}
    For an integer $g\geq 2$ and a fixed copy $R$ of $K_3$, there are at most $3n^{2g-1}$ copies of the $(R,g)$-sphere in $K_{n,n,n}$. For each such booster $(B,\cBon,\cBoff,R)$, there are at most $2^{|\cBon|+|\cBoff|}=2^{4g-1}$ packings $P$ such that $P\subseteq \cBon\cup\cBoff$.
\end{remark}

\begin{figure}
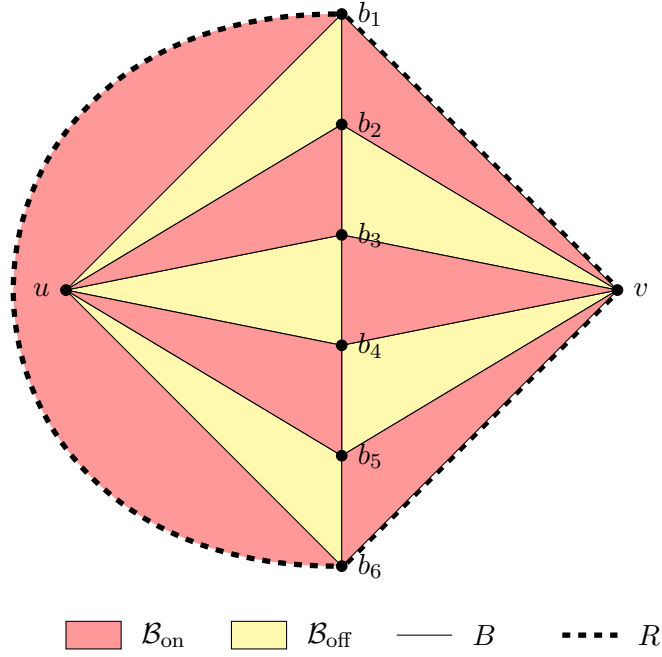

    \centering
    \TriangleBooster{.73}
    \caption{$(R,3)$-sphere rooted at $R=(vb_1b_6)$}
    \label{fig:TriangleBooster}
\end{figure}

\subsection{Treasuries}\label{sec:IntroTreasuries}

Treasuries provide a convenient framework for encoding both the hypergraph we wish to decompose and the configurations that must be avoided during the decomposition process. Roughly speaking, the design and reserve hypergraphs encode the available triangles, while the configuration hypergraph records forbidden low-girth configurations. We now formalize this framework and introduce the notions of regularity and abundance that will play a central role in the remainder of the proof.

\subsubsection{Treasuries And Bankrupted Treasuries.}

We say that a hypergraph $G=(A,B)$ is \emph{bipartite with parts $A$ and $B$} if $V(G)=A\cup B$ and  every edge of $G$ contains exactly one vertex from $A$. We further say that a matching of $G$ is \emph{$A$-perfect} if every vertex of $A$ is in an edge of the matching.

\begin{definition}[Configuration Hypergraph]\label{def:ConfigHypergrapph}
Let $G$ be a hypergraph. We say a hypergraph $H$ is a \emph{configuration hypergraph} for $G$ if $E(G)\subseteq V(H)$ and every edge $e$ of $H$ has size at least two and $e\cap E(G)$ is a matching of $G$. We say a matching of $G$ is \emph{$H$-avoiding} if it spans no edge of $H$.
\end{definition}

We now state the definition of a treasury, a key concept as introduced by Delcourt and Postle~\cite{DPII}. 
\begin{definition}[Treasury]\label{def:Treasury}
Let $r,g \ge 2$ be integers. An \emph{$(r,g)$-treasury} is a tuple $T=(G_1,G_2,H)$ where $G_1$ is an $r$-uniform hypergraph and $G_2=(A,B)$ is an $r$-bounded bipartite hypergraph such that $V(G_1)\cap V(G_2)=A$, and $H$ is a $g$-bounded configuration hypergraph of $G_1\cup G_2$.
A \emph{perfect matching} of $T$ is an $H$-avoiding $(V(G_1)\cap V(G_2))$-perfect matching of $G_1\cup G_2$.
\end{definition}

During our decomposition process of a graph $G$, for a given treasury $T=(G_1,G_2,H)$, the hypergraph $G_1$ often represents the triangles that may be selected to create a packing of $G$; the hypergraph $G_2$, aptly called a reserve hypergraph, encodes additional triangles that might be used to complete a packing of $G$ into a decomposition using some extra edges; while the configuration hypergraph $H$ denotes ``forbidden configurations'' that is families of triangles that are not allowed to be selected together (e.g. for low-girth reasons). To that end, we include the following definition, explaining the concept of a packing being $H$-avoiding and in $T$.

\begin{definition}\label{def:avoiding}
    Let $T=(G_1,G_2,H)$ be a treasury and let $P$ be a packing. We say that $P$ is {$H$-avoiding} if $P$ spans no edge of $H$, that is if for all $e\in E(H)$, $e\not\subseteq P$. We then say that $P$ is in $T$ if additionally we have $P\subseteq E(G_1)$.
\end{definition}

During the pre-seeding step of~\cref{thm:preseeding}, we use {\em bankrupted} treasuries with no reserve hypergraph, that is treasuries $T=(G_1,G_2,H)$ where $G_2$ is the empty graph on vertex set $V(G_1)$. In that case, we slightly abuse notation and write $T=(G_1,H)$ for the treasury.

\subsubsection{Regular Treasuries}\label{sec:RegularTreasuries}

By rephrasing decomposition problems into matchings in auxiliary hypergraphs, one of our main tasks in this article is to prove the existence of (distributions over) perfect matchings in treasuries, via forbidden submatching with reserve methodologies~\cite{DP22}. To that end, the treasuries must exhibit regularity properties, whose exact definitions depend on the following concepts.

\begin{definition}[$i$-codegree]\label{def:icodegree}
Let $G$ be a hypergraph and let $H$ be a configuration hypergraph of $G$. We define the \emph{$i$-codegree} of $e\in V(G)$ and $F\in E(G)\subseteq V(H)$ with $e\notin F$ as the number of edges of $H$ of size $i$ who contain $F$ and a vertex $F_e$ incident with $e$ in $G$. We then define the \emph{maximum $i$-codegree} of $G$ with $H$ as the maximum $i$-codegree over all $e\in V(G)$ and $F\in E(G)\subseteq V(H)$ with $e\notin F$. 
\end{definition}

\begin{definition}[common $2$-degree]\label{def:commontwodegree}
Let $G$ be a hypergraph and let $H$ be a configuration hypergraph of $G$. We define the \emph{common $2$-degree} of distinct vertices $F_1, F_2\in V(H)$ as $|\{F\in V(H): FF_1, FF_2\in E(H)\}|$. Similarly, we define the \emph{maximum common $2$-degree} of $H$ with respect to $G$ as the maximum of the common $2$-degree of $F_1$ and $F_2$ over all distinct pairs of vertex-disjoint edges $F_1,F_2$ of $G$.    
\end{definition}

With these definitions in mind, we are now ready to state the precise regularity conditions we require on treasuries. Roughly speaking, the degree and codegree requirements on the design and reserve hypergraphs ensure that there are sufficiently many available triangles to support a nibble argument, while the bounds on the configuration hypergraph ensure that the forbidden configurations are sparse enough for one to simultaneously find a matching that is independent in the configuration hypergraph. Therefore the regularity conditions are designed to guarantee that a treasury contains many perfect matchings while still allowing us to avoid the forbidden configurations encoded by its configuration hypergraph.

\begin{definition}[Regular Treasury]\label{def:RegularTreasury}
Let $T=(G_1,G_2,H)$ be a treasury with $G_2=(A,B)$, and $D, \sigma\ge 1$ and $\beta,\alpha \in (0,1)$ be reals. We say that the treasury $T$ is \emph{$(D,\sigma,\beta,\alpha)$-regular} if all of the following hold:
\begin{enumerate}[label=(RT\arabic*)]
    \item \textbf{Quasi-regularity.} Every vertex of $G_1$ has degree at most $D$ in $G_1$ and every vertex of $A$ has degree at least $D-\sigma$ in $G_1$,\label{item:RT_QuasiReg}
    
    \item \textbf{Reserve degrees.} Either $\Delta(G_2)=0$ (i.e. we are using a bankrupted treasury), or every vertex of $B$ has degree at most $D$ in $G_2$, and every vertex of $A$ has degree at least $D^{1-\alpha}$ in $G_2$,\label{item:RT_ReserveDeg}

    \item \textbf{Codegrees.} $G_1\cup G_2$ has codegrees at most $D^{1-\beta}$, and the maximum $2$-codegree of $G_1\cup G_2$ with $H$ and the maximum common $2$-degree of $H$ with respect to $G_1\cup G_2$ are both at most $D^{1-\beta}$,\label{item:RT_CoDeg}

    \item \textbf{Configuration (co)degrees} For all $2\le s\le g$ we have $\Delta_1\left(H^{(s)}\right) \le \alpha \cdot D^{s-1}\log D$. For all $2\le t<s\le g$, $\Delta_{t}\left(H^{(s)}\right) \le D^{s-t-\beta}$.\label{item:RT_ConfigDeg}
 
\end{enumerate}
\end{definition}

It might be helpful to the reader to note that if $T$ is $(D,\sigma,\beta,\alpha)$-regular, then $T$ is also $(D,\sigma',\beta',\alpha')$-regular for all $\sigma'\ge \sigma$, $\beta' \le \beta$ and $\alpha'\ge \alpha$. In our following work, we often cannot guarantee that the main treasuries we are studying are regular, but we prove that they contain ``subtreasuries'' that are regular enough, using the following natural definition.

\begin{definition}[Subtreasury]\label{def:SubTreasury}
Let $(G_1,G_2,H)$ be a treasury. We say that a treasury $(G_1',G_2',H')$ is a {\em subtreasury} of $(G_1,G_2,H)$ if $G_i'$ is a spanning subgraph of $G_i$ for each $i\in \{1,2\}$ and $H[E(G_1')\cup E(G_2')]$ is a subgraph of $H'$. 
\end{definition}

Observe that, as expected, this is a transitive relation, that is, if $T_1$ is a subtreasury of $T_2$ and $T_2$ is a subtreasury of $T_3$, then $T_1$ is a subtreasury of $T_3$. Informally, a subtreasury contains fewer edges, and includes more constraints from the configuration hypergraph. The following result from~\cite{DPII} is then a natural consequence of the definition of subtreasuries.

\begin{proposition}\label{prop:SubTreasury}
Let $T'$ be a subtreasury of a treasury $T$. If $M$ is a perfect matching of $T'$, then $M$ is a perfect matching of $T$.    
\end{proposition}

\subsubsection{Abundant Treasuries}\label{sec:AbundantTreasuries}

Recall that abundance was introduced to ensure that the forbidden configurations do not eliminate too many of the boosters required in the absorption step. Informally, a treasury is abundant if, for every root triangle, most suitable $g$-spheres avoid all forbidden configurations. For simplicity, we do not require these $g$-spheres to lie in the design hypergraph $G_1$ in the definition below; when abundance is used, this property will follow from the regularity of the treasury.

\begin{definition}
    Let $g\geq 2$ be an integer, $T=(G_1,G_2,H)$ be a treasury, and let $R$ be a copy of $K_3$ in $K_{n,n,n}$. We say that an $(R,g)$-sphere $(B,\cBon,\cBoff,R)$ is $H$-avoiding, if $\cBon$ and$\cBoff$ are both $H$-avoiding.
\end{definition}

\begin{definition}[Abundant treasuries]\label{def:abundant}
    For integers $g,n\geq 2$, let $T=(G_1,G_2,H)$ be a treasury.
    We say that $T$ is \emph{$g$-abundant} if for every copy $R$ of $K_3$ in $K_{n,n,n}$, at most 
    $n^{2g-3/2}$ copies of $(R,g)$-spheres are not $H$-avoiding.
\end{definition}
We note that the exponent $3/2$ in~\cref{def:abundant} is somehow arbitrary, and used only for convenience; $n^{2g-3/2}$ could be replaced by any amount between $n^{2g-2}\cdot\log^{8g^2}n$ and $n^{2g-1}/\log n$, while preserving the validity of this entire article.

\subsubsection{Bankrupted Design-Treasuries}

In order to introduce our main treasury of interest, we require the following auxiliary hypergraph from~\cite{DPI}, whose matchings encode packings of a given graph.

\begin{definition}[Design Hypergraph]\label{def:DesignHypergraph}
For graphs $F$ and $G$, the \emph{$F$-design hypergraph of $G$}, denoted ${\rm Design}(G,F)$ is the hypergraph $\Gamma$ with $V(\Gamma)=E(G)$ and $E(\Gamma) = \{F'\subseteq E(G): F' \text{ is isomorphic to } F\}$.
\end{definition}

The configuration hypergraph in our treasuries will then be used to avoid packings with low girth.
Recall that an {\em $(i+2,i)$-configuration} is a set of $i$ edge-disjoint triangles spanning at most $i+2$ vertices, and that the girth of a $K_3$-packing $S$ is the smallest $i$ such that $S$ contains an $(i+2,i)$-configuration. For $i\geq 4$, an \emph{\Erdos{} $(i+2,i)$-configuration} is a set of $i$ triangles on $i+2$ vertices with girth exactly $i$, hence containing no $(j+2,j)$-configuration for $4\leq j<i$.

\begin{definition}[Girth Configuration Hypergraph]\label{def:ConfigHypergraph}
Let $g\geq 4$ be an integer. Let $G$ be a graph and let $\mathbf{D}:={\rm Design}(G,K_3)$. The \emph{girth-$g$ $K_3$-configuration hypergraph} of $G$, denoted ${\rm Girth}^g(G,K_3)$ is the configuration hypergraph $H$ of $\mathbf{D}$ with 
\begin{align*}
V(H)&:=E(\mathbf{D}),\\
E(H) &:= \{S\subseteq E(\mathbf{D}): S \text{ is an $(i+2,i)$ \Erdos-configuration of } G \text{ for some $4\le i< g$}\}.   
\end{align*}
\end{definition}

We remark that ${\rm Girth}^g(G,K_3)$ is a $g$-bounded hypergraph. The smallest edges contain $4$ vertices, representing intercalates, i.e. $2\times 2$ Latin squares. We are now ready to define the main treasuries used throughout the pre-seeding step.

\begin{definition}[Bankrupted Design-Treasury]\label{def:BankruptedDesignTreasury}
Let $g\geq 4$ be an integer. Let $G$ be a graph and let $G'\subseteq G$. The \emph{girth-$g$ $K_3$-design treasury of $G$}, denoted ${\rm Treasury}^g(G,G',K_3)$, is the bankrupted treasury 
\[T:=\left({\rm Design}(G',K_3),~{\rm Girth}^g(G,K_3)\right).\]
\end{definition}

\subsubsection{Girth Projection}\label{sec:GirthProjection}

Recall that we reason in two steps to prove our main theorem,~\cref{thm:main}, first finding a high girth packing containing all edges from a set $Y$ in the pre-seeding step, and then finding a high girth packing of the ``left-overs''. However, this is not sufficient to ensure that we obtain a high girth decomposition of $K_{n,n,n}$. We need to ensure that these packings together are still high girth, that is that we do not create a small cycle combining cliques in both packings. To that end, we use the following treasuries from~\cite{DPII}. 

\begin{definition}[Common projection]\label{def:CommonProjection}
Let $T=(G_1,G_2,H)$ be a treasury. Let $\mathcal{M}=\{M_1,\ldots,M_k\}$ be such that for all $i\in [k]$, we have $M_i\subseteq V(H)$ and $M_i\cap (G_1\cup G_2)$ is a matching of $G_1\cup G_2$. The \emph{common projection of $T$ by $\mathcal{M}$}, denoted $T\perp \mathcal{M}$ is the treasury  $T'=(G_1',G_2',H')$ where for all $j\in \{1,2\}$, $V(G_j'):=V(G_j)$, and
\[V(H') := V(H)\setminus \{ F\in V(H): \exists Z\in E(H),~M_i\in \cM \text{ such that }~F\not\in M_i \text{ and } F\in Z \subseteq M_i\cup\{F\}\},\]

and for all $j\in \{1,2\}$, $E(G_j') = E(G_j)\cap V(H')$ and
\[E(H'):= \{ P \subseteq V(H'): |P|\geq 2,~\exists Z\in E(H),~M_i\in \cM \text{ such that } P\cap M_i=\emptyset\text{ and }P\subseteq Z\subseteq P\cup M_i\}.\]
\end{definition}

We now turn to defining the projection of a packing.

\begin{definition}[Girth-$g$ projection of a packing]\label{def:ProjectionPacking}
Let $g\geq 4$ be an integer, $G$ be a graph, and $S$ be a $K_3$-packing of $G$. We define the \emph{girth-$g$ projection treasury} of $S$ onto $G$ as:
\[{\rm Proj}^g(G,S,K_3):= {\rm Treasury}^g(G,~G\setminus E(S),K_3) \perp \{S\}.\]
\end{definition}

\begin{remark}\label{rem:ProjectionSubtreasury}
    For a $K_3$-packing $S$ of a graph $G$, the girth-$g$ projection treasury of $S$ onto $G$ is a subtreasury of ${\rm Treasury}^g(G,~G\setminus E(S),K_3)$. Indeed let $(G_1,H):={\rm Treasury}^g(G,~G\setminus E(S),K_3)$ and $(G'_1,H'):={\rm Proj}^g(G,S,K_3)$. If $P$ is a packing in $H[E(G'_1)]$, with $E(G'_1)=E(G_1)\cap V(H')$ and $ G_1={\rm Design(G\setminus E(S),K_3)}$, we obtain that $P\cap S=\emptyset$. Then, with $H={\rm Girth}^g(G,K_3)$ we also have trivially $P\subseteq P\subseteq P\cup S$, hence $P\in H'$. Observe that this is not true in general for projection of a treasury $T=(G_1,G_2,H)$, unless for every packing $P$ of $H[G_1\cup G_2]$, there exists a matching $M\in\cM$ such that $P\cap M=\emptyset$.
\end{remark}

With these formal definitions, we are now ready to prove~\cref{prop:ProjSubTreasury,prop:completion_treasury}, relating the projection treasury of a packing $S$ of $K_{n,n,n}$, with the treasury of the $K_{n,n,n}\setminus E(S)$. We recall their statement for convenience to the reader.  

\ProjSubTreasury*

\begin{proof}[Proof of~\cref{prop:ProjSubTreasury}]
    Recall~\cref{def:ProjectionPacking,def:BankruptedDesignTreasury},
    \begin{align*}
        {\rm Proj}^g(K_{n,n,n},S,K_3) &= {\rm Treasury}^g(K_{n,n,n},~K_{n,n,n}\setminus E(S),K_3) \perp \{S\}.\\
        &= {\rm Treasury}^g(K_{n,n,n},G,K_3) \perp \{S\}\\
        &=\left({\rm Design}(G,K_3),~{\rm Girth}^g(K_{n,n,n},K_3)\right) \perp \{S\}.\\
        {\rm Treasury}^g(G,G,K_3) 
        &= \left({\rm Design}(G,K_3),~{\rm Girth}^g(G,K_3)\right).\\        
    \end{align*}

    As ${\rm Girth}^g(G,K_3)$ is a subhypergraph of ${\rm Girth}^g(K_{n,n,n},K_3)$, then 
    \[\left({\rm Design}(G,K_3),~{\rm Girth}^g(K_{n,n,n},K_3)\right)\text{ is a subtreasury of }{\rm Treasury}^g(G,G,K_3).\]
    Then by~\cref{rem:ProjectionSubtreasury},
    \[{\rm Proj}^g(K_{n,n,n},S,K_3)\text{ is a subtreasury of }\left({\rm Design}(G,K_3),~{\rm Girth}^g(K_{n,n,n},K_3)\right),\]
    and therefore, by transitivity,
    \[{\rm Proj}^g(K_{n,n,n},S,K_3)\text{ is a subtreasury of }{\rm Treasury}^g(G,G,K_3).\qedhere\]    
\end{proof}

Informally, ${\rm Proj}^g(K_{n,n,n},S,K_3)$ is based on the same family of $K_3$-packings as ${\rm Treasury}^g(G,G,K_3)$, excluding the same low-girth packings, but additionally excluding any packing that might be extended into a low-girth packing using $S$.

\CompletionTreasury*

\begin{proof}[Proof of~\cref{prop:completion_treasury}]
    Let $G'_1,H'$ be hypergraphs such that
    \begin{align*}
        {\rm Proj}^g(K_{n,n,n},S,K_3)&= {\rm Treasury}^g(K_{n,n,n},~K_{n,n,n}\setminus E(S),K_3) \perp \{S\},\\
        &=(G_1,H)\perp \{S\}\\
        &=(G'_1,H'),
    \end{align*}
    where
    \[ G_1:={\rm Design}(K_{n,n,n}\setminus E(S),K_3),\qquad H:={\rm Girth}^g(K_{n,n,n},K_3),\]
    and $G'_1$ is the spanning subhypergraph of $G_1$ with $E(G'_1)=E(G_1)\cap V(H')$.     
    Recall that $M$ is an $H'$-avoiding perfect matching of $G'_1$, hence of $G_1$. Then $S\cup M$ is a $K_3$-decomposition of $K_{n,n,n}$. It remains to prove that this decomposition has girth at least $g$. Towards a contradiction, assume that $S\cup M$ has girth less than $g$. Then there exists a $K_3$-packing $Z$ of $K_{n,n,n}$ such that $Z\in E(H)$ and $Z\subseteq S\cup M$. Because $S$ has girth at least $g$, we have that $Z\not\subseteq S$, and there exist disjoint $K_3$-packings $P_S,P_M$ such that $Z=P_S\cup P_M$ with $P_S\subseteq S$, $P_M\neq\emptyset$ and $P_M\subseteq M$. Observe that $M$ is a matching of $G'_1$ with $E(G'_1)=E(G_1)\cap V(H')$, therefore $P_M\subseteq M$ implies that $P_M\subseteq V(H')$. We obtain that
    \[P_M\subseteq Z\subseteq P_M \cup S,\text{ and }\qquad P_M\cap S = \emptyset.\]
    By~\cref{def:CommonProjection} of common projection, we have that
    \begin{align*}
    V(H') &:= V(H)\setminus \{ F\in V(H): \exists Z\in E(H)\text{ such that }F\not\in S \text{ and } F\in Z \subseteq S\cup\{F\}\},\\
    E(H') &:= \{ P \subseteq V(H'): |P|\geq 2,~\exists Z\in E(H)\text{ such that } P\cap S=\emptyset\text{ and }P\subseteq Z\subseteq P\cup S\}.        
    \end{align*}

    It follows that $\size*{P_M}\geq 2$ and $P_M\in E(H')$, hence $M$ spans an edge of $H'$, a contradiction since $M$ is $H'$-avoiding.
\end{proof}

\subsection{Quantum Packings}

As mentioned in~\cref{sec:PreSeedingIntro}, we appeal to randomness to ensure that the packings we build during the pre-seeding step are high-girth and 'well-spread'. For each edge $e \in E(Y)$, we want to choose a copy of $K_3$ containing $e$, independently at random with some small probability; that is, instead of choosing one triangle for each $e$, we will choose ``many'' of them so that the set of chosen triangles for a given $e$ has size roughly logarithmic in $n$. To that end, we define a more general object than a packing, a \emph{quantum packing}, in line with the definition of quantum booster from~\cite{DPII}. 

\begin{definition}[Quantum packings]\label{def:QuantumPacking}
    Let $Y\subseteq K_{n,n,n}$. A {\em $Y$-enclosed quantum $K_3$-packing} of $K_{n,n,n}$ is a set $\cK$ of copies of $K_3$ such that, for every $F\in \cK$, we have $|E(F)\cap E(Y)|=1$. For a given edge $e\in E(Y)$, we always denote by $\cK_e$ the set of triangles in $\cK$ intersecting $E(Y)$ on $\{e\}$, that is $\cK_e:=\{F\in\cK\colon e\in F\}$.

    We denote by $\cM(\cK)$ the {\em matching set of $\cK$}, defined as the family of all $Y$-fully-enclosed $K_3$-packings of $K_{n,n,n}$ using cliques from $\cK$, that is,
    \[\cM(\cK) := \set*{M\in \prod_{e\in E(Y)}\cK_e\colon M\text{ is a matching of }{\rm Design}(K_{n,n,n},K_3)}.\]   

    For each $e\in E(Y)$, we define ${\rm Disjoint}(\cK,\cK_e)$ to be the cliques containing $e$ that are edge-disjoint from all the other cliques in $\cK\setminus\cK_e$, that is
    \[{\rm Disjoint}(\cK,\cK_e):=\set*{F\in \cK_e\colon \forall F'\in\cK\setminus\cK_{e}, E(F)\cap E(F')=\emptyset}.\]

    For an integer $g\geq 4$ and each $e\in E(Y)$, we define ${\rm HighGirth}^g(\cK,\cK_e)$ to be the cliques containing $e$ that are not part of any low-girth $K_3$-packing $S$ of $K_{n,n,n}$ that itself can be created from cliques in $\cK$, that is
    \[{\rm HighGirth}^g(\cK,\cK_e):=\set*{F\in \cK_e\colon \nexists S\in{\rm Girth}^g(K_{n,n,n},K_3) \textrm{ where } F\in S\subseteq M\text{ for some }M\in\cM(\cK)}.\]    
\end{definition}    

We require an extension of~\cref{def:Spread}, for the spread of a probability distribution over a family of quantum-packings.
\begin{definition}\label{def:spreadpackings}
A probability distribution over a family $\cL$, each element of which is a set of triangles of $K_{n,n,n}$, is {\em $q$-spread} if, for all triangle-packings $S$ of $K_{n,n,n}$ and a random $\cK\in\cL$ (chosen according to the distribution), we have
\[\Prob{S\subseteq \cK}\leq q^{|S|}.\]
\end{definition}

We also extend~\cref{def:ProjectionPacking} of girth-$g$ packing projections to quantum packings as follows. Informally, this is equivalent to projecting all possible choices of packings in a given quantum packing. 

\begin{definition}\label{def:ProjectionQuantumPacking}
    Let $g\geq 4$ be an integer, $Y\subseteq K_{n,n,n}$, and $\cK$ be a $Y$-enclosed quantum $K_3$-packing of $K_{n,n,n}$. The {\em girth-g projection treasury} of $\cK$ is
    \[{\rm Proj}^g(K_{n,n,n},\cK,K_3) := {\rm Treasury}^g(K_{n,n,n},K_{n,n,n},K_3) \perp \cM(\cK),\]
    where $\cM(\cK)$ is the matching set of $\cK$.  
\end{definition}

We now state a quantum version of the pre-seeding theorem,~\cref{thm:preseeding}.
\begin{thm}[Quantum pre-seeding]\label{thm:SpreadQuantumPreseeding}
    For every integer $g\geq 4$ there exists $a \ge 1$ such that, for every reals $\alpha,\beta,\sigma\in(0,1)$, there exist $n_0 > 0$ such that the following holds for all $n\ge n_0$: Let $Y \subseteq K_{n,n,n}$  with $\Delta(Y)\leq \frac{n}{\log^{ag} n}$. Let $\cL_\cK$ be the family of $Y$-enclosed quantum $K_3$-packings $\cK$ such that all of the following properties hold:
    
    \begin{enumerate}[topsep=.1in,itemsep=.1in]
        \item $\Delta(\cK)\leq \frac{4n}{\log^{a(g-1)}n}$. 
        \label{item:QuantPreSeedDelta}
        
        \item For all $e\in E(Y)$, $\size*{{\rm Disjoint}(\cK,\cK_e)\cap{\rm HighGirth}^g(\cK,\cK_e)}\geq (1-\sigma)\log^a n$. \label{item:QuantPreSeedDisjHG}
        
        \item ${\rm Proj}^g(K_{n,n,n},\cK,K_3)$ is $(n,\alpha n,\beta,\alpha)$-regular.\label{item:QuantPreSeedTreasury}

        \item ${\rm Proj}^g(K_{n,n,n},\cK,K_3)$ is $g$-abundant.\label{item:QuantPreSeedAbundant} 
    
    \end{enumerate}
    
    \noindent Then $\cL_\cK$ is non-empty. Furthermore, there exists a probability distribution that is $\frac{\log^a n}{(1-\sigma)n}$-spread over $\cL_\cK$.
\end{thm}

Before proving this theorem in the next section, we show how the pre-seeding theorem,~\cref{thm:preseeding}, follows from this quantum version.

\begin{lateproof}{thm:preseeding}

    For an integer $g\geq 4$ and reals $\alpha,\beta,\varepsilon\in(0,1)$, let $\sigma>0$ be small enough so that $(1+\varepsilon)(1-\sigma)^2>1$. Let $a,n_0$ be such that~\cref{thm:SpreadQuantumPreseeding} holds with parameters $g$, $\alpha/2$, $\beta$, and $\sigma$, and let $n\geq n_0$ be a large enough integer.\medskip
    
    Let $Y \subseteq K_{n,n,n}$ with $\Delta(Y)\leq \frac{n}{\log^{ag} n}$, and let $\cL_P$ be the family of $Y$-fully-enclosed $K_3$-packings $S$ of $K_{n,n,n}$ with girth at least $g$ and with $\Delta(S) \le \frac{4n}{\log^{a(g-1)} n}$, such that there exists a subtreasury of ${\rm Proj}^g(K_{n,n,n},S,K_3)$ that is $(n,\alpha n,\beta,\alpha)$-regular and $g$-abundant. We recall that we want to prove that~\cref{thm:SpreadQuantumPreseeding} implies the existence of a $\frac{1+\varepsilon}{n}$-spread probability distribution over $\cL_P$.\medskip

    Let $p=\frac{\log^a n}{n}$ and let $\cL_\cK$ be the family of $Y$-enclosed quantum $K_3$-packings of $K_{n,n,n}$ as defined in~\cref{thm:SpreadQuantumPreseeding}. By~\cref{thm:SpreadQuantumPreseeding}, there exists a probability distribution $\nu$ that is $\frac{p}{1-\sigma}$-spread over $\cL_\cK$. Let $\cK$ be a quantum packing in $\cL_\cK$ selected at random using the probability distribution $\nu$. \medskip
    
    For all $e\in E(Y)$ select $F_e$ in ${\rm Disjoint}(\cK,\cK_e)\cap{\rm HighGirth}^g(\cK,\cK_e)$ uniformly and independently at random, and let $S:=\{F_e: e\in E(Y)\}$. We show that $S$ is in $\cL_P$.\medskip

    \begin{claim}\label{clm:PackingIsEnclosed}
        $S$ is a $Y$-fully-enclosed packing of $K_{n,n,n}$, with $\Delta(S)\leq \frac{4n}{\log^{a(g-1)}n}$.
    \end{claim}
    \begin{proofclaim}
        For every $e\in E(Y)$, we have $F_e\in {\rm Disjoint}(\cK,\cK_e)\cap{\rm HighGirth}^g(\cK,\cK_e)$. Therefore $S$ is a $K_3$-packing of $K_{n,n,n}$ with girth at least $g$.
        
        For every $F\in S$, there exists $e\in E(Y)$ such that $E(F)\cap E(Y)=\{e\}$. Therefore $S$ is $Y$-enclosed. For every  $e\in E(Y)$, there exists $F\in S$ such that $e\in F$. Therefore $S$ is $Y$-fully-enclosed.

        Finally, for every $F\in S$, there exists $\cK_e\subseteq \cK$ such that $F\in \cK_e$. Therefore
        \[\Delta(S)\leq \Delta(\cK)\leq 4\Delta(Y)\log^an\leq \frac{4n}{\log^{a(g-1)}n}.\qedhere\]
    \end{proofclaim}    
    
    It remains to prove that ${\rm Proj}^g(K_{n,n,n},S,K_3)$ contains a subtreasury that is $(n,\alpha n,\beta,\alpha)$-regular and $g$-abundant. For brevity, denote by $\mathbf{G}$ the girth-configuration hypergraph ${\rm Girth}^g(K_{n,n,n},K_3)$ and by $\mathbf{D}$ the design hypergraph ${\rm Design}(K_{n,n,n},K_3)$. Let $G_1,H,\cG_1,\cH$ be hypergraphs and $T,\cT$ be treasuries such that
    \[T =(G_1,H) := {\rm Proj}^g(K_{n,n,n},S,K_3),\quad 
    \cT=(\cG_1,\cH):={\rm Proj}^g(K_{n,n,n},\cK,K_3).\]
    Recall that by~\cref{thm:SpreadQuantumPreseeding}, $\cT$ is an $(n,\alpha n/2,\beta,\alpha/2)$-regular treasury. By~\cref{def:ProjectionPacking,def:ProjectionQuantumPacking} of the projections of packings and quantum packings, we have
    \begin{align*}
        E(G_1) &= E({\rm Design}(K_{n,n,n}\setminus E(S),K_3))\cap V(H)\\
        V(H)   &= V(\mathbf{G})\setminus
        \{ F\in V(\mathbf{G}): \exists Z\in \mathbf{G} \text{ such that } F\notin S,~F\in Z\subseteq S\cup\{F\}  \}.\\            
        E(H)   &= \{P\subseteq V(H)\colon |P|\geq 2,~\exists Z\in \mathbf{G} \text{ such that }
        P\cap S=\emptyset \text{ and }P\subseteq Z\subseteq P\cup S.\}
    \intertext{and}
        E(\cG_1) &= E(\mathbf{D})\cap V(\cH)\\
        V(\cH)   &= V(\mathbf{G})\setminus
        \{ F\in V(\mathbf{G}): \exists Z\in \mathbf{G},M\in\cM(\cK) \text{ such that } F\notin M,~F\in Z \subseteq M\cup\{F\} \}.\\   
        E(\cH)   &= \{P\subseteq V(\cH)\colon |P|\geq 2,~\exists Z\in \mathbf{G},M\in\cM(\cK) \text{ such that } P\cap M=\emptyset \text{ and }P\subseteq Z\subseteq P\cup M.\}            
    \end{align*}

    By definition we have $S\in\cM(\cK)$ and therefore $V(\cH)\subseteq V(H)$. 
    Let $G'_1=\cG_1\cap {\rm Design}(K_{n,n,n}\setminus E(S))$, and let $T'$ be the treasury \[T':=(G'_1,H'),\text{ with }H'=\cH[E(G'_1)].\] 
    Given that $H'$ is a subgraph of $\cH$, then for any $R$, any $\cH$-avoiding $(R,g)$-sphere is also $H'$-avoiding. Therefore $\cT$ being $g$-abundant immediately implies that $T'$ is $g$-abundant.
    
    \begin{claim}\label{clm:ExistSubtreasury}
    $T'$ is a subtreasury of ${\rm Proj}^g(K_{n,n,n},S,K_3)$.  
    \end{claim}
    \begin{proofclaim}
        We first show that $T'$ is a subtreasury of ${\rm Proj}^g(K_{n,n,n},S,K_3)$. Because $E(\cG_1) = E(\mathbf{D})\cap V(\cH)$, we have 
        \[G'_1={\rm Design}(K_{n,n,n}\setminus E(S),K_3)\cap V(\cH) \]
        and recall that
        \[G_1={\rm Design}(K_{n,n,n}\setminus E(S),K_3)\cap V(H).\]
        With $V(\cH)\subseteq V(H)$, we obtain that $G'_1$ is a spanning subgraph of $G_1$. Finally, using $S\in\cM(\cK)$, we have 
        \[H[E(G'_1)]\subseteq H[V(\cH)]\subseteq \cH,\]
        therefore 
        \[H[E(G'_1)]\subseteq \cH[E(G'_1)] = H',\]
        hence $T'$ is a subtreasury of ${\rm Proj}^g(K_{n,n,n},S,K_3)$.
    \end{proofclaim}
    
    \begin{claim}\label{clm:ExistRegularSubtreasury}
    $T'$ is $(n,\alpha n,\beta,\alpha)$-regular.  
    \end{claim}

    \begin{proofclaim}
        Observe first that as $G'_1$ is a 3-uniform hypergraph, then it has codegrees at most $1\leq n^{1-\beta}$. With $G'_1\subseteq \cG_1$ and $H'$ being an induced subgraph of $\cH$, and as $\cT$ is $(n,\alpha n/2,\beta,\alpha/2)$-regular, then 
        \begin{enumerate}[label=\roman*),topsep=.1in,itemsep=.1in]
            \item the maximum $2$-codegree of $G'_1$ with $H'$ is at most the maximum $2$-codegree of $\cG_1$ with $\cH$, and therefore at most $n^{1-\beta}$,
            \item the maximum common $2$-degree of $H'$ with respect to $G'_1$ is at most the maximum common $2$-degree of $\cH$ with respect to $\cG_1$, and therefore at most $n^{1-\beta}$,       
            \item  for all $2\le i\le g$ we have 
            \[\Delta_1\left({H'}^{(i)}\right) \leq \Delta_1\left(\cH^{(i)}\right) \leq \alpha \cdot n^{i-1}\log n,\]
            and for all $2\le t<s\le g$, we have
            \[\Delta_{t}\left({H'}^{(s)}\right) \leq\Delta_{t}\left(\cH^{(s)}\right) \leq n^{s-t-\beta}.\]
        \end{enumerate}
        Therefore Properties~\cref{item:RT_CoDeg,item:RT_ConfigDeg} of regular treasuries are verified (see~\cref{def:RegularTreasury}), while~\cref{item:RT_ReserveDeg} is trivially true, as we are working with bankrupted treasuries. It remains to show that~\cref{item:RT_QuasiReg} holds. By~\cref{thm:SpreadQuantumPreseeding}, we know that $\cT$ is $(n,\alpha n/2,\beta,\alpha/2)$-regular. Fix a vertex $e\in V(G'_1)\subseteq E(K_{n,n,n})\setminus E(S)$. For $n$ large enough, given that $\Delta(S)$ is at most $\frac{4n}{\log^{a(g-1)}n}$, there exists at most $\frac{4n}{\log^{a(g-1)}n}\ll \alpha n/2$ triangles in $K_{n,n,n}$ containing the edge $e$ and an edge from $S$, hence
        \[d_{\cG_1}(e)-d_{G'_1}(e)\leq\alpha n/2,\]
        and \(\delta(G'_1)\geq n-\alpha n,\)
        as desired.
    \end{proofclaim}

    It follows from~\crefrange{clm:PackingIsEnclosed}{clm:ExistRegularSubtreasury} that $S\in\cL_P$. We claim that this random construction of a packing $S$ in $\cL_P$ induces a $\frac{1+\varepsilon}{n}$-spread distribution $\eta$ over $\cL_P$. To that end, let $P$ be a fixed $K_3$-packing of $K_{n,n,n}$. Trivially, if $P$ is not $Y$-enclosed, then $\mathbb{P}_{\eta}\brackets*{P\subseteq S}=0$. Otherwise, using the fact that the event $\{P\subseteq S\}$ implies the event $\{P\subseteq \cK\}$, we have
    \begin{align*}
    \mathbb{P}_{\eta}\brackets*{P\subseteq S}
    &=\mathbb{P}_{\eta}\brackets*{P\subseteq S\mid P\subseteq \cK}\cdot \mathbb{P}_{\nu}\brackets*{P\subseteq \cK}\\
    &\leq\parens*{\prod_{\substack{F\in P\\ \{e\}:=E(F)\cap E(Y)}}\frac{1}{\size{{\rm Disjoint}(\cK,\cK_e)\cap{\rm HighGirth}^g(\cK,\cK_e)}}}\cdot \mathbb{P}_{\nu}\brackets*{P\subseteq \cK}
    \end{align*}  
    Recall that $\nu$ is a $\frac{p}{1-\sigma}$-spread distribution over $\cL_\cK$, while $\cK\in\cL_\cK$ implies that $$\size{{\rm Disjoint}(\cK,\cK_e)\cap{\rm HighGirth}^g(\cK,\cK_e)}\geq (1-\sigma)\log^a n.$$ Therefore
    \[
    \mathbb{P}_{\eta}\brackets*{P\subseteq S}    
    \leq \parens*{\frac{1}{(1-\sigma)\log^a n}}^{|P|}\cdot \parens*{\frac{p}{1-\sigma}}^{|P|}\\
    = \parens*{\frac{p}{(1-\sigma)^2\log^an}}^{|P|}\\
    \leq \parens*{\frac{1+\varepsilon}{n}}^{|P|},\]
    where we used that $(1+\varepsilon)(1-\sigma)^2>1$ and $p=\frac{\log^a n}{n}$. It follows that $\eta$ is a $\frac{1+\varepsilon}{n}$-spread distribution over~$\cL_P$, as desired.
\end{lateproof}

\subsection{Proof of Quantum Pre-seeding}\label{sec:ProofQuantumPreSeeding}

The following lemma provides an upper bound for the number of low-girth packings $S$ that contain a fixed packing $Z$.

\begin{lem}\label{lem:ExtendPacking}
    For integers $i,k,g$ with $1\leq i <k\leq g$ and $4\leq k$, let $Z$ be a $K_3$-packing of $K_{n,n,n}$ of size $i$, and let $Y\subseteq K_{n,n,n}$. Then for some constant $c_k$ depending only on $k$,
    \[\size*{\set*{S\in{\rm Girth}^g(K_{n,n,n},K_3)\colon |S|=k,\text{ and } Z\subset S}}\leq 
    \begin{cases}
    c_k\cdot n^{k-i-1}&\text{ if }i\geq 2,\\
    c_k\cdot n^{k-1}&\text{ if }i = 1.
    \end{cases}
    \]
    Furthermore, if $i= 1$, then for some constant $c_k$ depending only on $k$,
    \[\size*{\set*{S\in{\rm Girth}^g(K_{n,n,n},K_3)\colon |S|=k,~Z\subset S\text{ s.t. }\exists~F\in S\setminus Z\text{ that is $Y$-enclosed } }}\leq c_k\cdot n^{k-2}\cdot\Delta(Y).\]
\end{lem}

\begin{proof}
    If $Z$ has girth at most $i$, then by minimality, no edge of ${\rm Girth}^g(K_{n,n,n},K_3)$ of size $k>i$ contains $Z$, i.e. the edge $Z$ itself. Assume now that $Z$ has girth greater than $i$, and that $i\geq 2$. Then $Z$ is a set of $i$ triangles in $K_{n,n,n}$ spanning at least $i+3$ vertices. Given that any packing $S\in{\rm Girth}^g(K_{n,n,n},K_3)$ of size $k$ is a set of $k$ triangles spanning at most $k+2$ vertices, there are at most $n^{(k+2)-(i+3)}$ choices for $V(S)$, and therefore at most $c_k\cdot n^{k-i-1}$ choices of packings $S$ containing $Z$, for some constant $c_k$.
    Observe that if $i=1$, then similarly there are at most $n^{(k+2)-3}$ choices for $V(S)$, and therefore at most $c_k\cdot n^{k-i}$ choices of packings $S$ containing $Z$.
    
    For the second inequality, assume that $i=1$ and let $Z=\{F_1\}$. Let $S$ be a packing in ${\rm Girth}^g(K_{n,n,n},K_3)$ of size $k$, such that $S$ contains $F_1$, and such that there exists $F_2\in S\setminus Z$ that is $Y$-enclosed (i.e. containing exactly one edge from $Y$).

    Given that $S$ is a packing with $\{F_1,F_2\}\subseteq S$, we know that $\size*{V(F_1)\cap V(F_2)}\leq 1$. There exist at most $n^2 \Delta(Y)$ choices for $V(F_2)$ such that $\size*{V(F_1)\cap V(F_2)}=0$, and  then at most $n^{(k+2)-6}$ choices for the remaining vertices of $V(S)$. There are therefore at most $c_k\cdot n^{k-2}\Delta(Y)$ choices for such $S$. Similarly, there exist at most $3n \Delta(Y)$ choices for $V(F_2)$ such that $\size*{V(F_1)\cap V(F_2)}=1$, at most $n^{(k+2)-5}$ choices for the remaining vertices of $V(S)$, and at most $c_k\cdot n^{k-2}\Delta(Y)$ choices for such $S$. 
\end{proof}

We are now ready to prove the quantum pre-seeding theorem.
\begin{lateproof}{thm:SpreadQuantumPreseeding}

    Fix an integer $g\geq 4$ and reals $\alpha,\beta,\sigma\in(0,1)$, and let $n$ be large enough throughout this proof. Let $a:=8g$ and $p:=\frac{\log^a n}{n}$. Let $Y \subseteq K_{n,n,n}$ with $\Delta(Y)\leq \frac{n}{\log^{ag} n}$, and let $\cL_\cK$ be the family of $Y$-enclosed quantum $K_3$-packings as defined in the statement of~\cref{thm:SpreadQuantumPreseeding}. For brevity, we sometimes denote by $\mathbf{G}$ the girth-configuration hypergraph ${\rm Girth}^g(K_{n,n,n},K_3)$ and by $\mathbf{D}$ the design hypergraph ${\rm Design}(K_{n,n,n},K_3)$. For each $e\in E(Y)$, let $\Omega_e$ be the set of triangles in $K_{n,n,n}$ whose intersection with $Y$ is exactly $e$,
    \[\Omega_e=\set*{F\in K_{n,n,n}\colon F\cong K_3, E(Y)\cap E(F)=\{e\}},\]
    and let $\Omega = \bigcup_{e\in E(Y)} \Omega_e$ be a $Y$-enclosed quantum packing of $K_{n,n,n}$. Observe that for any distinct edges $e,e'\in E(Y)$, $\Omega_{e}$ and $\Omega_{e'}$ are disjoint, and  
    \[n-2\Delta(Y)\leq |\Omega_{e}|\leq n.\]

    Let $\cV_p$ be a random subset of $V(\mathbf{G})=\set{F\in K_{n,n,n}\colon F\cong K_3}$, where each copy of $K_3$ is selected independently at random with probability $p$. Let $\cK=\Omega\cap\cV_p$ be a random $Y$-enclosed quantum $K_3$-packing of $K_{n,n,n}$. Let $\cG_1,\cH$ be hypergraphs such that $\cT=(\cG_1,\cH):={\rm Proj}^g(K_{n,n,n},\cK,K_3)$ is the girth-$g$ projection of the quantum packing $\cK$. We argue that each one of the following properties holds with probability at least $1-\sigma/10$.

    \begin{enumerate}[label=(P\arabic*),topsep=.1in,itemsep=.1in]
        \item $\Delta(\cK)\leq 4n/\log^{a(g-1)}n$.\label{prop:FIRSTQuantum}\label{prop:QuantumMaxDeg}
        \item For each $e\in E(Y)$, $|\cK_e|\geq (1-\frac{\sigma}3
        )\log^an$.\label{prop:QuantumMinSize}
        \item For all $e\in E(Y)$, $\size*{\cK_e\setminus {\rm HighGirth}^g(\cK,\cK_e)}\leq\frac{\sigma}{3}\log^a n$. \label{prop:QuantumHighGith}
        \item For all $e\in E(Y)$, $\size*{\cK_e\setminus {\rm Disjoint}(\cK,\cK_e)}\leq\frac{\sigma}{3}\log^a n$. \label{prop:QuantumDisjoint}      
        \item Every vertex of $\cG_1$ has degree at least $(1-\alpha)n$.\label{prop:QuantumRegDegree}
        \item For all $i\in\{2,\ldots,g\}$ we have $\Delta_{1}(\cH^{(i)})\leq \alpha n^{i-1}\log n$.\label{prop:QuantumRegConfigDegree}
        \item For all $t<s$ in $\{2,\ldots,g\}$ we have $\Delta_{t}(\cH^{(s)})\leq n^{s-t-\beta}$.  \label{prop:QuantumRegConfigCoDegree}
        \item The maximum common $2$-degree of $\cH$ with respect to $\cG_1$ is at most $n^{1-\beta}$.\label{prop:QuantumRegMaxCommon2Deg}
        \item The maximum $2$-codegree of $\cG_1$ with $\cH$ is at most $n^{1-\beta}$.\label{prop:QuantumTwoCodegrees}
        \item The treasury $\cT$ is $g$-abundant.\label{prop:QuantumAbundant}\label{prop:LASTQuantum}
    \end{enumerate}

    Observe that~\cref{prop:QuantumMaxDeg} ensures that~\cref{thm:SpreadQuantumPreseeding}\ref{item:QuantPreSeedDelta} holds, while~\cref{prop:QuantumMinSize}-\cref{prop:QuantumDisjoint} ensure that~\cref{thm:SpreadQuantumPreseeding}\ref{item:QuantPreSeedDisjHG} holds. Note that $\cG_1$ has trivially codegrees at most $1\leq D^{1-\beta}$, therefore~\cref{prop:QuantumRegDegree}-\cref{prop:QuantumTwoCodegrees} are sufficient to ensure that the girth-$g$ projection of $\cK$ is regular, hence that~\cref{thm:SpreadQuantumPreseeding}\ref{item:QuantPreSeedTreasury} holds.
    Finally~\cref{prop:QuantumAbundant} ensures that~\cref{item:QuantPreSeedAbundant} holds.
    
    We first remark that \cref{prop:QuantumMinSize,prop:QuantumMaxDeg} hold with probability at least $1-\sigma/10$, by standard applications of Chernoff bound. Each remaining property is shown in one claim below, all proved in a similar way. Each statement is of the form ''the probability that for all $x\in X$ we have $N(x)\leq b$ is large'', where $N$ enumerates objects with some desired property. We reason by union bound over all elements $x\in X$ (e.g. the family of all edges in $E(Y)$), showing that each event $\{N(x)>b\}$ happens with low probability. To do so, for each $x\in X$, we define a multi-hypergraph $J$ on vertex set $V(J)\subseteq V(\mathbf{G})$ and such that $N(x)\leq e(J_p)$, where $J_p$ is the subhypergraph of $J$ induced on the random set $\cV_p$. We then use Kim-Vu~\cref{cor:KimVu} to estimate a probabilistic bound on $e(J_p)$, summing over all uniform subhypergraphs of $J$ if necessary.      

    
    \begin{claim}\label{cl:HighGirthHolds}
        \cref{prop:QuantumHighGith}, ``For all $e\in E(Y)$, $\size*{\cK_e\setminus {\rm HighGirth}^g(\cK,\cK_e)}\leq\frac{\sigma}{3}\log^a n$'',  holds with probability at least $1-\sigma/10$.
    \end{claim}
    \begin{proofclaim}
        Recall that we denote by $\mathbf{G}$ the girth-configuration hypergraph ${\rm Girth}^g(K_{n,n,n},K_3)$ and by $\mathbf{D}$ the design hypergraph ${\rm Design}(K_{n,n,n},K_3)$. 
        We reason via union bound on $E(Y)$. For each $e\in E(Y)$, we build a hypergraph $J$ such $|\cK_e\setminus{\rm HighGirth}^g(\cK,\cK_e)|\leq e(J_p)$, and we show that $e(J_p)>\frac{\sigma}{3}\log^a n$ with small probability using Kim-Vu's~\cref{cor:KimVu}.
    
        Fix $e\in E(Y)$ and let $J=J(e)\subseteq \mathbf{G}$ be the multi-hypergraph such that $V(J)=V(\mathbf{G})$ and
        \[E(J)=\bigcup_{F\in\Omega_e} \set*{S\in \mathbf{G}\colon F\in S\subseteq M\text{ for some }M\in\cM(\Omega)}.\]
        
        The hypergraph $J$ contains all $Y$-enclosed $K_3$-packing of $K_{n,n,n}$ with girth less than $g$ and containing (exactly) one copy $F$ of of $K_3$  with $e\in F$. Let $J_p$ be the random subhypergraph of $J$, induced on the vertex set $\cV_p$. Observe that for any edge $S\in E(J)$, there exists at most one $F\in\Omega_e$ such that $F\in S$. Furthermore, if $F$ is a copy of $K_3$ in $\cK_e\setminus{\rm HighGirth(\cK_e)}$, then there exists $S\in\mathbf{G}$, with $F\in S\subseteq M$ for some $M\in\cM(\cK)$. It follows that every triangle in $S$ is in $\cK$, hence $S$ is an edge of $J_p=J[\cV_p]$, and
        \begin{equation*}
          \size*{\cK_e\setminus{\rm HighGirth(\cK_e)}}\leq e(J_p)=\sum_{k=4}^{g-1}e(J_p^{(k)}).  
        \end{equation*}

        Let $k\in\set{4,\ldots g-1}$ and let $$\kappa_k :=\frac{\sigma}{6g}\frac{n^k}{\log^{a(k-1)}n}.$$
        
        Recall that any matching $M\in\cM(\Omega)$ is $Y$-enclosed. Therefore~\cref{lem:ExtendPacking} implies that 
        \begin{align*}
            e(J^{(k)}) &= \sum_{F\in\Omega_e} \size*{\set*{S\in \mathbf{G}\colon F\in S\subseteq M\text{ for some }M\in\cM(\Omega)}}\\
            &= |\Omega_e|\cdot O\parens*{n^{k-2}\Delta(Y)}\\
            &= O\parens*{n^{k}/\log^{ga}n}\\
            &\leq \kappa_k.
        \end{align*}
        
        Fix $Z\subseteq V(J)$ with $|Z|=i\in\{1,\ldots,k\}$. If $Z$ is not a packing then $d_J(Z)=0$, otherwise observe that there exist a unique $F_e\in Z\cap \Omega_e$, and therefore, if $i=k$ we have $d_J(Z)=1$. Assume now that $i<k$. It follows that
        \begin{align*}
          \Delta_i(J^{(k)}) &= 
          \size*{\set*{S\in \mathbf{G}\colon \exists M\in\cM(\Omega), Z\subseteq S\subseteq M}}
          +\sum_{\substack{F\in\Omega_e\\F\neq F_e}}\size*{\set*{S\in \mathbf{G}\colon \exists M\in\cM(\Omega), (Z\cup\{F\})\subseteq S\subseteq M}},  
        \end{align*}
        hence, by~\cref{lem:ExtendPacking}, if $i=1$ 
        \[  \Delta_1(J^{(k)}) = O\parens*{\frac{n^{k-1}}{\log^{ga}n}} + n\cdot O\parens*{n^{k-3}} = O\parens*{\frac{n^{k-1}}{\log^{ga}n}},\]
        and if $i\in\{2,\ldots,k-1\}$, 
        \[  \Delta_i(J^{(k)}) = O\parens*{n^{k-i-1}} + n\cdot O\parens*{n^{k-i-2}} = O\parens*{n^{k-i-1}}.\]
        
        We apply~\cref{cor:KimVu} to the $k$-uniform hypergraph $J^{(k)}$. Observe that,
        \[N:=v(J^{(k)})=v(\mathbf{G})=n^3.\]    
    
        For $i\in[k-1]$, we check that $(*_i)$ holds. 
        \[
        \frac{\Delta_i(J^{(k)})}{\kappa_k}\log^{4k+2}N
        = O\parens*{\frac{n^{k-i}}{\log^{ga}n}}\cdot\frac{6g}{\sigma}\cdot\frac{\log^{a(k-1)}n}{n^k}\log^{4k+2} n
        \leq \frac{\log^{ia}n}{n^i}
        =p^i,
        \]
        using $a=8g>4k+2$. We now verify that $(*_k)$ also holds
        \[
        \frac{\Delta_k(J^{(k)})}{\kappa_k}\log^{4k+2}N
        \leq  \frac{6g}{\sigma}\frac{\log^{a(k-1)}n}{n^k}\log^{4k+2}n
        \leq \frac{\log^{ak}n}{n^k}
        =p^k,
        \]
        again using $a=8g>4k+2$. It follows from~\cref{cor:KimVu} that we have,
        \[e(J^{(k)}_p)>2p^k\kappa_k = \frac{\sigma\log^a n}{3g}\]
        with probability at most $e^{-\log^2 N}$. By union bound it follows that
        \begin{align*}
            \Prob{\exists e\in E(Y), \size*{\cK_e\setminus{\rm HighGirth(\cK_e)}}>\frac{\sigma}{3})\log^a n} 
            &\leq e(Y)\cdot \max_{e\in E(Y)}\Prob{\size*{\cK_e\setminus{\rm HighGirth(\cK_e)}}>\frac{\sigma}{3}\log^a n}\\
            &\leq e(Y)\cdot \max_{e\in E(Y)}\Prob{e(J_p)>\frac{\sigma}{3}\log^a n}\\
            &\leq e(Y)\cdot g\cdot \Prob{e(J_p^{(k)})>\frac{\sigma}{3g}\log^a n}\\
            &\leq e(Y)\cdot g\cdot e^{-\log^2 N}\\
            &<g\cdot n^2\cdot e^{-\log^2n}\\
            &<\sigma/10
        \end{align*}
        given that $N>n$ is large enough and that $e(Y)\leq n^2$.   
    \end{proofclaim}

    \begin{claim}\label{cl:DisjointHolds}
        \cref{prop:QuantumDisjoint}, ``For all $e\in E(Y)$, $\size*{\cK_e\setminus {\rm Disjoint}(\cK,\cK_e)}\leq\frac{\sigma}{3}\log^a n$'', holds with probability at least $1-\sigma/10$.    
    \end{claim}
    \begin{proofclaim}

    Fix $e\in E(Y)$. Let $J=J(e)$ be the graph defined by $V(J)=V(\mathbf{G})$ and 
    \[E(J)=\set*{\set{F,F'}\in \Omega_e \times (\Omega\setminus\Omega_e)\colon F\cap F'\neq\emptyset}.\]
    
    Let $J_p$ be the random subhypergraph of $J$, induced on the vertex set $\cV_p$. As desired, we have that 
    \[|\cK_e\setminus{\rm Disjoint}(\cK,\cK_e)|\leq e(J_p).\]

    Recall that any copy of $K_3$ in $\Omega$ intersect $E(Y)$ in exactly one edge, therefore for any $F\in\Omega_e$, we have 
    \[d_J(F)\leq \size*{\set*{F'\in(\Omega\setminus\Omega_e)\colon  F\cap F'\neq\emptyset}}\leq 3\Delta(Y),\]
    hence $\Delta_1(J) \leq 3\Delta(Y)$, while $\Delta_2(J) = 1$. Then, as every edge of $J$ is incident with some $F\in\Omega_e$,
    \[e(J)\leq \sum_{F\in\Omega_e}d_J(F)\leq |\Omega_e|\cdot 3\Delta(Y) \leq\frac{3n^2}{\log^{ga}n}.\]
    Let \[\kappa := \frac{\sigma n^2}{6\log^{a}n},\] hence satisfying $e(J)\leq \kappa$. We check that $(*)_i$ hold for $i\in\{1,2\}$. With $a>10$, we have
    
    \[\frac{\Delta_1(J)}{\kappa}\log^{10}n \leq \frac{3n}{\log^{ga}n}\frac{6\log^{a}n}{\sigma n^2}\log^{10}n = \frac{18}{\sigma\log^{ga-10}n}\frac{\log^an}{n}\leq p,\]
    and
    \[\frac{\Delta_2(J)}{\kappa}\log^{10}n \leq \frac{6\log^{a}n}{\sigma n^2}\log^{10}n \leq \frac{6}{\sigma\log^{a-10}n}\frac{\log^{2a}n}{n^2}\leq p^2.\]
    
    It follows from~\cref{cor:KimVu} that we have,
    \[e(J_p)>2p^2\kappa = \frac{\sigma}{3}\log^a n\]
    with probability at most $e^{-\log^2 N}$, where $N:=v(J)= n^3$. By union bound it follows that
        \begin{align*}
            \Prob{\exists e\in E(Y), \size*{\cK_e\setminus{\rm Disjoint(\cK_e)}}>\frac{\sigma}{3}\log^a n} 
            &\leq e(Y)\cdot \max_{e\in E(Y)}\Prob{\size*{\cK_e\setminus{\rm Disjoint(\cK_e)}}>\frac{\sigma}{3}\log^a n}\\
            &\leq n^2\cdot \max_{e\in E(Y)}\Prob{e(J_p)>\frac{\sigma}{3}\log^a n}\\
            &\leq n^2\cdot e^{-\log^2 N}\\
            &<\sigma/10
        \end{align*}
        given that $N>n$ is large enough.       
    \end{proofclaim}

    \begin{claim}\label{cl:RegularityDegreeHolds}
        \cref{prop:QuantumRegDegree}, ``Every vertex of $\cG_1$ has degree at least $(1-\alpha)n$'', holds with probability at least $1-\sigma/10$.    
    \end{claim}
    \begin{proofclaim} 
        Recall that we denote by $\mathbf{G}$ the girth-configuration hypergraph ${\rm Girth}^g(K_{n,n,n},K_3)$ and by $\mathbf{D}$ the design hypergraph ${\rm Design}(K_{n,n,n},K_3)$, and that $(\cG_1,\cH)$ is the projection treasury of $\cK$, that is $(\cG_1,\cH)={\rm Proj}^g(K_{n,n,n},\cK,K_3)$.
        
        By~\cref{def:ProjectionQuantumPacking} of projection, we have $E(\cG_1)=E(\mathbf{D})\cap V(\cH)$, and with $\cH\subseteq V(\mathbf{G})=E(\mathbf{D})$ we obtain that 
        \[E(\cG_1)=V(\cH)=V(\mathbf{G})\setminus\set{F\in V(\mathbf{G})\colon \exists S\in\mathbf{G}, \exists M\in\cM(\cK), F\notin M,
        F\in S\subseteq M\cup\{F\}}.\]
        
        Fix $e\in V(\cG_1)$. We define $J=J(e)$ to be the multi-hypergraph with $V(J)=V(\mathbf{G})$ and 
        \[E(J)=\bigcup_{\substack{F\in\mathbf{D}\\e\in F}} \set*{S\setminus\{F\}\colon S\in \mathbf{G}\text{ s.t } \exists M\in\cM(\Omega), F\notin M, F\in S\subseteq M\cup\{F\}}.\]
        
        Observe that any edge in $J$ has size in $\{3,\ldots,g-1\}$. Let $J_p$ be the random subhypergraph of $J$, induced on the vertex set $\cV_p$. Assume that there exists an edge $F\in\mathbf{D}$ with $e\in F$, such that $F\notin\cG_1$. Then there must exist $S\in\mathbf{G}$ such that $F\in S\subseteq M\cup\{F\}$ for some $M\in\cM(\cK)$. It follows that $S\setminus\{F\}$ is an edge of $J_p$, hence
        \[d_{\mathbf{D}}(e)-d_{\cG_1}(e)< e(J_p).\]

        Fix $k\in\{3,\ldots,g-1\}$, and let \[\kappa_k:=\frac{\alpha}{2g}\cdot\frac{n^{k+1}}{\log^{ak}n}.\]
        
        Let $Z\subseteq V(J)$ with $|Z|=i\in[k]$. Recall that for every $M\in\cM(\cK)$, $M$ is $Y$-enclosed, and observe that if $Z$ is not a packing, or if $Z$ contains a copy $F_e$ of $K_3$ with $e\in F_e$, then $d_J(Z)=0$. We can then assume that $Z$ is a packing not containing the edge $e$. For every edge of $J^{(k)}$ that contains $Z$, there exist $S\in\mathbf{G}^{(k+1)}$ and $F\in\mathbf{D}$ such that $e\in F$, $Z\subseteq S$, and $S\setminus \{F\}$ is $Y$-enclosed. It follows from~\cref{lem:ExtendPacking} that if $i<k$ we have
        \[d_{J^{(k)}}(Z) \leq \sum_{\substack{F\in\mathbf{D}\\e\in F}}  \size*{\set*{S\in\mathbf{G}^{(k+1)}\colon Z\cup\{F\}\subseteq S}}= n\cdot O\parens*{ n^{(k+1)-(i+1)-1}}= O\parens*{n^{k-i}}.\]
        If $|Z|=i=k$, observe that if $Z\in\mathbf{G}$, then $Z\cup \{F\}$ is not in $\mathbf{G}^{(k+1)}$ for any $F$, and again $d_{J^{(k)}}(Z)=0$. We can then assume that $Z$ is a set of $k$ triangles spanning at least $k+3$ vertices. If $Z\cup \{F\}=S\in\mathbf{G}^{(k+1)}$ for some $F$, it is then a set of $k+1$ triangles spanning at most $k+3$ vertices. there exists only a constant number of possible such triangles $F$, therefore $d_{J^{(k)}}(Z)=O(1)=O(n^{k-i})$. We furthermore have
        \begin{align*}
            e(J^{(k)})&=\sum_{\substack{F\in\mathbf{D}\\e\in F}} \size*{\set*{S\setminus\{F\}\colon S\in \mathbf{G}^{(k+1)}\text{ s.t } \exists M\in\cM(\Omega), F\notin M, F\in S\subseteq M\cup\{F\}}}\\
            &\leq\sum_{\substack{F\in\mathbf{D}\\e\in F}} \size*{\set*{S\setminus\{F\}\colon S\in \mathbf{G}^{(k+1)}\text{ s.t } F\in S, \text{ and } S\setminus \{F\}\text{ is $Y$-enclosed}}}\\    
            &\leq n\cdot O\parens*{ n^{(k+1)-2}\cdot\Delta(Y)}\\
            &= O\parens*{\frac{n^{k+1}}{\log^{ag}n}}\\
            &\leq \kappa_k.
        \end{align*}        
                
        We apply~\cref{cor:KimVu} to the $k$-uniform hypergraph $J^{(k)}$. Observe that,
        \[N:=v(J^{(k)})= v(\mathbf{G})= n^3.\]    
    
        For $i\in[k]$, we check that $(*_i)$ holds. 
        \[
        \frac{\Delta_i(J^{(k)})}{\kappa_k}\log^{4k+2}N
        = O\parens*{n^{k-i}}\cdot \frac{2g}{\alpha}\frac{\log^{ak}n}{n^{k+1}}\cdot\log^{4k+2}n 
        \leq \frac{\log^{ia}n}{n^i}
        =p^i.
        \]

        It follows from~\cref{cor:KimVu} that we have,
        \[e(J^{(k)}_p)>2p^k\kappa_k = \frac{\alpha}{g}\cdot\frac{\log^{ak} n}{n^k}\cdot\frac{n^{k+1}}{\log^{ak} n} =  \frac{\alpha n}{g}\]
        with probability at most $e^{-\log^2 N}$. For every edge $e\in E(Y)$, we have that $d_{\mathbf{D}}(e)\leq n$, therefore by the union bound,
        \begin{align*}
            \Prob{\exists e\in E(K_{n,n,n}), d_{\cG_1}(e)<(1-\alpha) n}
            &=\Prob{\exists e\in E(K_{n,n,n}), d_{\mathbf{D}}(e)-d_{\cG_1}(e)> \alpha n}\\
            &\leq \Prob{\exists e\in E(K_{n,n,n}), e(J_p)> \alpha n}\\
            &\leq e(K_{n,n,n})\cdot \max_{e\in E(K_{n,n,n})} \Prob{\sum_{k=3}^{g-1}e(J_p^{(k)})> \alpha n}\\
            &\leq 3n^2\cdot g\cdot \max_{e\in E(K_{n,n,n})}\Prob{e(J_p^{(k)})> \alpha n/g}\\
            &= 3n^2\cdot g\cdot \max_{e\in E(K_{n,n,n})}\Prob{e(J_p^{(k)})> 2p^k \kappa_k}\\
            &\leq 3n^2\cdot g\cdot e^{-\log^2 N}\\
            &<\sigma/10
        \end{align*}
        given that $N>n$ is large enough.
    \end{proofclaim}

    \begin{claim}\label{cl:RegularityConfigDegreeHolds}
        \cref{prop:QuantumRegConfigDegree}, ``For all $s\in\{2,\ldots,g\}$ we have $\Delta_{1}(\cH^{(s)})\leq \alpha n^{s-1}\log n$.'', holds with probability at least $1-\sigma/10$.    
    \end{claim}
    \begin{proofclaim} 
    
    Recall that we denote by $\mathbf{G}$ the girth-configuration hypergraph ${\rm Girth}^g(K_{n,n,n},K_3)$ and by $\mathbf{D}$ the design hypergraph ${\rm Design}(K_{n,n,n},K_3)$, and that $(\cG_1,\cH)$ is the projection treasury of $\cK$, that is $(\cG_1,\cH)={\rm Proj}^g(K_{n,n,n},\cK,K_3)$, and that
    \[E(\cH)=\set*{P\subseteq V(\cH)\colon |P|\geq 2,~\exists S\in\mathbf{G},\exists M\in\cM(\cK) \text{ s.t. } P\subseteq S\subseteq M\cup P\text{ and } P\cap M=\emptyset}.\]
    
    Fix an integer $s\in\{2,\ldots,g\}$ and a copy $F$ of $K_3$. Let $J=J(F)$ be the multi-hypergraph defined by $V(J)=V(\mathbf{G})$ and 
    \[E(J) = \bigcup_{\substack{P\subseteq V(\mathbf{G})\\F\in P,\ |P|=s}}\set*{S\setminus P\colon S\in\mathbf{G},\text{ s.t. }\exists M\in\cM(\Omega), P\subsetneq S\subseteq M\cup P,\text{ and } P\cap M=\emptyset}.\]    

    Observe that every edge in $J$ has size in $\{1,\ldots, g-s\}$, and importantly that, in the definition of $J$, we ask for $P\subsetneq S$. Let $J_p$ be the random subhypergraph of $J$, induced on the vertex set $\cV_p$. Assume that there exists an edge $P\in\cH^{(s)}$ with $F\in P$, such that $P\notin\mathbf{G}^{(s)}$. Then there exists $S\in\mathbf{G}$ such that $P\subsetneq S\subseteq M\cup P$ for some $M\in\cM(\cK)$. It follows that $S\setminus P$ is an edge of $J_p$, and
    \[d_{\cH^{(s)}}(F) - d_{\mathbf{G}^{(s)}}(F)\leq \sum_{k=1}^{g-s}e(J^{(k)}_p).\]

    Fix $k\in[g-s]$ and let $$\kappa_k:=\frac{\alpha}{4g}\cdot{\frac{n^{k+s-1}}{\log^{ak-1}n}}.$$ Let $Z\subseteq V(\mathbf{G})$ with $|Z|=i\in[k]$. Note that if $Z$ is not a packing, or if $F\in Z$, then $d_J(Z)$=0. Otherwise, for every edge of $J^{(k)}$ that contains $Z$, there exist $S\in\mathbf{G}^{(k+s)}$ such that $Z\cup \{F\} \subseteq S$. Conversely, every edge $S\in\mathbf{G}^{(k+s)}$ induces at most $\binom{k+s}{s}$ edges in $J^{(k)}$, and it follows from~\cref{lem:ExtendPacking} that 
    \begin{align*}
        d_{J^{(k)}}(Z) \leq \binom{k+s}{s}\size*{\set*{S\in\mathbf{G}^{(k+s)}\colon Z\cup \{F\}\subseteq S}} = O\parens*{n^{(k+s)-(i+1)-1}} ,
    \end{align*}
    and
    \[\Delta_i(J^{(k)})=O\parens*{n^{k+s-i-2}}.\]
    
    Similarly we obtain that
    \begin{align*}
        e(J^{(k)}) \leq \binom{k+s}{s}\size*{\set*{S\in\mathbf{G}^{(k+s)}\colon \{F\}\subseteq S}} = O\parens*{\frac{n^{k+s-1}}{\log^{ag} n}}\leq \kappa_k,
    \end{align*}        
    using $k\leq g-s\leq g-2$.
    We apply~\cref{cor:KimVu} to the $k$-uniform multi-hypergraph $J^{(k)}$. Observe that,
    \[N:=v(J^{(k)})= v(\mathbf{G})= n^3.\]    

    For $i\in[k]$, we check that $(*_i)$ holds. 
    \[
    \frac{\Delta_i(J^{(k)})}{\kappa_k}\log^{4k+2}N
    = O\parens*{n^{k+s-i-2}}\cdot \frac{4g}{\alpha}\cdot\frac{\log^{ak-1}n}{n^{k+s-1}}\cdot\log^{4k+2}n = O\parens*{\frac{\log^{a(k+1)} n}{n^{i+1}}}
    \leq \frac{\log^{ai}n}{n^i}
    =p^i.
    \]

    It follows from~\cref{cor:KimVu} that we have,
    \[e(J^{(k)}_p)>2p^k\kappa_k = \frac{\alpha}{2g}\cdot\frac{\log^{ak} n}{n^k}\cdot\frac{n^{k+s-1}}{\log^{ak-1} n} =   \frac{\alpha}{2g}n^{s-1}\log n,\]
    with probability at most $e^{-\log^2 N}$. We obtain,
        \begin{align*}
            \Prob{\exists s\in \{2,\ldots,g\}, \Delta_{1}(\cH^{(s)})>\alpha n^{s-1}\log n}
            &\leq g\cdot \max_{s\in \{2,\ldots,g\}}\Prob{\Delta_{1}(\cH^{(s)})>\alpha n^{s-1}\log n}\\
            &\leq g\cdot \max_{s\in \{2,\ldots,g\}}\Prob{\exists F\in\mathbf{D}\text{ s.t. } d_{\cH^{(s)}}(F)>\alpha n^{s-1}\log n}.
        \end{align*}  
        Remark that, by~\cref{lem:ExtendPacking}, for every $F\in\mathbf{D}$, we have $d_{\mathbf{G}^{(s)}}(F)=O\parens*{n^{s-1}}$, and therefore
        \begin{align*}
        \Prob{\exists s\in \{2,\ldots,g\}, \Delta_{1}(\cH^{(s)})>\alpha n^{s-1}\log n}
            &\leq g\cdot \Prob{\exists F\in\mathbf{D}\text{ s.t. } d_{\cH^{(s)}}(F)-d_{\mathbf{G}^{(s)}}(F)>\frac{\alpha}{2} n^{s-1}\log n}\\
            &\leq g\cdot n^3\cdot \max_{s,F}\Prob{d_{\cH^{(s)}}(F)-d_{\mathbf{G}^{(s)}}(F)>\frac{\alpha}{2} n^{s-1}\log n}\\
            &\leq g\cdot n^3 \cdot \max_{s,F}\Prob{\sum_{k=1}^{g-s} e(J_p^{(k)})> \frac{\alpha}{2} n^{s-1}\log n}\\
            &\leq g^2\cdot n^3 \cdot \max_{s,F}\Prob{e(J_p^{(k)})> \frac{\alpha}{2g} n^{s-1}\log n}\\     
            &=g^2\cdot n^3 \cdot \max_{s,F}\Prob{e(J_p^{(k)})> 2\kappa_k p^k}\\            
            &\leq g^2\cdot n^3 \cdot e^{-\log^2 N}\\
            &<\sigma/10
        \end{align*}
        given that $N>n$ is large enough.   
    \end{proofclaim}


    \begin{claim}\label{cl:RegularityConfigCoDegreeHolds}
        \cref{prop:QuantumRegConfigCoDegree}, ``For all $t<s$ in $\{2,\ldots,g\}$ we have $\Delta_{t}(\cH^{(s)})\leq n^{s-t-\beta}$'', holds with probability at least $1-\sigma/10$.    
    \end{claim}
    \begin{proofclaim} 

    Fix integers $2\leq t<s\leq g$ and a packing $Q\subseteq V(\mathbf{G})$ of size $t$. Let $J=J(Q)$ be the multi-hypergraph defined by $V(J)=V(\mathbf{G})$ and 
    \[E(J) = \bigcup_{\substack{P\subseteq V(\mathbf{G})\\Q\subset P,\ |P|=s}}\set*{S\setminus P\colon S\in\mathbf{G},\text{ s.t. }\exists M\in\cM(\Omega), P\subsetneq S\subseteq M\cup P,\text{ and } P\cap M=\emptyset}.\]

    Observe that the definition of the multi-hypergraph $J$ is almost identical to the one in the proof of~\cref{cl:RegularityConfigDegreeHolds}, taking the union over all packings $P$ containing the fixed packing $Q$ instead of a fixed graph $F$. Then, every edge in $J$ has size in $\{1,\ldots, g-s\}$. Let $J_p$ be the random subhypergraph of $J$, induced on the vertex set $\cV_p$. Assume that there exists an edge $P\in\cH^{(s)}$ with $Q\subset P$, such that $P\notin\mathbf{G}^{(s)}$. Then there exists $S\in\mathbf{G}$ such that $P\subsetneq S\subseteq M\cup P$ for some $M\in\cM(\cK)$. It follows that $S\setminus P$ is an edge of $J_p$, and
    \[d_{\cH^{(s)}}(Q) - d_{\mathbf{G}^{(s)}}(Q)\leq \sum_{k=1}^{g-s}e(J^{(k)}_p).\]    

    Fix $k\in[g-s]$ and let 
    \[\kappa_k:=\frac{n^{k+s-t-\beta}}{4g\cdot\log^{ak} n}.\]
    Let $Z\subseteq V(\mathbf{G})$ with $|Z|=i\in[k]$. Note that if $Z$ is not a packing, or if $Q\cap Z\neq\emptyset$, then $d_J(Z)$=0. Otherwise, for every edge of $J^{(k)}$ that contains $Z$, there exist $S\in\mathbf{G}^{(k+s)}$ such that $Z\cup Q \subseteq S$. Conversely, every edge $S\in\mathbf{G}^{(k+s)}$ induces at most $\binom{k+s}{s}$ edges in $J^{(k)}$, and it follows from~\cref{lem:ExtendPacking} that
    \begin{align*}
        d_{J^{(k)}}(Z) \leq \binom{k+s}{s}\size*{\set*{S\in\mathbf{G}^{(k+s)}\colon Z\cup Q\subseteq S}} = O\parens*{n^{k+s-i-t-1}} 
    \end{align*}
    and
    \[\Delta_i(J^{(k)})=O\parens*{n^{k+s-i-t-1}}.\]

    Similarly we obtain that
    \begin{align*}
        e(J^{(k)}) \leq \binom{k+s}{s}\size*{\set*{S\in\mathbf{G}^{(k+s)}\colon Q\subseteq S}} = O\parens*{n^{k+s-t-1}}\leq \kappa_k,
    \end{align*}        
    where we used $\beta<1$ (and $n$ large enough) for the last inequality. We apply~\cref{cor:KimVu} to the $k$-uniform multi-hypergraph $J^{(k)}$. Observe that,
    \[N:=v(J^{(k)})= v(\mathbf{G})= n^3.\]    

    For $i\in[k]$, we check that $(*_i)$ holds, 
    \[
    \frac{\Delta_i(J^{(k)})}{\kappa_k}\log^{4k+2}N
    = O\parens*{n^{k+s-i-t-1}}\cdot \frac{4g\cdot \log^{ak}n}{n^{k+s-t-\beta}}\cdot\log^{4k+2}n = O\parens*{\frac{\log^{a(k+1)} n}{n^{i+1-\beta}}}
    \leq \frac{\log^{ai}n}{n^i}
    =p^i.
    \]
    
    We then apply~\cref{cor:KimVu} to the $k$-uniform hypergraph $J^{(k)}$, and we obtain
    \[e(J^{(k)}_p)>2p^k\kappa_k = \frac{1}{2g}n^{s-t-\beta},\]
    with probability at most $e^{-\log^2 N}$. Therefore
        \begin{align*}
            \Prob{\exists t<s\in \{2,\ldots,g\}, \Delta_{t}(\cH^{(s)})> n^{s-t-\beta}}
            &\leq g^2\cdot \max_{t<s}\Prob{\Delta_{t}(\cH^{(s)})>n^{s-t-\beta}}\\
            &\leq g^2\cdot \max_{t<s}\Prob{\exists \text{ a packing }Q\subset \mathbf{D}, |Q|=t, \text{ s.t. } d_{\cH^{(s)}}(Q)> n^{s-t-\beta}}.
        \end{align*}  
        Remark that, by~\cref{lem:ExtendPacking}, for every packing $Q$ of size $t$, we have $d_{\mathbf{G}^{(s)}}(Q)=O\parens*{n^{s-t-1}}$, and therefore
        \begin{align*}
        \Prob{\exists t<s\in \{2,\ldots,g\}, \Delta_{t}(\cH^{(s)})> n^{s-t-\beta}}
            &\leq g^2\cdot n^{3g}\cdot \max_{t<s,Q}\Prob{d_{\cH^{(s)}}(Q)-d_{\mathbf{G}^{(s)}}(Q)>\frac{1}{2}  n^{s-t-\beta}}\\
            &\leq g^2\cdot n^{3g} \cdot \max_{t<s,Q}\Prob{\sum_{k=1}^{g-s} e(J_p^{(k)})> \frac{1}{2}  n^{s-t-\beta}}\\
            &\leq g^3\cdot n^{3g} \cdot\max_{t<s,Q}\Prob{e(J_p^{(k)})> \frac{1}{2g} n^{s-t-\beta}}\\
            &= g^3\cdot n^{3g} \cdot \max_{t<s,Q}\Prob{e(J_p^{(k)})> 2\kappa_k p^k}\\            
            &\leq g^3\cdot n^{3g} \cdot e^{-\log^2 N}\\
            &<\sigma/10
        \end{align*}
        given that $N>n$ is large enough.
    \end{proofclaim}


    \begin{claim}\label{cl:RegularityConfigMaxCommon2DegHolds}
        \cref{prop:QuantumRegMaxCommon2Deg}, ``The maximum common $2$-degree of $\cH$ with respect to $\cG_1$ is at most $n^{1-\beta}$'', holds with probability at least $1-\sigma/10$.    
    \end{claim}
    \begin{proofclaim} 
    Recall that, following~\cref{def:commontwodegree}, the maximum common $2$-degree of $\cH$ with respect to $\cG_1$ is
    \[\Delta^{(2)}_{\cG_1}(\cH) := \max_{\substack{F_1,F_2\in V(\cH)\\E(F_1)\cap E(F_2)=\emptyset}}\size*{\set*{F\in V(\cH)\colon \{FF_1,FF_2\}\subseteq E(\cH)}},\]
    and that 
        \[E(\cH)=\set*{P\subseteq V(\cH)\colon |P|\geq 2,~\exists S\in\mathbf{G},\exists M\in\cM(\cK) \text{ s.t. } P\subseteq S\subseteq M\cup P\text{ and } P\cap M=\emptyset}.\]

    Fix $F_1,F_2$, two copies of $K_3$ in $V(\mathbf{G})$ that are vertex-disjoint in $\cG_1$, hence edge-disjoint in $K_{n,n,n}$. Let $J=J(F_1,F_2)$ be the multi-hypergraph defined by $V(J)=V(\mathbf{G})$ and 
  
    \begin{align*}
    E(J)=\bigcup_{F\in V(\mathbf{G})} \bigcup_{k_1,k_2=4}^g\Big\{(S_1\cup S_2)\setminus\set{F,F_1,F_2}\colon 
    &S_i\in\mathbf{G}^{(k_i)}, \exists M_i\in\cM(\Omega)\text{ s.t. }
    \{F,F_i\}\cap M_i = \emptyset\text{ and }\\
    &\{F,F_i\}\subseteq S_i\subseteq M_i\cup\{F,F_i\}, \forall i\in\{1,2\}
    \Big\}.    
    \end{align*}

    Observe that every edge of $J$ has size $k\in\{1,\ldots,2g-3\}$. Let $J_p$ be the random subhypergraph of $J$, induced on the vertex set $\cV_p$. It is easy to see that once again we have, as desired, 
    \[\size*{\set*{F\in V(\cH)\colon \{FF_1,FF_2\}\subseteq E(\cH)}}\leq e(J_p).\]

    Fix $k\in\{1,\ldots,2g-3\}$, and let \[\kappa_k=\frac{1}{4g}\cdot\frac{n^{k+1-\beta}}{\log^{ak} n}.\]
    
    For a given edge $P$ of $J^{(k)}$, there exist $F$ and packings $S_1,S_2$, with respective size $k_1,k_2$, such that $P=(S_1\cup S_2)\setminus\set{F,F_1,F_2}$. Let $S_3=S_1\cap S_2$ and $k_3=|S_3|$. Then we have
    \[k=k_1+k_2-k_3-3,\quad k_1,k_2\in\{4,\ldots,g\},\quad k_3\in\{1,\ldots,\min\{k_1,k_2\}\}, \]
    where we used that $F\in(S_1\cap S_2)$, hence $k_3>0$. It follows that $4\leq k_1\leq k+3$. Observe further that, if $k_1=k+3$, then $k_2=k_3$, hence $S_2\subseteq S_1$, and as $S_1,S_2\in\mathbf{G}$, we have by minimality that $S_1=S_2$.

    For a fixed graph $F\in V(\mathbf{G})$ and a fixed $k_1\in\{4,\ldots,k+3\}$, let 
    \[\cS_1(k_1, F) = \Big\{S_1\in\mathbf{G}, \exists M_1\in\cM(\Omega)\text{ s.t. }
    \{F,F_1\}\cap M_1 = \emptyset\text{ and }
    \{F,F_1\}\subseteq S_1\subseteq M_1\cup\{F,F_1\} \Big\},    \]
    and for a fixed set of triangles $Q$, let
    \[\cS_1'(k_1, Q) = \Big\{S_1\in\mathbf{G}, Q\subseteq S_1 \text{ and }\exists F'\in S_1\text{ s.t. }|E(F')\cap E(Y)|=1.\Big\}. \]
    Observe that, as $\cM(\Omega)$ is $Y$-enclosed and $k_1\geq 4$, then
    \[\cS_1(k_1,F)\subseteq \cS'_1(k_1,\{F,F_1\})\subseteq \cS'_1(k_1,\{F_1\}).\]

    We then have 
    \[
    e(J^{(k)})\leq
    \sum_{F\in V(\mathbf{G})}\
    \sum_{k_1=4}^{k+3} \
    \sum_{S_1\in\cS_1(k_1,F)}\
    \sum_{\substack{S_3\subseteq S_1\\ F\in S_3}} N(k,k_1,S_3)
    ,\]
    where
    \[N(k,k_1,S_3) := \size*{\set*{S_2\in\mathbf{G}^{(k_2)}\colon k_2:=k-k_1+|S_3|+3,\exists M_2\in\cM(\Omega), (S_3\cup F_2)\subseteq S_2\subseteq M_2\cup S_3\cup F_2 }}.\]
    Therefore
    \begin{align}
        e(J^{(k)})
        &\leq
            \sum_{F\in V(\mathbf{G})}\
            \sum_{k_1=4}^{k+3} \
            \sum_{S_1\in\cS'_1(k_1,\{F,F_1\})}\
            \sum_{\substack{S_3\subseteq S_1\\ F\in S_3}} N(k,k_1,S_3)\nonumber\\
        &\leq
            \sum_{k_1=4}^{k+3} \
            \sum_{S_1\in\cS'_1(k_1,\{F_1\})}\
            \sum_{F\in S_1}\
            \sum_{\substack{S_3\subseteq S_1\\ F\in S_3}} N(k,k_1,S_3)\label{eq:P8LongSum}
    \end{align}

    For fixed $k_1,S_1,F,S_3$, if $k_1=k+3$, then $k_2=|S_3|$, hence 
    \[N(k,k_1,S_3)= 
    \begin{cases}
        1 &\text{ if }S_3=S_1,\\
        0 &\text{ otherwise. }
    \end{cases}.\]
    Then, if $k_1<k+3$, by~\cref{lem:ExtendPacking},
    \[N(k,k_1,S_3) =
    \begin{cases}
        O\parens*{n^{k_2-|S_3|-1}} = O\parens*{n^{k-k_1+2}}&\text{ if } F_2\in S_3,\\
        O\parens*{n^{k_2-|S_3|-2}} = O\parens*{n^{k-k_1+1}}&\text{ otherwise.}
    \end{cases}\]

    We therefore split~\cref{eq:P8LongSum} between
    \begin{itemize}
        \item $\aleph_1$, for $k_1=k+3$,
        \item $\aleph_2$, $k_1<k+3$ and $F_2\in S_1$, 
        \item $\aleph_3$, $k_1<k+3$ and $F_2\notin S_1$.
    \end{itemize}

    First we have, by~\cref{lem:ExtendPacking},
    \begin{align*}
    \aleph_1 &\leq 
    \sum_{\substack{S_1\in\cS'_1(k+3,\{F_1\})\\F_2\in S_1}}\
    \size{\set{F\in S_1}}\\
    &= O\parens*{n^{(k+3)-3}}\cdot (k+3) \\
    &= O\parens*{n^k}.
    \end{align*}
    
    Then
    \begin{align*}
    \aleph_2&\leq 
    \sum_{k_1=4}^{k+2} \
    \sum_{\substack{S_1\in\cS'_1(k_1,\{F_1\})\\F_2\in S_1}}\
    \sum_{F\in S_1}\
    \sum_{\substack{S_3\subseteq S_1\\ F\in S_3}} N(k,k_1,S_3)\\
    &=
    \sum_{k_1=4}^{k+2} n^{k_1-3}\cdot k_1\cdot 2^{k_1}\cdot O\parens*{n^{k-k_1+2}}\\
    &=O\parens*{n^{k-1}}.
    \end{align*}

    And finally
    \begin{align*}
    \aleph_3&\leq 
    \sum_{k_1=4}^{k+2}\
    \sum_{\substack{S_1\in\cS'_1(k_1,\{F_1\})\\F_2\notin S_1}}\
    \sum_{F\in S_1}\
    \sum_{\substack{S_3\subseteq S_1\\ F\in S_3}} N(k,k_1,S_3)\\
    &=
    \sum_{k_1=4}^{k+2} \frac{n^{k_1-1}}{\log^{ag} n}\cdot k_1\cdot 2^{k_1}\cdot O\parens*{n^{k-k_1+1}}\\
    &=O\parens*{\frac{n^{k}}{\log^{ag} n}}.
    \end{align*}

    It follows that 
    \[e(J^{(k)})=O(n^k)\leq \frac{1}{4g}\cdot\frac{n^{k+1-\beta}}{\log^{ak} n} = \kappa_k.\]
    
    For $i\in[k]$, let $Z\subseteq V(\mathbf{G})$ with $|Z|=i$. Remark that if $Z$ is not a packing of $K_{n,n,n}$ or if $\{F_1,F_2\}\cap Z\neq\emptyset$, then $d_J(Z)=0$. Otherwise, if $Z$ is in an edge of $J$, then there exists $S_1,S_2\in\mathbf{G}$ such that $S_1\cup S_2$ contains $Z$. Then every element of $Z$ can be either in $S_1\setminus S_2$, in $S_2\setminus S_1$ or in $S_1\cap S_2$. The computation is $d_{J^{(k)}}(Z)$ is then similar to the one of $e(J^{(k)})$, summing additionally over all $3$-partitions of $Z=Z_1\sqcup Z_2\sqcup Z_3$, and requiring $Z_1\cup Z_3\subseteq S_1$ and $Z_2\cup Z_3\subseteq S_2$, hence $Z_3\subseteq S_3$. We obtain

        \[
    d_{J^{(k)}}(Z)\leq 
    \sum_{\substack{Z_1,Z_2,Z_3\subseteq Z\\ Z_1\sqcup Z_2\sqcup Z_3 = Z}}
    \sum_{k_1=4}^{k+3} 
    \sum_{\substack{S_1\in\cS'_1(k_1,Q)\\Q:=\{F_1\}\cup Z_1\cup Z_3}}
    \sum_{F\in S_1}
    \sum_{\substack{S_3\subseteq S_1\\ (\{F\}\cup Z_3)\subseteq S_3}} N(k,k_1,S_3,Z_2)
    ,\] 
    
    where $N(k,k_1,S_3,Z_2)$ is as $N(k,k_1,S_3)$ replacing
    \begin{align*}
        (S_3\cup F_2)&\subseteq S_2\subseteq (M_2\cup S_3\cup F_2),
        \shortintertext{by}
        (S_3\cup F_2\cup Z_2)&\subseteq S_2\subseteq (M_2\cup S_3\cup F_2\cup Z_2),      
    \end{align*}
    that is
    \[N(k,k_1,S_3,Z_2) := \size*{\set*{S_2\in\mathbf{G}^{(k_2)}\colon k_2:=k-k_1+|S_3|+3,\exists M_2\in\cM(\Omega), (S_3\cup F_2\cup Z_2)\subseteq S_2\subseteq M_2\cup S_3\cup F_2\cup Z_2 }}.\]    
    The rest of the computation is as done for $e(J^{(k)})$, gaining at least a factor $n^{-|Z_1|-|Z_3|}$ in the summation over $\cS'_1$, and a factor $n^{-|Z_2|}$  in $N(k,k_1,S_3,Z_2)$, hence we obtain
    \[\Delta_i(J^{(k)}) = O\parens*{n^{k}\cdot n^{-|Z_1|-|Z_2|-|Z_3|}}= O\parens*{n^{k-i}}.\]

    We verify that $(*_i)$ holds, with $N:=v(J)= n^3$,
    \[\frac{\Delta_i(J^{(k)})}{\kappa_k}\log^{4k+2}N =O\parens*{n^{k-i}}\cdot\frac{\log^{ak} n}{n^{k+1-\beta}}\cdot\log^{4k+2}n = O\parens*{\frac{\log^{a(k+1)} n}{n^{i+1-\beta}}}\leq\frac{\log^{ai}n}{n^i}=p^i,\]
    where we used $a=8g\geq 4k+2$, and $\beta\in(0,1)$.  
    By~\cref{cor:KimVu}, we have 
    \[e(J^{(k)}_p)>  2p^k\kappa_k = \frac{\log^{ak} n}{n^k}\cdot\frac{1}{2g}\cdot\frac{n^{k+1-\beta}}{\log^{ak} n}= \frac{1}{2g} n^{1-\beta}\]
    with probability at most $e^{-\log^2 N}$. We obtain,

        \begin{align*}
            \Prob{\Delta^{(2)}_{\cG_1}(\cH)>n^{1-\beta}}
            &\leq\Prob{\exists \text{ disjoint }F_1,F_2\in V(\cH), \size*{\set*{F\in V(\cH)\colon \{FF_1,FF_2\}\subseteq E(\cH)}>n^{1-\beta} }}\\
            &\leq O(n^6)\cdot\max_{F_1,F_2}\Prob{\size*{\set*{F\in V(\cH)\colon \{FF_1,FF_2\}\subseteq E(\cH)}>n^{1-\beta} }}\\
            &\leq O(n^6)\cdot\max_{F_1,F_2}\Prob{e(J_p)>n^{1-\beta} }\\
            &\leq O(n^6)\cdot(2g)\cdot\max_{F_1,F_2}\Prob{e(J^{k}_p)> \frac{1}{2g} n^{1-\beta} }\\         
            &\leq  O(n^6) \cdot(2g) \cdot e^{-\log^2 N}\\
            &<\sigma/10
        \end{align*}
        given that $N>n$ is large enough.  
    \end{proofclaim}


    \begin{claim}\label{cl:RegularityConfigTwoCodegreeHolds}
        \cref{prop:QuantumTwoCodegrees}, ``The maximum $2$-codegree of $\cG_1$ with $\cH$ is at most $n^{1-\beta}$'', holds with probability at least $1-\sigma/10$.    
    \end{claim}
    
    \begin{proofclaim} 

    Recall that, following~\cref{def:icodegree}, the maximum $2$-codegree of $\cG_1$ with $\cH$ is
    \begin{align*}
     \Delta^{(2)}(\cG_1,\cH)&= \max_{\substack{e\in V(\cG_1), F\in E(\cG_1)\\e\not\in F}}\size*{\set*{S\in \cH^{(2)}\colon F\in S \text{ and } \exists F'\in S \text{ s.t }e\in F'}}\\    
    & = \max_{\substack{e\in V(\cG_1), F\in E(\cG_1)\\e\not\in F}}\size*{\set*{FF'\in \cH^{(2)}\colon e\in F'}},
    \end{align*}
    
    and that 
    \[E(\cH)=\set*{P\subseteq V(\cH)\colon |P|\geq 2,~\exists S\in\mathbf{G},\exists M\in\cM(\cK) \text{ s.t. } P\subseteq S\subseteq M\cup P\text{ and } P\cap M=\emptyset}.\]

    Fix $e\in V(\mathbf{D})$ and $F\in E(\mathbf{D})$ with $e\not\in F$. Let $J=J(e,F)$ be the multi-hypergraph such that $V(J)=V(\mathbf{G})$ and 
    \[E(J) = \bigcup_{\substack{F'\in E(\mathbf{D})\\ e\in F'}}\set{ S\setminus\{F,F'\}\colon S\in\mathbf{G}, \exists M\in\cM(\Omega), \{F,F'\}\cap M=\emptyset, \{F,F'\}\subseteq S\subseteq \{F,F'\}\cup M }.\]

    Observe that every edge in $J$ has size in $\{2,\ldots,g-2\}$, and that
    \[\size*{\set*{FF'\in \cH^{(2)}\colon e\in F'}}\leq e(J_p).\]
    
    Fix $k\in\{2,\ldots,g-2\}$, and let 
    \[\kappa_k=\frac{1}{2g}\frac{n^{k+1-\beta}}{\log^{ak} n}.\]
    
    We obtain from~\cref{lem:ExtendPacking} that
    \[e(J^{(k)})\leq \sum_{\substack{F'\in E(\mathbf{D})\\ e\in F'}}
    \size*{\set*{S\in \mathbf{G}^{(k+2)}\colon F,F'\in S}}
    = n\cdot O\parens*{n^{(k+2)-2-1}}
    = O\parens*{n^{k}}<\kappa_k.\]

    Let $Z\subseteq V(\mathbf{G})$ with $|Z|=i\in[k]$. Observe that if $Z$ is not a $K_3$-packing, if $F\in Z$, or if $e\in E(Z)$, then $d_J(Z)=0$. Otherwise,
    \[d_{J^{(k)}}(Z)\leq \sum_{\substack{F'\in E(\mathbf{D})\\ e\in F'}} \size*{\set*{S\in \mathbf{G}^{(k+2)}\colon (Z\cup \{F,F'\})\subseteq S}}.\]
    If $i<k$ we obtain from~\cref{lem:ExtendPacking} that,   
    \begin{align*}
    d_{J^{(k)}}(Z)\leq n\cdot O\parens*{n^{(k+2)-(i+2)-1}} = O\parens*{n^{k-i}}.
    \end{align*}        
    If $i=k$, if $Z\cup \{F\}\in\mathbf{G}$ then $d_{J^{(k)}}(Z)=0$, otherwise let $F'\in E(\mathbf{D})$ such that $S=Z\cup \{F,F'\}\in\mathbf{G}^{(k+2)}$. We obtain that $Z\cup \{F\}$ is a set of $k+1$ triangles spanning at least $k+4$ vertices, while $Z\cup \{F,F'\}$ is a set of $k+2$ triangles spanning at most $k+4$ vertices. There are then only constantly many choices for $F'$, hence $d_{J^{(k)}}(Z)=O(1)=O(n^{k-i})$. Then, with $N:=v(J)= n^3$, we check that $(*_i)$ holds,
    \[ \frac{\Delta_i(J^{(k)})}{\kappa_k}\log^{4k+2}N
    = O\parens*{n^{k-i}}\cdot\frac{\log^{ak} n}{n^{k+1-\beta}}\cdot\log^{4k+2}n
    = O\parens*{\frac{\log^{a(k+2)}n}{n^{i+1-\beta}}}
    \leq p^i,\]
    using $a=8g>4k+2$ and $\beta\in(0,1)$. By~\cref{cor:KimVu}, we have,
    \[e(J^{(k)}_p)> 2p^k\kappa_k = 2\cdot\frac{\log^{ak} n}{n^k}\cdot\frac{1}{2g}\cdot\frac{n^{k+1-\beta}}{\log^{ak} n}= \frac{1}{g} n^{1-\beta} \]
    with probability at most $e^{-\log^2 N}$. We obtain,

        \begin{align*}
            \Prob{\Delta^{(2)}(\cG_1,\cH)>n^{1-\beta}}
            &\leq\Prob{\exists e\in V(\cG_1), F\in E(\cG_1),\text{ s.t. }  \size*{\set*{FF'\in \cH^{(2)}\colon e\in F'}}>n^{1-\beta} }\\
            &\leq3n^2\cdot n^3\cdot \max_{e,F}\Prob{\size*{\set*{FF'\in \cH^{(2)}\colon e\in F'}}>n^{1-\beta} }\\
            &\leq3n^5\cdot\max_{e,F}\Prob{e(J_p)>n^{1-\beta} }\\
            &\leq3n^5\cdot g \cdot\max_{e,F}\Prob{e(J^{k}_p)> \frac{1}{g} n^{1-\beta} }\\         
            &\leq3n^5\cdot g \cdot e^{-\log^2 N}\\
            &<\sigma/10
        \end{align*}
        given that $N>n$ is large enough.
    \end{proofclaim}


    \begin{claim}\label{cl:AbundantQuantumPacking}
        \cref{prop:QuantumAbundant}, ``The treasury $\cT$ is $g$-abundant'', holds with probability at least $1-\sigma/10$.    
    \end{claim}
    
    \begin{proofclaim} 
    
        We prove the stronger statement that, with probability at least $1-\sigma/10$, for every copy $R$ of $K_3$ in $K_{n,n,n}$, at most $n^{2g-2}\log^{8g^2}n$ copies of $(R,g)$-spheres are not $\cH$-avoiding, which directly implies the desired result as the definition of abundance is weaker, see~\cref{def:abundant}.

        Let $R$ be a copy of $K_3$ in $K_{n,n,n}$, and let $\cR$ be the family of all $(R,g)$-spheres in $K_{n,n,n}$. Observe that by~\cref{rem:countingSpheres}, we have that $|\cR|\leq 3 n^{2g-1}$. Let $J=J(R)$ be the multi-hypergraph such that
        \[E(J) = \bigcup_{(B,\cBon,\cBoff,R)\in\cR} \bigcup_{\cB^*\in\{\cBon,\cBoff\}}
        \bigcup_{\substack{P\subseteq \cB^*\\|P|\geq 2}}
        \set{ S\setminus P\colon S\in\mathbf{G}, \exists M\in\cM(\Omega), P\cap M=\emptyset, P\subsetneq S\subseteq P\cup M }.\]        
        
        Observe that every edge in $J$ has size in $\{1,\ldots, g-2\}$. Let $J_p$ be the random subhypergraph of $J$, induced on the vertex set $\cV_p$. Assume that there exists an $(R,g)$-sphere in $K_{n,n,n}$ that is not $\cH$-avoiding. Then there exists a packing $P$ either in $\cBon$ or in $\cBoff$ such that $P$ is an edge of $\cH$. Therefore, by definition of $\cH$, there exists $S\in\mathbf{G}$ and $M\in\cM(\cK)$ such that $P\subsetneq S\subseteq M\cup P$ and $P\cap M=\emptyset$. It follows that $S\setminus P$ is an edge of $J_p$, and therefore the total number of non $\cH$-avoiding $(R,g)$-spheres is at most $\sum_{k=1}^{g-2}e(J^{(k)}_p)$. Let $k\in\set{1,\ldots g-2}$ and let 
        $$\kappa_k := \frac{1}{2g}\cdot n^{2g+k-2}\cdot \log n.$$
        By~\cref{rem:countingSpheres,lem:ExtendPacking} we obtain that 
        \begin{align*}
            e(J^{(k)}) &= |\cR|\cdot \sum_{\substack{P\subseteq \cBon\cup\cBoff\\t:=|P|\geq 2}} \size*{\set{ S\in\mathbf{G}^{(k+t)}\colon P\subseteq S }}.\\
            &= 3n^{2g-1 }\cdot 2^{4g-1} \cdot O\parens*{n^{k-1}}\\
            &= O\parens*{n^{2g+k-2}}\\
            &\leq \kappa_k.
        \end{align*}
    
        Fix $Z\subseteq V(J)$ with $|Z|=i\in[k]$. If $Z$ is not a packing then $d_J(Z)=0$. Otherwise, by~\cref{lem:ExtendPacking}, if~$i<k$ we have
        \begin{align*}
            d_{J^{(k)}}(Z)
            &\leq \sum_{(B,\cBon,\cBoff,R)\in\cR} \sum_{\substack{P\subseteq \cBon\cup\cBoff\\t:=|P|\geq 2\\P\cap Z=\emptyset}} \set{ S\in\mathbf{G}^{(k+t)}\colon  P\cup Z\subseteq S}    \\
            &\leq 3n^{2g-1 }\cdot 2^{4g-1} \cdot O\parens*{n^{k-i-1}}\\
            &= O\parens*{n^{2g+k-i-2}}.\\     
        \end{align*}

        Suppose now that $i=k$, and fix a $k$-packing $Z$. Every occurrence of $Z$ as an edge of $J^{(k)}$ arises from an $(R,g)$-sphere, a choice $\cB^*\in\{\cBon,\cBoff\}$, a packing $P\subseteq\cB_\star$ with $t:=|P|\geq2$, and an Erd\H{o}s configuration $S=P\cup Z\in\mathbf G^{(t+k)}$. Assume first that $k\geq2$. Since $P$ and $Z$ are proper subpackings of $S$, we have $|V(P)|\geq t+3$ and $|V(Z)|\geq k+3$. As $|V(S)|=t+k+2$, we obtain $|V(P)\cap V(Z)|\geq4$. Since $|V(R)|=3$, some vertex of the fixed packing $Z$ lies in $V(P)\setminus V(R)$, and hence occupies a non-root position of the sphere. There are therefore only $O_g(n^{2g-2})$ possible spheres.
        
        If $k=1$, let $F$ such that $Z=\{F\}$. Observe first that $F\neq R$. Indeed, if $P\in\cBoff$, $F=R$ contradicts the girth of $\cBoff\cup\{R\}$, while if $P\in\cBon$ and $F=R$, then $P\cup\{R\}$ is a packing, hence $P$ contains none of the three triangles of $\cBon$ sharing an edge with $R$.It follows that $V(P)\cap V(R)\subseteq\{v\}$, contradicting $V(R)\subseteq V(P)$. Since $|V(P)|\geq t+3=|V(S)|$, we have $V(F)\subseteq V(P)$. As $F\neq R$, at least one vertex of $F$ lies outside of $V(R)$, and the same argument applies, that is, \[d_{J^{(k)}}(Z)=O(n^{2g-2}).\]
        
        With $N:=v(J)=n^3$, we check that $(*_i)$ holds for all $i\in[k]$,
        \[ \frac{\Delta_i(J^{(k)})}{\kappa_k}\log^{4k+2}N
        = O\parens*{n^{2g+k-2-i}}\cdot\frac{\log^{4k+2}n}{n^{2g+k-2}\cdot \log n} 
        = O\parens*{\frac{\log^{4k+1}n}{n^{i}}}
        \leq p^i,\]
        using $k\geq1$.
        By~\cref{cor:KimVu}, we have,
        \[e(J^{(k)}_p)> 2p^k\kappa_k = 2\cdot\frac{\log^{ak} n}{n^k}\cdot \frac{n^{2g+k-2}\cdot \log n}{2g}= \frac{n^{2g-2}}{g}\log^{ak+1}n.\]
        with probability at most $e^{-\log^2 N}$. Using $k\leq g-2$, it follows that $e(J^{(k)}_p)> \frac{n^{2g-2}}{g}\log^{ag}n$ with probability at most $e^{-\log^2 N}$, and we obtain,
    
        \begin{align*}
            \Prob{\cT \text{ is not $g$-abundant}}
            &\leq\Prob{\exists R\in K_3, \text{ s.t. } e(J_p)>n^{2g-3/2} }\\
            &\leq\Prob{\exists R\in K_3, \text{ s.t. } e(J_p)>n^{2g-2}\cdot\log^{ag}n }\\
            &\leq n^3\cdot\max_{R}\Prob{e(J_p)>n^{2g-2}\cdot\log^{ag}n }\\
            &\leq n^3\cdot g\cdot \max_{R}\Prob{e(J^{(k)}_p)>\frac{1}{g}n^{2g-2}\cdot\log^{ag}n }\\
            &\leq n^3\cdot g \cdot e^{-\log^2 N}\\
            &<\sigma/10
        \end{align*}
        given that $N>n$ is large enough.  
    \end{proofclaim}    

    \newcommand{\properties}{\cref*{prop:FIRSTQuantum}-\cref*{prop:LASTQuantum}}
    Hitherto we proved that if $\cK$ is a random $Y$-enclosed quantum $K_3$-packing of $K_{n,n,n}$, chosen by selecting each copy of $K_3$ from $\Omega$ independently at random with probability $p$, then each one of property~\cref{prop:FIRSTQuantum} to~\cref{prop:LASTQuantum} hold with probability at least $1-\sigma/10$. Therefore by union bound, they all hold simultaneously with probability at least $1-\sigma$.
    
    Let $\nu$ be the probability distribution over $\cL_\cK$ induced by this selection process, and let $\eta$ be the probability distribution $\nu$ conditioning~\crefrange{prop:FIRSTQuantum}{prop:LASTQuantum}. Let $S$ be a fixed $K_3$-packing of $K_{n,n,n}$ of size $s$. Trivially, if $S$ is not $Y$-enclosed, then $\mathbb{P}_{\eta}\brackets*{S\subseteq \cK}=0$. Otherwise, 

    \[\mathbb{P}_{\eta}\brackets*{S\subseteq\cK} =
    \mathbb{P}_{\nu}\brackets*{S\subseteq\cK\mid \properties}
    \leq\frac{\mathbb{P}_{\nu}\brackets*{S\subseteq\cK}}{\mathbb{P}_{\nu}\brackets*{\properties}}
    \leq \frac{p^{|S|}}{1-\sigma}
    \leq \parens*{\frac{p}{1-\sigma}}^{|S|}.\]
    
    Therefore, with $p=\frac{\log^a n}{n}$, $\eta$ is a probability distribution that is $\frac{\log^a n}{(1-\sigma)n}$-spread over the family of $Y$-enclosed quantum $K_3$-packings of $K_{n,n,n}$, as desired.
\end{lateproof}

\section{Additional Tools For Completion Theorems}\label{sec:AdditionalTools}

The purpose of the completion step is to extend the packing produced by the pre-seeding theorem into a full high-girth triangle-decomposition. Recall that the output of the pre-seeding step is not merely a high-girth packing $S$, but also a regular and abundant subtreasury of the projection treasury $\mathrm{Proj}^g(K_{n,n,n},S,K_3)$. By~\cref{prop:ProjSubTreasury}, this yields a regular and abundant subtreasury of the natural treasury associated with the leftover graph $G:=K_{n,n,n}\setminus E(S)$.

The overall strategy follows the refined absorption framework developed by Delcourt and Postle for the High Girth Existence Conjecture. However, our setting differs in two fundamental respects. First, rather than working directly in the ambient treasury, we are only given this restricted subtreasury. Consequently, every stage of the argument (reserving edges, constructing absorbers, maintaining abundance, regularity boosting, and applying forbidden-submatching methods) must be carried out while keeping track of a sequence of nested subtreasuries. Second, our final goal is not merely to prove existence of a high-girth decomposition, but to construct a sufficiently spread probability distribution over such decompositions. This requires spread versions of several ingredients of the absorption method, most notably for absorbers and for the final forbidden-submatching step.

\subsection{Solvent Treasuries}

\begin{definition}[Reserve Design Hypergraph]\label{def:ReserveHypergraph}
Let $F$ be a hypergraph. If $G$ is a hypergraph and $A,B$ are disjoint subsets of $E(G)$, then the \emph{$F$-design reserve hypergraph of $G$ from $A$ to $B$}, denoted ${\rm Reserve}(G,F,A,B)$,  is the bipartite hypergraph $\mathbf{R}=(A,B)$ with $V(\mathbf{R}):=A\cup B$ and $$E(\mathbf{R}) := \{S\subseteq A\cup B: S \text{ is isomorphic to } F,~|S\cap A|=1\}.$$
\end{definition}

Observe that, in particular, each edge of ${\rm Reserve}(G,F,A,B)$ is a copy of $F$ in $G$ with exactly one edge in $A$, and the other edges in $B$.

\begin{definition}[Design treasury]\label{def:DesignTreasury}
Let $g\geq 4$ be an integer. Let $G$ be a graph and let $X\subseteq G'\subseteq G$. The \emph{girth-$g$ $K_3$-design treasury of $G$ with reserve $X$}, denoted ${\rm Treasury}^g(G,G',K_3,X)$, is the treasury 
\[T:=\left({\rm Design}(G'\setminus X,K_3),~{\rm Reserve}(G,K_3,G'\setminus X, X),~{\rm Girth}^g(G,K_3)\right).\]
\end{definition}

We begin the completion step by setting aside a small reserve $X$, which will later be used by the absorber to cover the leftover edges. At the level of treasuries, this amounts to separating the triangles of the original bankrupted treasury into two types: those avoiding $X$, which remain available for the main packing, and those using exactly two edges of $X$, which form the reserve hypergraph.

\begin{definition}[Savings]\label{def:savings}
    Let $g\geq 4$ be an integer and let $G$ be a graph.
    For a given bankrupted subtreasury $T=(G_1,H)$ of ${\rm Treasury}^g(G,G,K_3)$, and a set $X\subseteq G$, the {\em saving of $T$ by $X$}, denoted by ${\rm Saving}(T,X)$, is the treasury $\tilde{T}:=(\tilde{G}_1,\tilde{G}_2,\tilde{H})$ defined by
    \[\begin{cases}
        \tilde{G}_1 &:= G_1\cap {\rm Design}(G\setminus X,K_3),\\
        \tilde{G}_2 &:= G_1\cap {\rm Reserve}(G,K_3,G\setminus X, X) \\ 
        \tilde{H} &:= H.
    \end{cases}\]
    Observe that $(\tilde{G}_1,\tilde{H})$ is a subtreasury of $T$, and that $\tilde{T}$ is a subtreasury of ${\rm Treasury}^g(G,G,K_3,X)$.
\end{definition}

After constructing an absorber $A$, we no longer allow the main packing to use edges of $A$, since these edges will be decomposed later by the absorber itself. The {\em spending} operation therefore restricts the design and reserve hypergraphs to triangles avoiding $A$, while keeping the same configuration hypergraph so that the final matching remains compatible with the absorber decompositions.

\begin{definition}[Spending]\label{def:spending}
    Let $g\geq 4$ be an integer, let $G$ be a graph and let $X\subseteq G$.
    For a given (solvent) subtreasury $T=(G_1,G_2,H)$ of ${\rm Treasury}^g(G,G,K_3,X)$, and a set $A\subseteq G$, the {\em spending of $T$ by $A$}, denoted by ${\rm Spending}(T,A,X)$, is the treasury $\widehat{T}:=(\widehat{G}_1,\widehat{G}_2,\widehat{H})$ defined by
    \[\begin{cases}
        \widehat{G}_1 &:= G_1\cap {\rm Design}(G\setminus (A\cup X),K_3),\\
        \widehat{G}_2 &:= G_2\cap {\rm Reserve}(G,K_3,G\setminus (A\cup X), X) \\
        \widehat{H} &:= H.
    \end{cases}\]
    Observe that $\widehat{T}$ is a subtreasury of ${\rm Treasury}^g(G,G\setminus A,K_3,X)$.
\end{definition}

\subsection{Reserve Theorems}

The following reserve lemma is a tripartite graph version of~\cite[Lemma 3.1]{DKPIV}, with the additional feature that setting aside the reserve preserves the regularity of the underlying treasury.

\begin{lem}\label{thm:reserve}
  For every integer $g\geq 4$, reals $\alpha,\beta\in(0,1)$, and small enough constants $\varepsilon,\sigma\in(0,1)$, let $n$ be a large enough integer. Let $G$ be a $K_{1,1,1}$-divisible spanning subgraph of $K_{n,n,n}$ with $\delta(G)\ge (2-\sigma)n$ and vertex-partition $V^1,V^2,V^3$, and such that there exists a subtreasury $T$ of ${\rm Treasury}^g(G,G,K_3)$ that is $(n,\alpha n,\beta,\alpha)$-regular. Then there exists a spanning subgraph $X\subseteq G$ such that the following holds.
  \begin{enumerate}
      \item $X$ is edge-balanced.
      \item $\Delta(X)\leq 2n^{1-\varepsilon}$.
      \item For every $i\in[3]$, and every vertex $v\in V^i$, we have
      \[\size*{d_{V^{i-1}}(v)-d_{V^{i+1}}(v)}\leq n^{2/3}.\]
      \item ${\rm Saving}(T,X)$ is $(n,4\alpha n,\beta,\alpha)$-regular.
  \end{enumerate}
  
\end{lem}

The proof is relatively straightforward: we first choose a small random subgraph of $G$ as a preliminary reserve, then add a few extra edges to make it edge-balanced; the regularity of the saving treasury follows from standard concentration results, and the final sprinkling is too small to affect it substantially.

\begin{proof}
    Let $\sigma,\varepsilon\in(0,1)$ be arbitrarily small constants, in particular with $\varepsilon<\min\{\frac{\alpha}{9},\frac{1}{12}\}$, and $n$ be large enough. Let $\delta=n^{-1/3}$, $X$ be the spanning subgraph of $G$ obtained by choosing each edge of $G$ independently at random with probability $p=n^{-4\varepsilon}$, and let $V^1,V^2,V^3$ be the vertex-partition of $K_{n,n,n}$. For $1\leq i,j\in[3]$, let $X_{i,j}=X[V^i\cup V^j]$, and for any vertex $v\in V^i$, let $d_{V^j}(v)$ be the degree of $v$ in $X_{i,j}$.
    
    Let $\xi=\set*{\exists i\in[3],~\exists v\in V^i,\text{ s.t. } \size*{d_{V^{i-1}}(v)-d_{V^{i+1}}(v)} > 2pn^{2/3}}$. For a fixed $i\in[3]$ and every vertex $v\in V^i$, given that $G$ is $K_{1,1,1}$-divisible, we have 
    \[(1-\sigma)pn\leq \Expect{d_{i+1}(v)}=\Expect{d_{i-1}(v)}=p\cdot d_G(v)/2\leq pn.\]
    Observe further that
    \[\begin{cases}
        \set*{\size*{d_{V^{i-1}}(v)-\Expect{d_{i-1}(v)}}\leq pn^{2/3}}\\
        \set*{\size*{d_{V^{i+1}}(v)-\Expect{d_{i+1}(v)}}\leq pn^{2/3}}
    \end{cases}\implies
    \set*{\size*{d_{V^{i-1}}(v)-d_{V^{i+1}}(v)}\leq 2pn^{2/3}}.\]
    Therefore,
    \begin{align*}
    \Prob{\size*{d_{V^{i-1}}(v)-d_{V^{i+1}}(v)}> 2pn^{2/3}}
    &\leq 2\cdot\Prob{\size*{d_{V^{i-1}}(v)-\Expect{d_{i-1}}}> pn^{2/3}}\\
    &\leq 2\cdot\mathbb{P}\Big[\size*{d_{V^{i-1}}(v)-\Expect{d_{i-1}}}> \delta\Expect{d_{i-1}}\Big],     
    \end{align*}
    and by the Chernoff bound we obtain
    \begin{equation*}
        \Prob{\size*{d_{V^{i-1}}(v)-d_{V^{i+1}}(v)}> 2pn^{2/3}}\leq 2\cdot e^{-\delta^2(1-\sigma)pn/3}.        
    \end{equation*}  
    Finally, by the union bound, using $4\varepsilon<1/3$, we conclude that 
    \begin{equation}\label{eq:reservoir_degrees}
    \Prob{\xi}\leq 6n\cdot e^{-\delta^2(1-\sigma)pn/3} = o(1)
    \end{equation}
    We obtain similarly that
    \begin{equation}\label{eq:reservoir_Delta}
        \Prob{\Delta(X)>3pn} = o(1).
    \end{equation}
    Finally, for every $i\in[3]$, let $W_i$ denote the random variables counting the number of edges in $X_{i,i+1}$. Then, with $G$ being $K_{1,1,1}$-divisible, for every $i\in[3]$ we have $(1-\sigma)pn^2\leq \mathbb{E}W_i = p\cdot e(G)/3\leq pn^2$, and
    \begin{align*}
        \Prob{\size*{W_i-\mathbb{E}W_i}>n^{1-\varepsilon}}
        &\leq\Prob{\size*{W_i-\mathbb{E}W_i}>n^{-1+3\varepsilon} \cdot \mathbb{E}W_i}\\
        &\leq e^{-n^{-2+6\varepsilon}\cdot \mathbb{E}W_i / 3}\\
        &\leq e^{-n^{2\varepsilon}/3} ,
    \end{align*}
    hence
    \begin{equation}\label{eq:reservoir_almostbalanced}
        \Prob{\size*{W_i-\mathbb{E}W_i}>n^{1-\varepsilon}} = o(1)
    \end{equation}
    
    Let $T=(G_1,H)$ be the $(n,\alpha n,\beta,\alpha)$-regular subtreasury of ${\rm Treasury}^g(G,G,K_3)$, and let $\tilde{T}={\rm Saving}(T,X)=(\tilde{G}_1,\tilde{G}_2,\tilde{H})$ as defined in~\cref{def:savings}. We prove that with high probability, $\tilde{T}$ is $(n,3\alpha n,\beta,\alpha)$-regular.
    \begin{enumerate}
        \item For every vertex $v\in V(G_1)$ we have $d_{G_1}(v)\geq (1-\alpha)n$. 
        Observe that any edge $vxy$ from $G_1$ ``survives'' in $\tilde{G}_1$ if and only if $x,y$ are not selected in $X$. Therefore for all $v\in V(\tilde{G}_1)$ we obtain \[\mathbb{E}d_{\tilde{G}_1}(v) = (1-p)^2\cdot d_{G_1}(v)\geq (1-2\alpha)n.\]
        By a standard application of the Chernoff bound,
        \[\Prob{d_{\tilde{G}_1}(v)<(1-3\alpha)n}\leq e^{-\alpha^2 \mathbb{E}d_{\tilde{G}_1}(v) /2} = o(1).\]
        \item Similarly, for all $v\in V(\tilde{G}_1)\cap V(\tilde{G}_2)=E(G)\setminus X$, any edge $vxy$ from $G_1$ is ``transferred'' into $\tilde{G}_2$ if and only if $x,y$ are both selected in $X$. Therefore 
        \[\mathbb{E}d_{\tilde{G}_2}(v) = p^2\cdot d_{G_1}(v)\geq (1-\alpha)n^{1-8\varepsilon}\geq n^{1-9\varepsilon},\]
        and with $9\varepsilon<\alpha$,
        \[\Prob{d_{\tilde{G}_2}(v)<n^{1-\alpha}}\leq\Prob{d_{\tilde{G}_2}(v)<(1-\alpha)n^{1-9\varepsilon}}\leq e^{-\alpha^2 \mathbb{E}d_{\tilde{G}_2}(v) /2} = o(1).\]
        For $v\in E(X)$, we trivially have $d_{\tilde{G}_2}(v)\leq d_{G_1}(v)\leq n$.
        
        \item By definition, we have $\tilde{G}_1\cup \tilde{G}_2\subseteq G_1$, therefore the codegree of $\tilde{G}_1\cup \tilde{G}_2$ is at most the codegree of $G_1$, which is at most $n^{1-\beta}$ by regularity of $T$. Then observe further that $\tilde{H}=H$, therefore we immediately obtain that the maximum $2$-codegree of $\tilde{G}_1\cup \tilde{G}_2$ with $\tilde{H}$ and the maximum common $2$-degree of $\tilde{H}$ with respect to $\tilde{G}_1\cup \tilde{G}_2$ are both at most $n^{1-\beta}$
        \item Similarly, we immediately obtain that for all $2\le s\le g$ we have $\Delta_1\left(\tilde{H}^{(s)}\right) \le \alpha \cdot n^{s-1}\log n$, and that for all $2\le t<s\le g$, $\Delta_{t}\left(\tilde{H}^{(s)}\right) \le n^{s-t-\beta}$
    \end{enumerate}
    By the union bound, we obtain that 
    \begin{equation}\label{eq:reservoir_treasuryregular}
        \Prob{\tilde{T}\text{ is not $(n,3\alpha n,\beta,\alpha)$-regular}}=o(1).    
    \end{equation}

    Combining~\crefrange{eq:reservoir_degrees}{eq:reservoir_treasuryregular}, for $n$ large enough, let $X$ be a fixed spanning subgraph of $G$ such that none of these events happens. Observe that~\cref{eq:reservoir_almostbalanced} implies that
    \[\max\set*{\abs{W_i-W_{j}},~i,j\in[3]}\leq 2n^{1-\varepsilon}.\]

    Let \(m=\max_i W_i\). For each $i$, add $m-W_i$ edges from $G[V^i,V^{i+1}]\setminus X$ to $X[V^i,V^{i+1}]$, choosing these edges as a matching. This is possible greedily, since $G[V^i,V^{i+1}]\setminus X$ has linear minimum degree, while $m-W_i=O(n^{1-\varepsilon})=o(n)$. Let $X'$ denote the resulting graph. Then $X'$ is edge-balanced, and the added edges have maximum degree at most $2$, hence
    \[\Delta(X')\leq \Delta(X)+2\le 2n^{1-\varepsilon}\]
    for $n$ large enough. Moreover, the local degree-balance condition changes by at most $2$ at each vertex, and hence is at most $2pn^{2/3}+2\le n^{2/3}$.

    It remains to check regularity. Since ${\rm Saving}(T,X)$ is
    $(n,3\alpha n,\beta,\alpha)$-regular and $X'\setminus X$ has bounded maximum degree, passing from $X$ to $X'$ removes only constantly-many triangles through any fixed edge. Thus (RT1) still holds with
    error $4\alpha n$, for $n$ large enough. The reserve lower-degree condition is monotone under increasing $X$, the reserve upper-degree condition remains bounded by $n$, and the codegree and configuration-degree bounds are inherited from the original treasury $T$. Therefore ${\rm Saving}(T,X')$ is $(n,4\alpha n,\beta,\alpha)$-regular.
\end{proof}

\subsection{Refined Absorbers}

We recall some key definitions from~\cite{DPI}.

\begin{definition}[Omni-Absorber]
Let $q \ge 3$ be an integer and $X$ be a graph. We say a graph $A$ is a \emph{$K_q$-omni-absorber} for $X$ with \emph{decomposition family} $\cF$ and \emph{decomposition function} $\mathcal{Q}_A$ if $V(X)=V(A)$, $X$ and $A$ are edge-disjoint, $\cF$ is a family of subgraphs of $X\cup A$ each isomorphic to $K_q$ such that $|F\cap X|\le 1$ for all $F\in\cF$, and for every $K_q$-divisible subgraph $L$ of $X$, there exists $\mathcal{Q}_A(L)\subseteq \cF$ that are pairwise edge-disjoint and such that $\bigcup \mathcal{Q}_A(L)=L\cup A$. 
\end{definition}

\begin{definition}[Refined Omni-Absorber]
 Let $C\ge 1$ be real. We say that a $K_q$-omni-absorber $A$ for a graph $X$ with decomposition family $\cF$ is \emph{$C$-refined} if $|\{F\in \cF : e\in E(F) \}| \le C$ for every edge $e\in X\cup A$.
\end{definition}

Delcourt and Postle proved in~\cite{DPI} that linearly efficient refined omni-absorbers exist. For our problem, we need partite absorbers, which we call Latin-Absorbers.   

\begin{definition}[Refined Latin-Absorber]\label{def:LatinAbsorbers}
Let $A,X\subseteq K_{n,n,n}$ be tripartite graphs and let $C\ge 1$ be real.  We say that $A$ is a \emph{Latin-absorber} for $X$ with \emph{decomposition family} $\cF$ and \emph{decomposition function} $\mathcal{Q}_A$ if $V(X)=V(A)$, $X$ and $A$ are edge-disjoint, $\cF$ is a family of subgraphs of $X\cup A$ each isomorphic to $K_3$ such that $|F\cap X|\le 1$ for all $F\in\cF$, and for every $K_{1,1,1}$-divisible subgraph $L$ of $X$, there exists $\mathcal{Q}_A(L)\subseteq \cF$ that are pairwise edge-disjoint and such that $\bigcup \mathcal{Q}_A(L)=L\cup A$. 

We further say that $A$ is \emph{$C$-refined} if $|\{F\in \cF : e\in E(F) \}| \le C$ for every edge $e\in X\cup A$.
\end{definition}

Thankfully, in a forthcoming article, Delcourt and Postle~\cite{DP2026Partite} extended their method of refined absorption to the partite setting.

\begin{thm}[Refined Latin-Omni-Absorber Theorem, Delcourt and Postle~\cite{DP2026Partite}]\label{thm:LatinAbsorberThm}
    For every $\zeta\in(0,1)$, there exist reals $\eta,\sigma\in(0,1)$ and $C\geq 1$, and an integer $n_0$ such that the following holds for all $n\geq n_0$. 
    Let $G$ be a spanning subgraph of $K_{n,n,n}$ with $\delta(G) \geq (2-\sigma)n$. If $X$ is a spanning subgraph of $G$ with $\Delta(X) \leq n^{1-\zeta}$, then there exists a $C$-refined Latin-omni-absorber $A \subseteq G$ for $X$ such that $\Delta(A)\leq n^{1-\eta}$.
\end{thm}

\subsubsection{Treasuries and Projections}

Recall that for a packing $P$ and a treasury $T=(G_1,G_2,H)$, we say that $P$ is in $T$ if every triangle in $P$ is an edge of $G_1$, i.e. every triangle is ``allowed'' to be used for a decomposition, and if $P$ spans no edge of $H$, i.e. if collectively $P$ contains no forbidden configuration. We extend this definition to the following one, regarding absorbers in a treasury. Observe that this is a natural generalization of the concept of absorbers with {\em collective high-girth} presented in~\cite{DPII}, given that we now need our decomposition to happen inside a fixed treasury. 

\begin{definition}\label{def:LatinAbsorberWrtTreasury}
Let $X$ be a graph and $T=(G_1,H)$ be a bankrupted treasury. We say that a Latin-absorber $A$ for $X$ with decomposition function $\mathcal{Q}_A$ is in $T$ if $A\subseteq V(G_1)$ and for every $K_{1,1,1}$-divisible subgraph $L$ of $X$, we have that $\mathcal{Q}_A(L)$ is in $T$, that is, $\mathcal{Q}_A(L)\subseteq G_1$ and $\mathcal{Q}_A(L)$ is $H$-avoiding.
\end{definition}

It is important to understand that it is however not enough for the proof of~\cref{thm:MinDegree} to use omni-absorbers in the given treasury $T$. The issue is that the various $\mathcal{Q}_A(L)$, individually or collectively, may forbid subconfigurations in $G\setminus (X\cup A)$ that would extend to forbidden configurations with $\mathcal{Q}_A(L)$. We will need to ensure there are not too many such forbidden subconfigurations. We therefore need to build absorbers that also do not {\em ``shrink too many''} configurations from $T$. The precise meaning of this requirement uses an extension of projection as defined for packings in~\cref{def:ProjectionPacking}, to projection of omni-absorbers.

\begin{definition}[Projection of Omni-Absorber]\label{def:OmniAbsorberProj}
Let $g\geq4$ be an integer, $G\subseteq K_{n,n,n}$, $X$ be a subgraph of $G$, and let $T$ be a bankrupted subtreasury of ${\rm Treasury}^g(G,~G,~K_3)$. Let $A\subseteq G$ be a Latin-omni-absorber for $X$ with decomposition function $\mathcal{Q}_A$, such that $A$ is in $T$. We define the \emph{projection treasury} of $A$ into $\tilde{T}:={\rm Saving}(T,X)$ as:
$${\rm Proj}^g(\tilde{T},A,X):= {\rm Spending}(\tilde{T},A,X) \perp \mathcal{M}(A)$$
where $$\mathcal{M}(A) := \{ \mathcal{Q}_A(L) : L\in{\rm Div}(X)\}.$$
\end{definition}

The following proposition explains how finding a perfect matching in the projection treasury relates to a perfect matching in the initial one. 

\begin{proposition}\label{prop:FindPM}
Let $g\geq 4$ be an integer. Let $G\subseteq K_{n,n,n}$ be $K_{1,1,1}$-divisible, $X$ be a subgraph of $G$, $T$ be a bankrupted subtreasury of ${\rm Treasury}^g(G,~G,~K_3)$, and let $\tilde{T}= {\rm Saving}(T,X)$. Let $A\subseteq G$ be a Latin-omni-absorber for $X$ in $T$ with decomposition function $\cQ_A$. If $M$ is a perfect matching of ${\rm Proj}^g(\tilde{T},A,X)$, then $M\cup \cQ_A(X\setminus\bigcup M)$ is a perfect matching of $T$.
\end{proposition}

\begin{proof}
    To keep track on the various operations made on treasuries, we define the following hypergraphs.
    \begin{enumerate}[topsep=.1in,itemsep=.1in]
        \item Let $G_1,H$ such that $T=(G_1,H)$ is the subtreasury of ${\rm Treasury}^g(G,~G,~K_3)$ given by assumptions.
        
        \item Let $\tilde{G}_1,\tilde{G}_2,\tilde{H}$ such that $\tilde{T}=(\tilde{G}_1,\tilde{G}_2,\tilde{H})$ is the treasury after ``saving'' the reserve $X$. Observe that we have in particular 
            \begin{equation}\label{eq:SavingHypergraph}
            \tilde{G}_1\cup\tilde{G}_2\subseteq G_1,\text{ and } \tilde{H} = H.    
            \end{equation}
        
        \item Let $\widehat{G}_1,\widehat{G}_2,\widehat{H}$ such that $\widehat{T}=(\widehat{G}_1,\widehat{G}_2,\widehat{H})$ is the treasury after ``spending'' the omni-absorber $A$, i.e
            \(\widehat{T}={\rm Spending}(\tilde{T},A,X),\)
            in particular with
            \begin{equation}\label{eq:ResHypergraph}
                \widehat{H}=\tilde{H},\qquad \widehat{G}_1\subseteq\tilde{G}_1,\text{ and }\qquad\widehat{G}_2\subseteq\tilde{G}_2.
            \end{equation}        

        \item Let $G'_1,G'_2,H'$ such that $T'=(G'_1,G'_2,H')$ is the projection of $A$ into $\tilde{T}$, that is
            \[ T' = {\rm Proj}^g(\tilde{T},A,X) = {\rm Spending}(\tilde{T},A,X)\perp \cM(A) = \widehat{T}\perp\cM(A),\]
            in particular with 
            \begin{equation}\label{eq:ProjHypergraph}
                G'_1\subseteq\widehat{G}_1,\text{ and }\qquad G'_2\subseteq\widehat{G}_2.
            \end{equation}  
    \end{enumerate}
     
    We are given $M$, a perfect matching of $T'$, i.e. an $H'$-avoiding $G\setminus(A\cup X)$-perfect matching of $G'_1\cup G'_2$. By definition of projection, it is an $\widehat{H}$-avoiding $G\setminus(A\cup X)$-perfect matching of $G'_1\cup G'_2$, such that for all $L\in{\rm Div}(X)$ we have 
    $M\cup \cQ_A(L)$ is $\widehat H$-avoiding. Let $L:=(G\setminus A)\setminus \bigcup M$. Observe that $G$ and $A$ are $K_{1,1,1}$-divisible, hence $L\in\div{X}$. It follows that $\cQ_A(L)$ decomposes $A\cup L$, while the projection ensures that $M\cup \cQ_A(L)$ is $\widehat{H}$-avoiding and all its triangles are in $G_1$. Combining\cref{eq:SavingHypergraph,eq:ResHypergraph,eq:ProjHypergraph}, we obtain that $M\cup \cQ_A(L)$ is a $H$-avoiding $G$-perfect matching of $G_1$, hence a perfect matching of $T$.
\end{proof}

We are now ready to state one of our main technical results, the existence of a spread distribution over omni-absorbers in a given (regular and abundant) treasury. While the statement might seem intricate, we observe that the assumptions and outputs of this theorem are relatively natural once the overview of the refined-absorption methodology is well-understood: Given a graph $G$ with high minimum degree, a parent treasury telling us what triangles may be used during the absorption process, a set of ``reserved'' edges $X$, and the saving treasury in which the remainder packing is built, we need to find an efficient omni-absorber $A$, that does not ``shrink'' too many configurations (see the remark after~\cref{def:LatinAbsorberWrtTreasury}). The theorem then further guarantees that we can find a sufficiently-spread distribution over these absorbers.

\begin{thm}[Spread Latin-Omni-Absorber Theorem]\label{thm:MainLatinAbsorbers}
For all integers $g\geq 4$ and reals $\zeta\in(0,1)$, $\alpha\in(0,1/2)$ and $\beta\in(0,1/(8g))$, there exist an integer $n_0\ge 1$ and reals $\eta,\sigma_0\in(0,1)$ such that the following holds for all $n\ge n_0$ and all $\sigma\leq\sigma_0$.

Let $G\subseteq K_{n,n,n}$ with $\delta(G)\geq(2-\sigma)n$, and assume that $X$ is a spanning subgraph of $G$ with $\Delta(X)\leq n^{1-\zeta}$ such that there exists a bankrupted subtreasury $T$ of ${\rm Treasury}^g(G,~G,~K_3)$ such that $T$ is $g$-abundant and both $T$ and $\tilde{T}= {\rm Saving}(T,X)$ are $\parens*{n,~\alpha n,~2\beta,~\alpha}$-regular.

Let $\cL_\cA$ be the family of Latin-omni-absorbers $A\subseteq G$ for $X$ in $T$, with decomposition family $\cF_A$ and decomposition function $\cQ_A$, with $\Delta(A)\leq n^{1-\eta}$, such that  ${\rm Proj}^g(\tilde{T},A,X)$ contains a subtreasury that is $\parens*{n,2\alpha n,\beta,2\alpha}$-regular.

Then $\cL_\cA$ is non-empty and furthermore there exists a probability distribution over $\cL_\cA$ such that for every $K_3$-packing $S$ of $G$,
  \begin{equation*}
    \Prob{\exists L\in \div{X}\colon S \subseteq \cQ_A(L)} \leq (n^{-1+3/(2g)})^{|S|} .
  \end{equation*}
\end{thm}

The fact that $\cL_\cA$ is non-empty is already highly non-trivial, let alone the existence of a spread distribution over $\cL_\cA$. The proof of this statement requires some lengthy explanation, with the introduction of new tools, and is therefore deferred to~\cref{sec:HighGirthAbsorption}.

\subsection{Regularity Boosting}

To handle step (3) of our proof outline, we use a variant of the Boost Lemma of Glock, K\"{u}hn, Lo, and Osthus~\cite{GKLO16}, due to Kwan, Sah, Sawhney, Simkin~\cite{KSSS2023substructures} in the tripartite and triangle case, called Triangle-regularization. Given a set of triangles in $G\setminus (A\cup X)$ with suitable regularity and ``extendability'' properties, the lemma finds a subset which is significantly more regular.

\begin{lem}[Kwan, Sah, Sawhney, Simkin~\cite{KSSS2023substructures}]\label{lem:RegBoost}
    There are $C_0\in\bb{N}$ and $n_0:\bb{N}\times \bb{R}\to\bb{N}$ such that the following holds. We are given $\kappa\in(0,1)$ and $C\geq C_0$, and let $n\geq n_0(C, \kappa)$. Suppose , $\xi= C^{-8}$ , and $p\in(n^{-1/12}, 1)$, let $G$ be a balanced tripartite graph on $3n$ vertices $V^1 \sqcup V^2\sqcup V^3$, and let $\cT$ be a collection of triangles of $G$, satisfying the following properties.
    \begin{enumerate}[label=(\alph*)]
        \item Every edge $e\in E(G)$ is in $(1\pm\xi)p^2\kappa n$ triangles of $\cT$.
        \item For every $j\in[3]$ and every set $Q\subseteq V^{j-1}\cup V^j$ with $|Q|\leq 6$  in $G$, there are $(1\pm\xi)p^{|Q|}n$ common neighbors of $Q$ in $V^{j+1}$.
        \item For every $j\in[3]$ and every set $Q$ of at most $6$ edges between $V^{j-1}$ and $V^j$, there are between $C^{-1}p^{|V(Q)|}n$ and $Cp^{|V(Q)|}n$ vertices $u\in V^{j+1}$ which form a triangle in $\cT$ with every edge in $Q$.
        \item $G$ has the same number of edges between each pair of parts, and for every $j\in[3]$ and every $v\in V^j$, we have $\size*{d_{V^{j-1}}(v)-d_{V^{j+1}}(v)}\leq n^{2/3}$.
    \end{enumerate}
    Then, there is a subcollection $\cT'\subseteq\cT$ such that every edge $e\in E(G)$ is in $(1\pm n^{-1/4})p^2\kappa n/4$ triangles of~$\cT'$.
\end{lem}

This regularity-boosting lemma will be applied on the treasuries given by the absorber theorem,~\cref{thm:MainLatinAbsorbers}, and yield a set of triangles that is suitable for the ``nibble-like'' approach of the Forbidden submatching theorem presented in the next subsection.

\subsection{Forbidden Submatchings With Reserves}\label{ss:ForbiddenSubmatching}

The last piece of the absorption-methodology puzzle is a nibble approach applied to the graph $G$ we obtained after removing the reserve $X$ and absorber $A$. The goal is for this to yield a triangle-packing of $G\setminus A$ that covers all edges in $G\setminus (A\cup X)$. The remaining edges will be decomposed using properties of the absorber $A$. The use of the Forbidden Submatching with reserve methodology, instead of a simpler nibble with reserve one, ensures that the resulting decomposition spans no edges from the configuration hypergraph (it avoids in particular low-girth packings). For our purpose, we require the base theorem as established by Delcourt and Postle~\cite{DP22}, on top of which we add a spread component.

\begin{thm}[Forbidden Submatchings with Reserves~\cite{DP22}]\label{thm:ForbiddenSubmatchingReserves}
For all integers $r,g \ge 2$ and real $\beta \in (0,1)$, there exist an integer $D_0\ge 0$ and real $\alpha_0 > 0$ such that the following holds for all $D\ge D_{0}$. If $T$ is an $(r,g)$-treasury that is $(D,D^{1-\beta},\beta,\alpha_0)$-regular, then there exists a perfect matching of $T$. 
\end{thm}

We then obtain a spread variant of the Forbidden Submatchings with reserve Theorem, by adding a sparsification step and therefore ensuring that no packing receive excessive weight in the resulting distribution. This line of argument has been used in~\cite{DKPIV} for instance, using a ``Local Lemma Distribution'' approach, on a nibble with reserve base theorem. We extend this to take into account the configuration hypergraph of the treasury, and basing the probabilistic arguments on the Chernoff's and Union bounds for simplicity.

\begin{thm}[Forbidden Submatchings with Reserves - spread version]\label{thm:ForbiddenSubmatchingReservesSpread}
For all integers $r,g \ge 2$ and reals $\beta,\gamma \in (0,1)$ and $c\geq 0$, the following holds for sufficiently large $D$ and sufficiently small $\alpha > 0$. If $T=(G_1,G_2,H)$ is an $(r,g)$-treasury that is $(D,D^{1-\beta},\beta,\alpha)$-regular, with $V(G_1)\cap V(G_2)=A$ and $\size*{V(G_1)\cup V(G_2)}\leq D^c$, then there exists a probability distribution over perfect matchings $M$ of $T$ such that, for every $S\subseteq E(G_1\cup G_2)$,
  \begin{equation*}
    \Prob{S \subseteq M} \leq \left(\frac{1}{D^{1 - \gamma}}\right)^{|S|} .
  \end{equation*}
\end{thm}

\begin{nonlateproof}{thm:ForbiddenSubmatchingReservesSpread}

    Observe that it suffices to prove the theorem for small enough $\gamma$, as $\Gamma>\gamma$ implies $\frac{1}{D^{1-\gamma}}< \frac{1}{D^{1-\Gamma}}$. In particular, we assume $\gamma<\beta$. Let $\alpha$ be arbitrarily small, with $6\alpha< \gamma$, and let $D$ be large enough so that various inequalities present in this proof hold.
    
    Let $T'=(G'_1,G'_2,H')$ be the random subtreasury of $T$ obtained by selecting every edge of $G_1,G_2$ independently at random with probability $p:=D^{\gamma/2-1}$ , and by setting $H'=H[E(G_1')\cup E(G_2')]$. Note that we can ask $G_1,G_2$ to be edge disjoint for simplicity. Let $B$ such that $A\sqcup B=V(G_2)$, and we define
    \[\delta := D^{-\gamma/6},\qquad \alpha':=4\frac{\alpha}{\gamma},\qquad D':=pD(1+\delta)=D^{\gamma/2}+D^{\gamma/3}.\] We now prove that with high probability, $T'$ is $(D',(D')^{1-1/4},\frac{1}{4},\alpha')$-regular. Observe that we can take $\alpha$ small enough and $D$ large enough so that~\cref{thm:ForbiddenSubmatchingReserves} would then apply.

    We often use the following observation, if $X$ is a binomial random variable $X\sim{\rm Bin}(N,p)$ on $N\leq f(D)$ trials, then $X$ is dominated by $Y\sim{\rm Bin}(f(D),p)$.
    \begin{remark}\label{rem:Chernoff_upper_bound}
        For positive integers $N_X<N_Y$ and reals $\delta>0$ and $p\in(0,1)$, let $X\sim{\rm Bin}(N_X,p)$ and $Y\sim{\rm Bin}(N_Y,p)$. Then, by the Chernoff bound~\cite{AS16},
        \[\Prob{X\geq (1+\delta)\mathbb{E}Y}
        \leq \Prob{Y\geq (1+\delta)\mathbb{E}Y}
        \leq \exp\brackets*{-\delta^2 \mathbb{E}Y/(2+\delta)}.
        \]
    \end{remark} 
    We split the proof that $T'$ is $(D',(D')^{1-1/4},\frac{1}{4},\alpha')$-regular with high probability into four claims, one for each one of Properties~\crefrange{item:RT_QuasiReg}{item:RT_ConfigDeg}.

    \begin{claim}\label{clm:spread_nibble_RT1}
        With high probability, every vertex of $G_1$ has degree at most $D'$ in $G'_1$ and every vertex of $A$ has degree at least $D'(1-{D'}^{-1/4})$ in $G'_1$, hence~\cref{item:RT_QuasiReg} holds with high probability.
    \end{claim}
    \begin{proofclaim}
        With $T$ being $(D,D^{1-\beta},\beta,\alpha)$-regular, we know that for every vertex $v\in G_1$, $\mathbb{E}d_{G'_1}(v)\leq pD = D^{\gamma/2}$. Using the union bound and~\cref{rem:Chernoff_upper_bound} with $X=d_{G'_1}(v)\sim{\rm Bin}(d_{G_1}(v),p)$ and $Y\sim{\rm Bin}(D,p)$, we obtain
        \begin{align*}
            \Prob{\exists v\in G_1,~d_{G'_1}(v)>D'}
            &=\Prob{\exists v\in G_1,~d_{G'_1}(v)>(1+\delta)pD}\\
            &\leq D^c\cdot\exp\brackets*{-\delta^2 D^{\gamma/2}/(2+\delta)}\\
            &\leq D^c\cdot\exp\brackets*{-D^{\gamma/6}/3}\\
            &=o(1).
        \end{align*}
        Meanwhile, with $\gamma<\beta$ and $D'=D^{\gamma/2}\parens*{1+D^{-\gamma/6}}\leq 2D^{\gamma/2}$, we have for $D$ large enough
        \begin{align*}
            D'(1-{D'}^{-1/4})
        &\leq D^{\gamma/2}\parens*{1+D^{-\gamma/6}}\parens*{1-(2D^{\gamma/2})^{-1/4}}\\
        &\leq D^{\gamma/2}\parens*{1-D^{-\gamma/7}}\\
        &\leq D^{\gamma/2}\parens*{1-D^{-\gamma/6}}\parens*{1-D^{-\beta}}.
        \end{align*}
        With $T$ being $(D,D^{1-\beta},\beta,\alpha)$-regular, we also know that for every vertex $v\in A$, $\mathbb{E}d_{G'_1}(v)\geq D^{\gamma/2}(1-D^{-\beta})$. Then by the Chernoff bound and the union bound, we obtain
        \begin{align*}
            \Prob{\exists v\in A,~d_{G'_1}(v)<D'(1-{D'}^{-1/4})}
            &\leq \Prob{\exists v\in A,~d_{G'_1}(v)<D^{\gamma/2}\parens*{1-D^{-\gamma/6}}\parens*{1-D^{-\beta}}}\\
            &\leq \Prob{\exists v\in A,~d_{G'_1}(v)<(1-\delta)\mathbb{E}d_{G'_1}(v)}\\
            &\leq D^c \cdot\exp\brackets*{-\delta^2 \mathbb{E}d_{G'_1}(v)/2}\\    
            &\leq D^c \cdot\exp\brackets*{-\delta^2 D^{\gamma/2}(1-D^{-\beta})/2}\\
            &\leq D^c \cdot\exp\brackets*{- D^{\gamma/6}(1-D^{-\beta})/2}\\
            &=o(1).\qedhere
        \end{align*}
    \end{proofclaim} 

    \begin{claim}\label{clm:spread_nibble_RT2}
        With high probability, every vertex of $B$ has degree at most $D'$ in $G'_2$, and every vertex of $A$ has degree at least $D'^{1-\alpha'}$ in $G'_2$, hence~\cref{item:RT_ReserveDeg} holds with high probability.
    \end{claim}

    \begin{proofclaim}
        With $T$ being $(D,D^{1-\beta},\beta,\alpha)$-regular, we know that for every vertex $v\in B$, $\mathbb{E}d_{G'_2}(v)\leq D^{\gamma/2}$, and that for every vertex $v\in A$, $\mathbb{E}d_{G'_2}(v)\geq D^{\gamma/2-\alpha}$. Using the union bound and~\cref{rem:Chernoff_upper_bound} we obtain
        \begin{align*}
            \Prob{\exists v\in B,~d_{G'_2}(v)>D'}
            &=\Prob{\exists v\in B,~d_{G'_2}(v)>(1+\delta)pD}\\
            &\leq D^c\cdot\exp\brackets*{-\delta^2 D^{\gamma/2}/(2+\delta)}\\
            &=o(1).
        \end{align*}        
        For any constant $\kappa>0$ and $D$ large enough, we have $(1-\delta)D^\alpha \geq (1+\delta)^{\kappa}$, therefore
        \[{D'}^{1-\alpha'}
         = D^{(\gamma/2)(1-4\alpha/\gamma)}(1+\delta)^{1-\alpha'}\\
        \leq D^{(\gamma/2)-2\alpha}(1+\delta)^{1-\alpha'}
        \leq  (1-\delta)D^{(\gamma/2)-\alpha}.\]
        By the Chernoff bound and the union bound, we obtain
        \begin{align*}
            \Prob{\exists v\in A,~d_{G'_2}(v)<{D'}^{1-\alpha'}}
            &\leq \Prob{\exists v\in A,~d_{G'_2}(v)<(1-\delta)D^{(\gamma/2)-\alpha}}\\
            &\leq \Prob{\exists v\in A,~d_{G'_2}(v)<(1-\delta)\mathbb{E}d_{G'_2}(v)}\\
            &\leq D^c \cdot\exp\brackets*{-\delta^2  D^{\gamma/2-\alpha}/2}\\
            &\leq D^c\cdot\exp\brackets*{-D^{\gamma/6-\alpha}/2}\\
            &=o(1).\qedhere    
        \end{align*}    
    \end{proofclaim} 
   
    \begin{claim}\label{clm:spread_nibble_RT3}
        With high probability, $G'_1\cup G'_2$ has codegrees at most ${D'}^{1-1/4}$, and the maximum $2$-codegree of $G'_1\cup G'_2$ with $H'$ and the maximum common $2$-degree of $H'$ with respect to $G'_1\cup G'_2$ are both at most ${D'}^{1-1/4}$, hence~\cref{item:RT_CoDeg} holds with high probability.
    \end{claim}
    \begin{proofclaim}
    
        For every pair of vertices $u,v$ in $G_1\cup G_2$, the expected codegree of $\{u,v\}$ is at most $pD^{1-\beta}=D^{(\gamma/2)-\beta}$. Observe that for $D$ large enough, using $\gamma<\beta$ and therefore $\beta-\gamma/8-\gamma/6>0$
        \begin{equation}
        D^{\beta-\gamma/8}(1+\delta)^{3/4}\geq (1+D^{\beta-\gamma/8}),\label{eq:bound_D}
        \end{equation}
        hence
        \begin{align*}
            {D'}^{3/4} 
        &= D^{3\gamma/8}(1+\delta)^{3/4}\\
        &= D^{(\gamma/2)-\beta+(\beta-\gamma/8)}(1+\delta)^{3/4}\\
        &\geq D^{(\gamma/2)-\beta}(1+D^{\beta-\gamma/8}).
        \end{align*}
        Using the union bound and~\cref{rem:Chernoff_upper_bound}, we obtain, 
        \begin{align*}
        \Prob{G'_1\cup G'_2 \text{ has codegree at least }{D'}^{1-1/4} }
        &\leq\Prob{G'_1\cup G'_2 \text{ has codegree at least }D^{(\gamma/2)-\beta}(1+D^{\beta-\gamma/8}) } \\
            &\leq D^{2c} \cdot\exp\brackets*{-D^{2\beta-\gamma/4} D^{\gamma/2-\beta}/(2D^{\beta-\gamma/8})}\\
            &\leq D^{2c}\cdot\exp\brackets*{-D^{3\gamma/8}/2}\\
            &=o(1).     
        \end{align*}
        
        We obtain with identical computation that the maximum $2$-codegree of $G'_1\cup G'_2$ with $H'$, and the maximum common $2$-degree of $H'$ with respect to $G'_1\cup G'_2$ are both with high probability at most ${D'}^{1-1/4}$. 
    \end{proofclaim}
    
    \begin{claim}\label{clm:spread_nibble_RT4}
        With high probability, for all $2\le s\le g$ we have $\Delta_1\left({H'}^{(s)}\right) \le \alpha' \cdot (D')^{s-1}\log D'$, and for all $2\le t<s\le g$, $\Delta_{t}\left({H'}^{(s)}\right) \le (D')^{s-t-1/4}$, hence~\cref{item:RT_ConfigDeg} holds with high probability.
    \end{claim}
    \begin{proofclaim}
        Recall that by regularity of $T$, for all $2\le s\le g$ we have 
        \[\Delta_1\left(H^{(s)}\right) \le \alpha \cdot D^{s-1}\log D,\] and for all $2\le t<s\le g$, \[\Delta_{t}\left(H^{(s)}\right) \le D^{s-t-\beta}.\]
        Furthermore, as $H$ is a configuration hypergraph for the $r$-bounded hypergraph $G_1\cup G_2$ with $\size*{V(G_1)\cup V(G_2)}\leq D^c$, it follows that, for any $1\leq t\leq g$, the number of $t$-subsets of $V(H)$ is at most $D^{crt}$. Then for all $F\in V(H')$ and all $2\leq s\leq g$, we have
        \[\mathbb{E}d_{{H'}^{(s)}}(F)=p^{s-1}d_{{H}^{(s)}}(F)\leq \alpha \cdot p^{s-1}D^{s-1}\log D.\]

        For concentration, we apply~\cref{cor:KimVu} as we did multiple times in the proof of~\cref{thm:SpreadQuantumPreseeding}. Fix $1\leq t<s\leq g$, and a $t$-set $Q\subseteq V(H)$. Let $J=J(Q)$ be $(s-t)$-uniform hypergraph defined by
        \[J:=\set*{Z\setminus Q:Z\in H^{(s)},\ Q\subseteq Z}.\]
        Let $N:=D^{c+1}$, so that $v(J)\leq N$. If $Q\subseteq V(H')$, we then have $d_{H'^{(s)}}(Q)=e(J_p)$, and otherwise this degree is null. By regularity of $T$,
        \[e(J)\leq
          \begin{cases}
            \alpha D^{s-t}\log D,&\text{ if }t=1,\\
            D^{s-t-\beta},&\text{ if }t\geq2,
          \end{cases}
        \]
        and, for every $1\le i<s-t$,
        \[\Delta_i(J)\leq D^{s-t-i-\beta},
          \qquad \Delta_{s-t}(J)\leq 1.
        \]
        Let \[
          \kappa_t:=
          \begin{cases}
            \alpha D^{s-t}\log D,&\text{ if }t=1,\\
            \frac14D^{s-t-\gamma/8},&\text{ if }t\geq2.
          \end{cases}
        \]
        Since $p=D^{-1+\gamma/2}$, $\gamma<\beta$, and as $D$ is sufficiently large, we have $e(J)\leq\kappa_t$ and, for every $i\in[s-t]$, 
        \[ \frac{\Delta_i(J)}{\kappa_t}\log^{4(s-t)+2}N\leq p^i.\]
        Hence~\cref{cor:KimVu} yields that, with probability at least $1-e^{-\Omega(\log^2D)}$, we have
        \[ d_{H'^{(s)}}(Q)\le 2p^{s-t}\kappa_t.\]
        If $t=1$, then
        \[ 2p^{s-t}\kappa_t \leq \alpha'(D')^{s-t}\log D',\]
        where we used $\alpha'=4\alpha/\gamma$ and $D'=pD(1+\delta)$. If $t\geq2$,
        then
        \[2p^{s-t}\kappa_t =\frac12D^{(s-t)\gamma/2-\gamma/8} \leq (D')^{s-t-1/4}.
        \]
        A union bound over all $s,t,Q$ proves~\cref{item:RT_ConfigDeg} with high probability. 
    \end{proofclaim}
    Let $\cE$ be a ``bad'' event where $T'$ is not $(D',(D')^{1-1/4},1/4,4\alpha/\gamma)$-regular. By the union bound and~\crefrange{clm:spread_nibble_RT1}{clm:spread_nibble_RT4}, we obtain that \(\Prob{\cE}=o(1)\). Conditioning on $\cE^c$, $T'$ is $(D',(D')^{1-1/4},1/4,4\alpha/\gamma)$-regular, hence  by~\cref{thm:ForbiddenSubmatchingReserves}, $T'$ contains a perfect matching $M$, and by~\cref{prop:SubTreasury}, $M$ is then also a perfect matching of $T$. Let $S\subseteq E(G_1\cup G_2)$. If $S=\emptyset$, then we are done. Otherwise, since $M\subseteq E(G'_1\cup G'_2)$, we have
    \begin{align*}
        \Prob{S\subseteq M}
        \leq \Prob{S\subseteq E(G'_1\cup G'_2)\mid\cE^c}
        =\frac{\Prob{S\subseteq E(G'_1\cup G'_2)\text{ and }\cE^c}}{\Prob{\cE^c}}
        \leq\frac{p^{|S|}}{1-o(1)}
        \leq 2p^{|S|}
        \leq \parens*{\frac{1}{D^{1-\gamma}}}^{|S|},
    \end{align*}
where we used $D$ sufficiently large, $|S|\geq1$ and $2\leq D^{\gamma|S|/2}$.
\end{nonlateproof}

Observe that, by the remark that follows~\cref{def:RegularTreasury}, it is sufficient for the treasury $T$ to be $(D,D^{1-\beta},\beta,\alpha)$ for some $\alpha\leq \alpha_0$ for~\cref{thm:ForbiddenSubmatchingReservesSpread} to apply.


\section{Proof of the completion Theorem}\label{sec:ProofCompletion}

We are now ready to prove our completion result,~\cref{thm:MinDegree}. To do so, using ideas already present in~\cite{DKPIV}, we combine the spread distribution over absorbers obtained in~\cref{thm:MainLatinAbsorbers} with the spread distribution obtained during the Forbidden submatching step in~\cref{thm:ForbiddenSubmatchingReservesSpread}.
.


\begin{lateproof}{thm:MinDegree}
    
    Let $g\geq4$ and set $\gamma:=\frac1g$. Let $C=C_0$ be given by~\cref{lem:RegBoost}, and let $\xi=C^{-8}$. Let $\kappa\in(0,1)$ be any constant such that $(1+\xi)\kappa>1$. Let $\beta'>0$ such that $\beta'<\min\set*{\frac14,\frac1{16g}}$, and let $\alpha'>0$ be such that~\cref{thm:ForbiddenSubmatchingReservesSpread} holds with parameters $c=r=3$, $\beta'$, and $\gamma$. Let $\beta:=4\beta'$, and let $\alpha,\sigma,\varepsilon\in(0,1)$ be arbitrarily small reals, in particular ensuring that
    \begin{align}
         (1-\xi)&<(1-8\alpha),\label{eq:apply_boosting_alpha}\\
         (1-\xi)&<(1-12\sigma),\label{eq:apply_boosting_sigma}\\
         \xi^{1/8}&<(1-48\alpha),\label{eq:apply_boosting_small}\\
         8\alpha &< \alpha'\cdot (\kappa/4)^{g+1}.\label{eq:alpha_small},
    \end{align}
    and such that~\cref{thm:reserve,thm:MainLatinAbsorbers} hold with parameters  $g,4\alpha$, $2\beta'$ and $\zeta:=\varepsilon/2$. Let $n$ be a large enough integer, and let $p\in(0,1)$ such that $p(n)$ tends arbitrarily quickly to $1$. Let $G$ be a $K_{1,1,1}$-divisible spanning subgraph of $K_{n,n,n}$ with $\delta(G)\ge (2-\sigma)n$ and vertex-partition $V^1,V^2,V^3$, and such that there exists a subtreasury $T$ of ${\rm Treasury}^g(G,G,K_3)$ that is $(n,\alpha n,\beta,\alpha)$-regular and $g$-abundant. By~\cref{thm:reserve}, there exists a spanning subgraph $X\subseteq G$ such that the following holds.
    \begin{enumerate}
        \item $X$ is edge-balanced.
        \item $\Delta(X)\leq 2n^{1-\varepsilon}\leq n^{1-\zeta}$.
        \item For every $i\in[3]$, and every vertex $v\in V^i$, we have
        \(\size*{d_{V^{i-1}}(v)-d_{V^{i+1}}(v)}\leq n^{2/3}.\)
        \item $\tilde{T} = {\rm Saving}(T,X)$ is $(n,4\alpha n,\beta,\alpha)$-regular.
    \end{enumerate}

    Observe first that the configuration hypergraph of $\tilde{T}$ and $T$ are equal, hence that $\tilde{T}$ is trivially also $g$-abundant. We therefore obtain that $T$ and $\tilde{T}$ are both $g$-abundant and $(n,4\alpha n,\beta,4\alpha)$-regular (by monotonicity of regularity).
    
    Let $\cA$ be the family of Latin-omni-absorbers $A\subseteq G$ for $X$ in $T$ with decomposition family $\cF_A$ and decomposition function $\cQ_A$, with $\Delta(A)\leq n^{1-\eta}$, and such that ${\rm Proj}^g(\tilde{T},A,X)$ contains a subtreasury $\widehat{T}$ that is $\parens*{n,8\alpha n,2\beta',8\alpha}$-regular. By~\cref{thm:MainLatinAbsorbers}, there exists a probability distribution over $\cA$ such that for every $K_3$-packing $S$ of $G$, 
      \begin{equation}\label{eq:ProofCompletion_SpreadAbs}
        \Prob{\exists L\in\div{X}\colon S \subseteq \cQ_A(L)} \leq \left(\frac{1}{n^{1 - 3/(2g)}}\right)^{|S|} .
      \end{equation}
    
    For each Latin-omni-absorber $A\in\cA$, let $J_A := G\setminus (X \cup A)$, and let $G_1,G_2,H$ be hypergraphs such that $\widehat{T}=(G_1,G_2,H)$. Observe that $V(G_1)=J_A$ while $G_2$ is a bipartite hypergraph with parts $J_A$ and $X$. We now prove that the requirements of~\cref{lem:RegBoost}, the regularity boosting lemma from~\cite{KSSS2023substructures}, are satisfied by the graph $J_A$ and the collection of triangles represented by $E(G_1)$.
    \begin{enumerate}[label=(\alph*)]
        \item By regularity of $\widehat{T}$, for all vertices $e\in J_A$  we have 
        \((1-8\alpha)n\leq d_{G_1}(e)\leq n.\)
        Using $(1+\xi)\kappa>1$, $\kappa\in(0,1)$, $p=1-o(1)$, and~\cref{eq:apply_boosting_alpha}, we obtain that for $n$ large enough, every edge of $J_A$ is in $(1\pm\xi)p^2\kappa n$ copies of $K_3$ in $G_1$.

        \item For $j\in[3]$, let $Q$ be a set of at most $6$ vertices in $V^{j-1}\cup V^{j+1}$. With $\Delta(X),\Delta(A)=o(n)$ and $\delta(G)\geq(2-\sigma)n$, we have $\delta(J_A)\geq (2-2\sigma)n$. Therefore there are at least $(1-12\sigma)n$ common neighbors of $Q$ in $V^j$.
        Using $p=1-o(1)$ and~\cref{eq:apply_boosting_sigma}, we obtain that for $n$ large enough, there are $(1\pm \xi)p^{|Q|}n$ common neighbors of $Q$ in $V^j$.

        \item For $j\in[3]$, let $Q$ be a set of at most $6$ edges between $V^{j-1}$ and $V^{j+1}$. 
        Again by regularity of $\widehat{T}$, for every vertex $e\in V(G_1)=E(J_A)$ we have $(1-8\alpha)n\leq d_{G_1}(e)\leq n$. Therefore there are between $(1-48\alpha)n$ and $n$, hence by~\cref{eq:apply_boosting_small} between $C^{(-1)}p^{|V(Q)|}n$ and $Cp^{|V(Q)|}n$ vertices $u\in V^{j}$ which form a triangle in $G_1$ with every edge in $Q$.

        \item As $G$ and $A$ are $K_{1,1,1}$-divisible and $X$ is edge-balanced, $J_A$ is also edge-balanced. Then for every $j\in[3]$ and every $v\in V^j$, observe that $G$ and $A$ being  $K_{1,1,1}$-divisible implies that in $G\setminus A$ we have $d_{V^{j-1}}(v)=d_{V^{j+1}}(v)$. Therefore, by properties of $X$, we conclude that in $J_A$ we have $\size*{d_{V^{j-1}}(v)-d_{V^{j+1}}(v)}\leq n^{2/3}$.
        
    \end{enumerate}
    By~\cref{lem:RegBoost}, there exists a subhypergraph $G'_1$ of $G_1$ such that every edge $e\in J_A$ is in $(1\pm n^{-1/4})p^2\kappa n/4$ hyperedges of $G'_1$, and therefore in $(1\pm 2n^{-1/4})\kappa n/4$ hyperedges of $G'_1$. Let $T_A$ be the subtreasury of $\widehat{T}$ defined by 
    \[H':=H[E(G'_1)\cup E(G_2)],\quad T_A:=(G'_1,G_2,H').\]
    We now prove that the treasury $T_A$ is $(D',\sigma',\beta',\alpha')$-regular, where
    \[D':=(1+2n^{-1/4})\frac{\kappa n}{4},\qquad \sigma'=4n^{-1/4}\cdot\frac{\kappa n}{4}.\]    
    
    First, we directly obtain from above that every edge of $J_A$ has degree at most $D'$ in $G'_1$ and degree at least $D'-\sigma'$ in $G'_1$.
    We then check the degrees in $G_2$. For every edge $e\in X$, we deduce from $\Delta(X)\leq 2n^{1-\varepsilon}$ that $d_{G_2}(e)\leq 2\Delta(X)\leq 4n^{1-\varepsilon}\leq D'$. Then, for every edge $e\in J_A$, it follows from the regularity of $\widehat{T}$ that $d_{G_2}(e)\geq n^{1-8\alpha}\geq (D')^{1-\alpha'}$.

    For codegrees, we know from the regularity of $\widehat{T}$ that $G_1\cup G_2$ has codegree at most $n^{1-2\beta'}$. With $G'_1,H'$ being subgraphs of $G_1,H$, respectively, using  $n\geq D'\geq \frac{\kappa n}{4}$ we obtain that $G_1\cup G_2$ has codegree at most 
    \(n^{1-2\beta'}\leq (4\kappa^{-1}D')^{1-2\beta'}\leq (D')^{1-\beta'}.\)
    Similarly the maximum 2-codegree of $G'_1 \cup G_2$ with $H'$ and the maximum common 2-degree of $H'$ with respect to  $G'_1 \cup G_2$ are at most $(D')^{1-\beta'}$.

    For all $2\leq t<s\leq g$, we also deduce from the regularity of $\widehat{T}$ that
    \[\Delta_t(H^{'(s)})\leq n^{s-t-2\beta'}\leq (4\kappa^{-1}D')^{s-t-2\beta'}\leq (D')^{s-t-\beta'},\]
    and for all $2\leq s\leq g$,
    \begin{align*}
    \Delta_1(H^{'(s)})\leq 8\alpha\cdot n^{s-1}\log n 
    &\leq 8\alpha\cdot(4\kappa^{-1})^{s-1}\cdot (D')^{s-1} \log (4\kappa^{-1}D') \\
    &\leq 8\alpha\cdot(4\kappa^{-1})^{g+1}\cdot (D')^{s-1}\log D'\\
    &\leq \alpha'\cdot (D')^{s-1}\log D'.    
    \end{align*}

    We obtain that $T_A$ is $(D',\sigma',\beta',\alpha')$-regular. Finally with $\beta'\leq \frac{1}{4}$ we conclude that $\sigma'<(D')^{1-\beta'}$, and therefore $T_A$ is $(D',(D')^{1-\beta'},\beta',\alpha')$-regular. With $n$ and therefore $D'$ large enough, by choice of $\alpha',\gamma$, we can apply~\cref{thm:ForbiddenSubmatchingReservesSpread} to obtain a probability distribution over perfect matchings $M_A$ of $T_A$ such that, for every $S\subseteq E(G'_1\cup G_2)$ and every $A'\in \cA$
    \begin{equation}\label{eq:ProofCompletion_SpreadMatch}
        \Prob{S \subseteq M_A~\mid~A=A'} \leq \left(\frac{1}{{D'}^{1 - 1/g}}\right)^{|S|} . 
    \end{equation}
    
    By~\cref{prop:SubTreasury}, any such packing $M_A$ is also a perfect matching of $\widehat{T}$, and therefore a perfect matching of ${\rm Proj}^g(\tilde{T},A,X)$. Let  $L_A:=(G\setminus A)\setminus\bigcup M_A$.
    By~\cref{prop:FindPM}, we have $L_A\in\div{X}$ and $M_A\cup \cQ_A(L_A)$ is a perfect matching of $T$. Fix a $K_3$-packing $S$ of $G$. Since $L_A\in\div{X}$, for every $S'\subseteq S$ we have
    \[\set*{S'\subseteq M_A \text{ and } S\setminus S'\subseteq\cQ_A(L_A)}
    \subseteq
    \set*{S'\subseteq M_A  \text{ and } \exists L\in\div{X}\colon S\setminus S'\subseteq\cQ_A(L)},\]
    hence
    \begin{align*}
    \Prob{S\subseteq M_A\cup\cQ_A(L_A)}
    &\leq \sum_{S'\subseteq S} \Prob{S'\subseteq M_A \text{ and } \exists L\in\div{X}\colon S\setminus S'\subseteq\cQ_A(L)}
    \end{align*}
    Observe that for every $K_3$-packing $R$ of $G$, the event \(\set*{\exists L\in\div{X}\colon R\subseteq\cQ_A(L)}\) depends only on the chosen absorber $A$ and its decomposition function $\cQ_A$. Therefore we have
    \begin{align*}
    \Prob{S\subseteq M_A\cup\cQ_A(L_A)}&\leq \sum_{S'\subseteq S} \sum_{\substack{A'\in\cA\\ \exists L\in\div{X},~S\setminus S'\subseteq\cQ_{A'}(L)}}
        \Prob{S'\subseteq M_A\mid A=A'}\Prob{A=A'}\\    
    &\leq \sum_{S'\subseteq S} \parens*{\frac{1}{(D')^{1-1/g}}}^{|S'|}\cdot \Prob{\exists L\in\div{X}\colon S\setminus S'\subseteq\cQ_A(L)}  &\text{ By }~\cref{eq:ProofCompletion_SpreadMatch}\\      
    &\leq \sum_{S'\subseteq S} \parens*{\frac{1}{(D')^{1-1/g}}}^{|S'|}
    \cdot \parens*{\frac{1}{n^{1-3/(2g)}}}^{|S\setminus S'|}&\text{ By }~\cref{eq:ProofCompletion_SpreadAbs}\\
    &=\parens*{\frac{1}{(D')^{1-1/g}} + \frac{1}{n^{1-3/(2g)}}}^{|S|}.
\end{align*}

    Using $D'=\Omega(n)$, we obtain
    \begin{align*}
      \Prob{S\subseteq M_A\cup \cQ_A(L_A)}
      \leq \left(\frac{O(1)}{{n}^{1 - 3/(2g)}}\right)^{|S|}
      \leq \left(\frac{1}{{n}^{1 - 2/g}}\right)^{|S|},
    \end{align*}    
    as desired.
\end{lateproof}


\section{Spread And High Girth Absorption}\label{sec:HighGirthAbsorption}

It remains to prove our absorber theorem,~\cref{thm:MainLatinAbsorbers}. Leaving aside for a moment the question of spreadness, the goal of~\cref{thm:MainLatinAbsorbers} is to find, for a given treasury $T$ with configuration hypergraph $H$, an efficient omni-absorber $A$ in $T$, while ensuring that most of the triangle-packings in the remainder of the graph are jointly with $A$ still $H$-avoiding. This generalizes the notion of ``high collective girth'' present in the work of Delcourt and Postle on the ``High Girth Existence Conjecture''~\cite{DPII}. In their article, there are no pre-existing constraints on the allowed packings beyond forbidding low-girth configurations (while here we need to take into account the packing given by the pre-seeding step). Therefore the concept of absorbers with high collective girth is equivalent to absorbers existing in ${\rm Treasury}^g(G,G,K_3,X)$.

To achieve our goal, we view the absorber from~\cref{thm:LatinAbsorberThm} as a template, and we embed boosters for each clique in its decomposition family $\cF$, using the abundance of the treasury to obtain a ``boosted absorber'' in $T$. This is possible only because the absorbers from~\cref{thm:LatinAbsorberThm} are $C$-refined, meaning that every edge is only in a constantly-many triangles of $\cF$. 

If $A$ is a Latin-absorber for $X$ with decomposition family $\cF$ and decomposition function $\cQ$ and $(B, \cBon, \cBoff, R)$ is a Latin-booster rooted at $R \in \cF$ such that $B$ is edge-disjoint from $A\cup X$, then $A \cup B$ is also a Latin-absorber for $X$ with decomposition family $\cF \cup \cBon \cup \cBoff \setminus\{R\}$ and decomposition function 
\begin{equation*}
    L \mapsto \left\{\begin{array}{l l}
      \cQ(L) \cup \cBon \setminus \{R\} & \text{if } R \in \cQ(L)\\
      \cQ(L) \cup \cBoff & \text{otherwise.}
    \end{array}\right.
\end{equation*}
In this way, we can replace $R$ with the cliques of the two decompositions $\cBon$ and $\cBoff$ which may have more desirable properties than $R$; in particular for us, they will be in $T$ and be well-spread. 
Using the fact that $T$ is abundant, we can select boosters in $T$, ensuring that each booster is individually $H$-avoiding. We then need to prove that we can select boosters such that there are collectively still in $T$, that is that they collectively avoid any edge from $H$.

\subsection{Omni Booster}

We recall that, to prove~\cref{thm:MainLatinAbsorbers}, we add a private $(R,g)$-sphere for each clique in the decomposition family of the omni-absorber; we do this simultaneously and randomly such that with high probability this is done in a spread manner. To that end, we define the following omni-booster.

\begin{definition}[Latin-Omni-Booster]\label{def:OmniBooster}
Let $G\subseteq K_{n,n,n}$ be a graph, $X$ be a spanning subgraph of $G$, and let $A$ be a Latin-omni-absorber for $X$ with decomposition family $\cF$. We say that a graph $B\subseteq G$ is a \emph{Latin-omni-booster} for $A$ and $X$ in $G$ if $B$ is edge-disjoint from $A\cup X$ and $B$ is the edge-disjoint union of a set of graphs $\cB:= (B_F: F\in\cF)$ (called the \emph{booster family} of $B$) where for each $F\in \cF$, $B_F$ is a Latin-booster rooted at $F$.
The \emph{matching set} of $\cB$ is
$$\cM(\cB):= \set*{ M \in \prod_{F\in \mathcal{F}} \big\{ (B_{F})_{{\rm on}},~ (B_{F})_{{\rm off}} \big\} : M \text{ is a matching of } {\rm Design}(X\cup A\cup B,K_3)}.$$
If $g\geq 4$ and $\tilde{T}={\rm Saving}(T,X)$ for some bankrupted subtreasury of ${\rm Treasury}^g(G,G,K_3)$, we say that $B$ is in $T$ if every $M\in \cM(\cB)$ is in $T$.
\end{definition}

The next definition, codifying treasuries and the underlying graphs, drastically shortens some of the following theorem and lemma statements.

\begin{definition}[$K_3$-sponge]\label{def:sponge}
We say a triple $\cS=(G,A,X,n)$ is a \emph{$K_3$-sponge} if $X$ is a spanning subgraph of $G\subseteq K_{n,n,n}$ and $A\subseteq G\setminus X$ is a Latin-omni-absorber for $X$. We will also write $\cS=(G,A,X,n,\cF)$ where $\cF$ is the decomposition family of $A$ if we desire to specify the decomposition family.

We say $\cS$ is 
\begin{itemize}
    \item \emph{$a$-bounded} if $\Delta(A\cup X)\leq \frac{n}{\log^a n}$,
    \item \emph{$C$-refined} if $A$ is $C$-refined.    
\end{itemize}
\end{definition}

An omni-booster yields a natural omni-absorber decomposition function, called the canonical decomposition, as introduced in~\cite{DPII}. 
 
\begin{proposition}\label{prop:CanonicalBoost}
Let $(G,A,X,n,\cF)$ be a $K_3$-sponge and let $\cQ_A$ denote the decomposition function of $A$. If $B$ is a Latin-omni-booster for $A$ and $X$ in $G$ with booster family $\mathcal{B}=(B_F: F\in\mathcal{F})$, then $A\cup B$ is a Latin-omni-absorber for $X$ with decomposition family $\bigcup_{F\in \mathcal{F}} (\mathcal{B}_F)_{{\rm on}}\cup (\mathcal{B}_F)_{{\rm off}}$ and decomposition function
$$\mathcal{Q}_{A\cup B}(L):= \bigcup_{F\in \mathcal{Q}_A(L)} (\mathcal{B}_F)_{{\rm on}}~\cup \bigcup_{F \in \mathcal{F}\setminus \mathcal{Q}_A(L)} (\mathcal{B}_F)_{{\rm off}}.$$    
We refer to $A\cup B$ with decomposition function $\mathcal{Q}_{A\cup B}$ as the \emph{canonical omni-absorber} for $A$ and $B$.
\end{proposition}

To prove~\cref{thm:MainLatinAbsorbers}, it behooves us to prove that $A\cup B$ is in some treasury $T$, and therefore that $\mathcal{Q}_{A\cup B}(L)$ is $H$-avoiding for every $L\in\div{X}$. As this is highly dependent on the structure of $A$ and $X$, it is instead easier to prove that all triangle-packings that are the union of `on'/`off' decompositions are in $T$. To that end, we make the following definition.

\begin{definition}[Projection of Omni-Booster]\label{def:OmniBoosterProj}
Let $g\geq 4$ be an integer, $(G,A,X,n,\cF)$ be a $K_3$-sponge and let $\tilde{T}={\rm Saving}(T,X)$ where $T$ is a bankrupted subtreasury of ${\rm Treasury}^g(G,G,K_3)$. Let $B$ a Latin-omni-booster for $A$ and $X$ in $T$, with booster family $\mathcal{B}=(B_F:F\in\cF)$. We define the \emph{projection treasury} of $B$ into $\tilde{T}$ as:
\[{\rm Proj}^g(\tilde{T},A,B,X):= {\rm Spending}(\tilde{T},A\cup B,X) \perp \mathcal{M}(\cB).\]
\end{definition}

The following proposition relates the treasury of an omni-booster and its canonical omni-absorber.

\begin{proposition}\label{prop:BoosterSubTreasury}
Under the assumptions of~\cref{def:OmniBoosterProj}, we have that ${\rm Proj}^g(\tilde{T},A,B,X)$ is a subtreasury of ${\rm Proj}^g(\tilde{T},A\cup B,X)$.
\end{proposition}
\begin{proof}
Let $\cB$ be the booster-family of $B$, and let $\cM(\cB)$ be the matching set of $\cB$, used for the projection ${\rm Proj}^g(\tilde{T},A,B,X)$. Let $\cM(A\cup B)$ be the matching set for the canonical omni-absorber $A\cup B$ as used for the projection ${\rm Proj}^g(\tilde{T},A\cup B,X)$. Observe that $\cM(A\cup B)\subseteq \cM(\cB)$. By definition of common projection, projecting by a family yields a subtreasury of the projection by subfamily. Therefore
${\rm Spending}(\tilde T,A\cup B,X)\perp\cM(\cB)$ is a subtreasury of ${\rm Spending}(\tilde T,A\cup B,X)\perp\cM(A\cup B)$ as desired.
\end{proof}

\subsubsection{Vertex-spread vs triangle-spread}

We are now almost ready to state an omni-booster variant of the omni-absorber theorem,~\cref{thm:MainLatinAbsorbers}. Our goal is that by proving the existence of (a spread distribution over) omni-booster in the given treasury, we obtain the desired (spread distribution over) omni-absorbers via their canonical decompositions. We separate the spread argument into two steps. We first construct a {\em vertex-spread} distribution on the embeddings of the spheres. We then use~\cref{lem:VertexToTriangleSpread} to convert vertex-spreadness into spreadness of the triangles appearing in their on/off decompositions. We will then use the fact that triangles are spread-out inside a given $(R,g)$-sphere, to deduce the ``{\em triangle-spreadness}'' of~\cref{thm:MainLatinAbsorbers}. To achieve such a result, we focus specifically on rooted boosters that are isomorphic to $(R,g)$-spheres. 

\begin{definition}[Sphere-type]
    We say that a Latin omni-booster $B$ with booster-family $\cB$ is {\em $g$-sphere-type} if every rooted-booster in $\cB$ rooted at some graph $R$ is a copy of an $(R,g)$-sphere. 
\end{definition}

\begin{definition}[Vertex-spread]\label{def:vertexspread}
Let $(G,A,X,n,\cF)$ be a $K_3$-sponge, and let $\phi$ be a probability distribution over $g$-sphere-type Latin-omni-boosters for $A$ and $X$ in $G$.For each $F\in\cF$, fix a $(F,g)$-sphere $\Sigma_F$. An omni-booster with booster family $\cB=(B_F\colon F\in\cF)$ defines, for every $F\in\cF$,  an embedding $\psi_F:\Sigma_F\to G$ whose image is $B_F$.

We say that $\phi$ is \emph{$p$-vertex-spread} if, for every choice
of sets $U_F\subseteq V(\Sigma_F)\setminus V(F)$ with $F\in\cF$, and mappings $f_F:U_F\to V(G)$, we have 
\[\mathbb{P}_{\phi}\brackets*{ \psi_F(x)=f_F(x)  \text{ for every }F\in\cF\text{ and }x\in U_F}
    \leq p^{\sum_{F\in\cF}|U_F|},\]
that is, if fixing the images of $t$ non-root vertices costs a factor of at most $p^t$.
\end{definition}

\begin{lem}[Vertex-spread to triangle-spread for spheres]\label{lem:VertexToTriangleSpread}
For every integers $a\geq 1$ and $g\geq 2$, and reals $C,K\geq 1$,
there exist $c>0$ and $n_0$ such that the following holds for every $n\geq n_0$.

Let $(G,A,X,n,\cF)$ be a $C$-refined $a$-bounded $K_3$-sponge, and let $\phi$ be a $p$-vertex-spread probability distribution over $g$-sphere-type Latin-omni-boosters for $A$ and $X$, where $p\leq K/n$. Let $B$ be chosen according to $\phi$, and let $\cB=(B_F\colon F\in\cF)$ be its booster family. Then, for every $K_3$-packing~$S$ of~$G$,
    \[\mathbb{P}_{\phi}\brackets*{S\subseteq M\text{ for some }M\in\cM(\cB)}
\leq \parens*{c\cdot p^{1-1/g}}^{|S|}.\]
\end{lem}

\textit{Proof Intuition:}
For a fixed packing $S$, we encode the event that $S$ is contained in some $M\in\cM(\cB)$ by specifying, for each triangle of $S$, the sphere and the on/off decomposition in which it appears. We first analyse the triangles assigned to a single sphere. After fixing one selected triangle, the root and the remaining $k-1$ selected triangles can be exposed so that all but at most $k/g$ of them introduce a new vertex. Each newly exposed vertex costs a factor $p$ by vertex-spreadness, while the possible choices for its image contribute at most a compensating factor $n$; since $np\leq K$, each non-exceptional triangle therefore contributes $O(p)$. The choices for the root are controlled similarly by the $C$-refinement of the absorber. Thus $k$ triangles occurring in one sphere have total weight at most
\[O\parens*{p^{(1-1/g)k}}.\]

To pass from one sphere to all spheres, we fix a triangle $T\in S$, sum over all possibilities for the sphere containing $T$, and then apply the same argument inductively to the remaining triangles. Summing the resulting geometric series gives
\[ \Prob{S\subseteq M\text{ for some }M\in\cM(\cB)} \leq \parens*{c\cdot p^{1-1/g}}^{|S|}. \]

The loss of $p^{1/g}$ comes precisely from the occasional closing step in the sphere.

\begin{nonlateproof}{lem:VertexToTriangleSpread}
    Fix a $K_3$-packing $S$ of $G$. The result is trivially true for $S=\emptyset$, therefore assume that $|S|\geq1$. If at least one triangle of $S$ contains an edge of $A\cup X$, let $T$ be such an arbitrary triangle, and let $T\in S$ arbitrarily otherwise.

    Let $(\Sigma,\Sigma_{\rm on},\Sigma_{\rm off},R_0)$ be an $(R_0,g)$-sphere, with root triangle $R_0$. For $k\in[2g]$, a \emph{local $k$-witness at $T$} consists of
    \begin{itemize}
        \item a choice $\ell\in\{{\rm on},{\rm off}\}$;
        \item a set $\cT$ of $k$ triangles of $\Sigma_\ell$, and a distinguished triangle $U\in\cT$;
        \item an embedding $\psi$ of $R_0\cup\bigcup_{Q\in\cT}Q$ into $G$ such that 
        \[\psi(R_0)\in\cF,\qquad \psi(U)=T, \qquad \psi(Q)\in S \text{ for every }Q\in\cT,\]
        and
        \[\psi\big(E(Q)\setminus E(R_0)\big)\cap E(A\cup X)=\emptyset \qquad\text{for every }Q\in\cT\]
        The last condition records that, in an actual omni-booster, every edge of a sphere outside its root belongs to the booster graph and is therefore disjoint from $A\cup X$.
    \end{itemize}
    Informally, the local $k$-witness of $T$ specifies one possible way in which $T$, together with $k-1$ other triangles of $S$, could all be embedded in the same sphere in the random omni-booster. Let
    \[W(\cT):=\parens*{\bigcup_{Q\in\cT}V(Q)}\setminus V(R_0),\]
    and let $t:=|W(\cT)|$. Thus \(W(\cT)\) is the set non-root vertices of the sphere intersecting the triangles
    of $\cT$. Once the root \(R=\psi(R_0)\) has been chosen, the local witness requires that every $x\in W(\cT)$ is mapped to the prescribed vertex $\psi(x)\in V(G)$. No condition is imposed on the non-root vertices of the sphere that do not occur in $\cT$. Hence, by $p$-vertex-spreadness, the probability that the random sphere agrees with this local witness is at most $p^t$.

    We first bound the total weight of all local $k$-witnesses at $T$. Consider one choice of $\ell\in\{{\rm on},{\rm off}\}$, $\cT$, and $U$. If some triangle of $\cT$ contains an edge of $R_0$, then its image contains an edge of $A\cup X$. By the choice of $T$, the triangle $T=\psi(U)$ also contains an edge of $A\cup X$; the last condition in the definition of a local witness then implies that $U$ contains an edge of $R_0$. 
    Write the vertices of the sphere $(\Sigma,\Sigma_{\rm on},\Sigma_{\rm off},R_0)$ as $u,v,b_1,\ldots,b_{2g}$ as in~\cref{def:spheres}.
    \begin{claim}\label{clm:SphereWitnessExceptions}
        There exists an ordering $Q_1,\ldots,Q_k$ of $(\cT\setminus\{U\})\cup\{R_0\}$ such that
        \[h:=\Big|\Big\{ i\in[k]\colon V(Q_i)\subseteq V(U)\cup\bigcup_{j<i}V(Q_j)\Big\}\Big|\leq1,\]
        with equality only if $k\geq2g-1$.        
    \end{claim}
    \begin{proofclaim}
        For the off-decomposition, the triangles of $\Sigma_{\rm off}$ together with $R_0$ follow the cycle
        \[b_1b_2,\ b_2b_3,\ldots,b_{2g-1}b_{2g},\ b_{2g}b_1. \]
        Starting from $U$, expose each interval of selected triangles from one of its ends. Every exposed element then introduces a new $b_i$, except possibly the last member closing the entire cycle, which requires $k=2g-1$. For the on-decomposition $\Sigma_{\rm on}$, expose $R_0$ first. If the triangle $(ub_{2g}b_1)$ belongs to $\cT\setminus\{U\}$, expose it next. The assumption on $U$ ensures that these members introduce a new vertex. The remaining triangles correspond to intervals of the path $b_1b_2,\ b_2b_3,\ldots,b_{2g-1}b_{2g}$, and may again be exposed from their ends. The only possible exception is the last member completing this entire path, which requires $k\geq2g-1$. 
    \end{proofclaim}

    In particular, we can assume that $h\leq k/g$. We now count the embeddings $\psi$. There are at most $3!=6$ isomorphisms mapping $U$ onto $T$. Since $S$ is a $K_3$-packing, for every $v\in V(G)$ we have
    \[|S|\leq n^2,\qquad |\{F\in S:v\in V(F)\}|\leq n, \]
    and every edge of the host graph $G$ is contained in at most one triangle of $S$. Similarly, since $A$ is $C$-refined, by~\cref{def:LatinAbsorbers} for every $v\in V(G)$ and every $e\in E(G)$ we have
    \[ |\cF|=O(n^2),\qquad  |\{F\in\cF:v\in V(F)\}|\leq Cn,\qquad |\{F\in\cF:e\in E(F)\}|\leq C.\]

    It follows that, given $U,Q_1,\ldots,Q_{i-1}$, if the next triangle $Q_i$ introduces $j\geq1$ new vertices, there are at most $O(n^{j-1})$ choices for their images. If it introduces no new vertex, there are only $O(1)$ choices.

    Recall that $t=|W(\cT)|$ denotes the number of non-root vertices contained in $\cT$ (and depends on the chosen witness). Since
    \[
        \left|V(R_0)\cup\bigcup_{Q\in\cT}V(Q)\right|=t+3
    \]
    and the three vertices of $U$ are already fixed, the members
    $Q_1,\ldots,Q_k$ introduce altogether exactly $t$ new vertices. By~\cref{clm:SphereWitnessExceptions}, precisely $k-h$ of these members introduce at least one new vertex. Hence, for some constant $d_0=d_0(g,C)$, the number of embeddings $\psi$ for the fixed choice of $\ell$, $\cT$, and $U$ is at most
    \[d_0^k n^{t-k+h}.\]
    The sum of their probability bounds is therefore at most
    \[
        d_0^k n^{t-k+h}p^t
        =d_0^k(np)^{t-k+h}p^{k-h}
        \leq d_1^k p^{(1-1/g)\cdot k},
    \]
    for some $d_1=d_1(g,C,K)$, using $np\leq K$, $t\leq 2g-1$ (since any sphere contains $2g-1$ non-root vertices), and $h\leq k/g$. Since there are only constantly many choices of $\ell$, $\cT$, and $U$, there exists $d=d(g,C,K)$ such that
    \begin{equation}\label{eq:LocalSphereWitness}
    \sum_{\omega\in\cW_k(T)}p^{t(\omega)} \leq \parens*{dp^{1-1/g}}^k,        
    \end{equation}
    where $\cW_k(T)$ denotes the family of all local $k$-witnesses at $T$ and $t(\omega):=|W(\cT)|$ for $\omega=(\ell,\cT,U,\psi)\in\cW_k(T)$. 
    
    We now move from local witnesses, which concern a single sphere, to witnesses for the whole packing $S$, that is we now use the one-sphere estimate~\eqref{eq:LocalSphereWitness} to bound witnesses involving all the triangles of $S$. A witness for $S\subseteq M$ for some $M\in\cM(\cB)$ assigns the triangles of $S$ to spheres, chooses the on- or off-decomposition in each sphere, and specifies the images of all non-root vertices occurring in those triangles. If such a witness specifies altogether $t$ formal non-root vertices, then its probability is at most $p^t$. Let $w(S)$ denote the sum of $p^t$ over all witnesses for $S$, where $t$ is the number of formal non-root vertices specified by the witness.

    \begin{claim}
    \[w(S)\leq\parens*{2d p^{1-1/g}}^{|S|}.\]        
    \end{claim}
    \begin{proofclaim}
        We prove this statement by induction on $|S|$. The base case $S=\emptyset$ is trivially true. For $|S|\geq 1$, choose $T\in S$ as above. In any witness, let $S_T$ be the set of triangles of $S$ assigned to the same sphere as $T$, and let $k:=|S_T|$. The restriction of the witness to this sphere is a local $k$-witness at $T$. After deleting the part of the witness corresponding to this sphere, the remaining part is a witness for $S\setminus S_T$. Thus, by~\eqref{eq:LocalSphereWitness} and the induction hypothesis,
        \[ w(S) \leq
            \sum_{k=1}^{\min\{2g,|S|\}} \parens*{d p^{1-1/g}}^{k} \cdot \parens*{2d p^{1-1/g}}^{|S|-k}
            \leq \parens*{2d p^{1-1/g}}^{|S|}\cdot \sum_{k=1}^{2g}2^{-k}
            \leq \parens*{2d p^{1-1/g}}^{|S|}.\qedhere\]
    \end{proofclaim}

    Every occurrence of the event in the statement yields such a witness, and hence
    \[\mathbb{P}_{\phi}\brackets*{ S\subseteq M\text{ for some }M\in\cM(\cB)} \leq w(S)
        \leq \parens*{2d\cdot p^{1-1/g}}^{|S|},\]
    as desired by taking $c:=2d$.
\end{nonlateproof}

\subsubsection{Omni-booster theorem}

\begin{thm}[Omni-Booster Theorem]\label{thm:HighGirthOmniBooster}
For all integers $g\geq 4$, $C\geq 1$, and reals $\alpha\in(0,1/2)$ and $\beta\in(0,1/(8g))$, there exist integers $a,n_0\ge 1$ and a real $\sigma\in(0,1)$ such that the following holds for all $n\ge n_0$.

Let $(G,A_0,X,n,\cF)$ be a $(ga)$-bounded and $C$-refined $K_3$-sponge with $\delta(G)\geq(2-\sigma)n$, such that there exists a subtreasury $T$ of ${\rm Treasury}^g(G,G,K_3)$ such that $T$ is $g$-abundant and both $T$ and $\tilde{T}={\rm Saving}(T,X)$ are $\parens*{n,~\alpha n,~2\beta,~\alpha}$-regular, and let 
$\Delta:= \max\set*{ \Delta(A_0\cup X),~\sqrt{n}\log n}$.

Let $\cL_B$ be the family of $g$-sphere-type Latin-omni-boosters $B\subseteq G$ for $A_0$ and $X$ in $T$ with $\Delta(B)\leq a\Delta\cdot \log^{a} \Delta$, such that ${\rm Proj}^g(\tilde{T},A_0, B, X)$ contains a subtreasury that is $\parens*{n,2\alpha n,\beta,2\alpha}$-regular.

Then $\cL_B$ is non-empty and there exists an $n^{-1+3/(2g)}$-spread probability distribution over $\cL_B$. That is, if $B$ is chosen according to this distribution and $\cB$ denotes its booster family, then, for every $K_3$-packing
$S$ of $G$,
\[\Prob{\exists M\in\cM(\cB)\colon S\subseteq M}\leq n^{(-1+3/(2g))|S|}.\]
\end{thm}

Before proving the existence of such a spread-distribution over omni-boosters, we show that~\cref{thm:HighGirthOmniBooster} implies~\cref{thm:MainLatinAbsorbers}, the existence of a spread-distribution over omni-absorbers.

\begin{proof}[Proof of~\cref{thm:MainLatinAbsorbers}]\

For $g\geq 4$, $\zeta\in(0,1)$, $\alpha\in(0,1/2)$ and $\beta\in(0,1/(8g))$, let $C\geq 1$, $\eta_{abs}$ and $\sigma_{abs}\in(0,1)$ such that~\cref{thm:LatinAbsorberThm} holds, and let $a,n_0$ and $\sigma_{boost}\in(0,1)$ such that~\cref{thm:HighGirthOmniBooster} holds for the parameters $g,C,\alpha,\beta,$. Let $\sigma:=\min\{\sigma_{abs},\sigma_{boost}\}$, and $\eta:=\min\{\zeta,\eta_{abs}\}/8$.

Let $G\subseteq K_{n,n,n}$ with $\delta(G)\geq(2-\sigma)n$, and assume that $X$ is a spanning subgraph of $G$ with $\Delta(X) \leq n^{1-\eta_{abs}}$ such that there exists a bankrupted subtreasury $T$ of ${\rm Treasury}^g(G,~G,~K_3)$ where both $T$ and $\tilde{T}= {\rm Saving}(T,X)$ is $\parens*{n,~\alpha n,~2\beta,~\alpha}$-regular and $g$-abundant.

By~\cref{thm:LatinAbsorberThm}, there exists a $C$-refined Latin-omni-absorber $A_0\subseteq G$ for $X$, with decomposition family $\cF_0$, such that $\Delta(A_0)\leq n^{1-\eta_{abs}}$. Let $\Delta:=\max\{\Delta(A_0\cup X),\sqrt{n}\log n\}$. For $n$ sufficiently large, we obtain that 
\[\Delta(A_0\cup X) \leq \Delta(A_0)+\Delta(X)\leq n^{1-\eta_{abs}}+n^{1-\zeta}\leq \frac{n}{\log^{ga}n}.\]
It follows that $(G,A_0,X,n,\cF_0)$ is a $C$-refined and $(ag)$-bounded $K_3$-sponge. By~\cref{thm:HighGirthOmniBooster} there exists a $n^{-1+3/(2g)}$-spread distribution over $\cL_B$. Observe that for any $L\in\div{X}$, $\cQ_{A_0\cup B}(L)\in\cM(\cB)$, therefore there exists a probability distribution $\eta_B$ over $\cL_B$ such that for every $K_3$-packing $S$ of $G$ we have
\[\mathbb{P}_{\eta_B}\brackets*{\exists L\in \div{X}\colon S\subseteq\cQ_{A_0\cup B}(L)}\leq
\mathbb{P}_{\eta_B}\brackets*{\exists M\in \cM(\cB)\colon S\subseteq M}\leq
\parens*{\frac{1}{n^{1-3/(2g)}}}^{|S|}.\]

Let $\cA=\set*{A_0\cup B\colon B\in\cL_B}$, and let $A=A_0\cup B\in\cA$. By~\cref{prop:CanonicalBoost}, $A$ is an omni-absorber for $X$, with decomposition family $\cF_A = \bigcup_{F\in \cF_0} (\mathcal{B}_F)_{{\rm on}}\cup (\mathcal{B}_F)_{{\rm off}}$ and decomposition function
$$\mathcal{Q}_{A}(L)=\mathcal{Q}_{A_0\cup B}(L)= \bigcup_{F\in \cQ_{A_0}(L)} (\cB_F)_{{\rm on}}~\cup \bigcup_{F \in \cF_0\setminus \cQ_{A_0}(L)} (\cB_F)_{{\rm off}}.$$  
It follows that, for any $L\in\div{X}$, we have $\mathcal{Q}_{A}(L)\in\cM(\cB)$. Observe that, by definition of $\cL_B$, we have $B\in T$, hence every $M\in \cM(\cB)$ is in $T$, therefore we obtain that $A$ is in $T$.

Recall that $8\eta:=\min\{\zeta,\eta_{abs}\}$, hence $(1-4\eta)\geq1/2$ and 
\[\Delta=\max\{\Delta(A_0\cup X),\sqrt{n}\log n\}
\leq \max\{n^{1-\eta_{abs}}+n^{1-\zeta},\sqrt{n}\log n\}
\leq \max\{n^{1-4\eta},\sqrt{n}\log n\}
\leq n^{1-4\eta},
\]
We obtain that 
\[\Delta(A)\leq\Delta(A_0)+\Delta(B)\leq n^{1-\eta_{abs}}+a\Delta\log^a\Delta \leq n^{1-4\eta}+n^{1-2\eta}\leq n^{1-\eta}.\]

Finally by properties of $\cL_B$, there exists a subtreasury of ${\rm Proj}^g(\tilde{T},A_0,B,X)$ that is $(n,2\alpha n, \beta,2\alpha)$-regular. By~\cref{prop:BoosterSubTreasury}, this is a subtreasury of ${\rm Proj}^g(\tilde{T},A,X)$, and we conclude that $A\in\cL_\cA$, as defined in the statement of~\cref{thm:MainLatinAbsorbers}, hence $\cA\subseteq \cL_\cA$. 

Finally, recall that $\eta_B$ is the distribution over $\cL_B$ given by~\cref{thm:HighGirthOmniBooster}, and let $\eta_{\cA}$ be the distribution over $\cL_\cA$ induced by $A:=A_0\cup B$ with $B$ taken at random in $\cL_B$ using $\eta_B$. We obtain immediately that for every $K_3$-packing $S$ of $G$, 
\[\mathbb{P}_{\eta_{\cA}}\brackets*{\exists L\in\div{X}\colon S \subseteq \cQ_{A}(L)} = 
\mathbb{P}_{\eta_{B}}\brackets*{\exists L\in\div{X}\colon S \subseteq \cQ_{A_0\cup B}(L)} \leq\parens*{\frac{1}{n^{1-3/(2g)}}}^{|S|}.\qedhere\]
\end{proof}

\subsection{Proof of omni-booster theorem}

\begin{proof}[Proof of~\cref{thm:HighGirthOmniBooster}]
Write $T=(G_1,H)$ and $\tilde{T}=(\tilde{G}_1,\tilde{G}_2,H):={\rm Saving}(T,X)$. Let \(b:=2g-1\) be the number of non-root vertices of an $(R,g)$-sphere. We choose $\sigma>0$ arbitrarily  small and $a,n_0$ arbitrarily large throughout the proof. 

For each $F\in\cF$, let $\Omega_F$ be the family of all $(F,g)$-spheres $(B,\cBon,\cBoff,F)$ embedded in $G\setminus(A_0\cup X)$ such that every triangle in $\cBon\cup \cBoff$ is an edge of $G_1$, and both $\cBon$ and $\cBoff$ are $H$-avoiding. By standard greedy argument (embedding a rooted booster of a given size), there exist constants $c_0,C_0>0$, such that
\begin{equation}\label{eq:SphereEmbeddingCount}
    c_0 n^b\leq |\Omega_F|\leq C_0n^b
\end{equation}
for every $F\in\cF$. Moreover, after fixing the images of $t$ non-root vertices, there are at most 
\begin{equation}\label{eq:PartialSphereEmbeddingCount}
    C_0n^{b-t}
\end{equation}
members of $\Omega_F$ extending these choices. Indeed, at every step of the greedy embedding, the minimum-degree condition on $G$, the regularity of $T$, and the $(ga)$-boundedness of the sponge ensure that $\Delta(A_0\cup X)\leq \frac{n}{\log^{ga}n}=o(n)$, and therefore that we have linearly many choices.  By $g$-abundance, for each root $F$, at most
\(n^{2g-\frac32}=n^{b-\frac12}=o(n^b)\)
of these spheres fail to be $H$-avoiding.\\ 

The following closely parallels the calculation done for quantum boosters in~\cite{DPII} and for quantum pre-seeding in~\cref{sec:ProofQuantumPreSeeding}. Let $\rho:=\frac{\log^a n}{n^b}$. For every $F\in\cF$, let $\cK_F$ be the random subset of $\Omega_F$, selecting each element independently at random with probability $\rho$. Denote the resulting quantum booster collection by $\cK=(\cK_F:F\in\cF)$. We extend the notion of matching of a quantum packing from~\cref{def:QuantumPacking} by defining $\cM(\cK)$ to be the family of all matchings obtained by choosing, for each $F\in\cF$, one sphere $B_F\in\cK_F$ and one of its two decompositions, that is,
    \[\cM(\cK):= \set*{ M\in \prod_{F\in\cF} \set*{(B_F)_{\rm on},(B_F)_{\rm off}\colon B_F\in\cK_F} \colon M\text{ is a matching of }{\rm Design}(G,K_3)}.\]   
For $B\in\cK_F$, we say that $B$ is {\em available} if it can be selected for its root regardless of the subsequent choices made for the other roots, that is, if the following two conditions are satisfied:
\begin{enumerate}[label=(\roman*)]
    \item \textbf{Edge-disjointness:} $B$ is edge-disjoint from every member of $\cK_{F'}$ with $F'\neq F$; and
    \item \textbf{Compatibility with the treasury:} there do not exist $P\in H$ and $M\in\cM(\cK)$ such that
    \[  P\subseteq M,\text{ and}\qquad P\cap(\cBon\cup \cBoff)\neq\emptyset.\]
\end{enumerate}

Let ${\rm Available}(\cK_F)$ denote the family of available members of $\cK_F$, and let $T_{\cK}:= {\rm Spending}(\tilde{T},A_0,X)\perp\cM(\cK)$ be the corresponding quantum projection treasury.\\

We claim that, with probability at least some absolute constant $c_{\mathcal{E}}>0$, all of the
following hold (with $c_1>0$ some fixed constant):
\begin{enumerate}[label=(P\arabic*)]
    \item \textbf{Max degree:} $\Delta\parens*{\bigcup_{F\in\cF}\bigcup\cK_F} \leq a\Delta\log^a\Delta$,\label{eq:quantum-booster-degree}
    \item \textbf{Availability:} for every $F\in\cF$, we have $\size*{{\rm Available}(\cK_F)} \geq c_1\log^a n,$\label{eq:many-avail-boosters}
    \item \textbf{Subtreasury:} $T_{\cK}$ contains a subtreasury which is $(n,3\alpha n/2,\beta,3\alpha/2)$-regular.\label{eq:quantum-projection-regular}
\end{enumerate}

This is proved exactly as Delcourt and Postle did in
\cite[Lemmas~4.12 and~4.13]{DPII}, specialized to \(q=3\), \(r=2\),
and to the tripartite host graph. Equations~\cref{eq:SphereEmbeddingCount,eq:PartialSphereEmbeddingCount} replace their estimates for a full quantum booster. Their intrinsic argument gives~\cref{eq:quantum-booster-degree,eq:many-avail-boosters}; their extrinsic argument gives~\cref{eq:quantum-projection-regular}. The same~\cref{cor:KimVu} calculations apply: configurations contained in one sphere have already been removed using abundance, while configurations meeting at least two spheres are bounded using (RT3) and (RT4). Working in the given subtreasuries causes no further difficulty, since all upper-degree estimates are monotone under deleting available
triangles, while the required lower-degree estimates are supplied by (RT1) and (RT2). This is also the same modification of the Delcourt--Postle argument used in the proof of~\cref{thm:SpreadQuantumPreseeding}.

Let $\cE$ be the event that \cref{eq:quantum-booster-degree} to~\cref{eq:quantum-projection-regular} all hold. Conditioning on $\cE$, for every $F\in\cF$, choose $B_F$ uniformly and independently at random in ${\rm Available}(\cK_F)$, and set
\[ \cB:=(B_F:F\in\cF),\text{ and}\qquad B:=\bigcup_{F\in\cF}B_F.\]

By the definition of availability, the graphs $B_F$ are pairwise edge-disjoint, and hence $B$ is a $g$-sphere-type Latin-omni-booster for $A_0$ and $X$. Moreover, every $M\in\cM(\cB)$ belongs to $\cM(\cK)$, and condition~(ii) in the definition of availability ensures that $M$ is $H$-avoiding. Since all triangles of the two decompositions of every $B_F$ belong to $G_1$, it follows that $B$ is in $T$. We also have $\Delta(B)\leq a\Delta\log^a\Delta$. Moreover, by the $(ga)$-boundedness of the sponge, $\Delta \leq \max\set*{ \frac{n}{\log^{ga}n},\sqrt n\log n}$, hence  $a\Delta\log^a\Delta=o(n)$ and $\Delta(B)=o(n)$.\\

It remains to verify the regularity of the projection. Take the regular subtreasury of $T_{\cK}$ given by~\cref{eq:quantum-projection-regular}, and delete from its design and reserve hypergraphs every triangle that is not edge-disjoint from $B$. Since
$\cM(\cB)\subseteq\cM(\cK)$, the resulting treasury is a subtreasury of ${\rm Proj}^g(\tilde{T},A_0,B,X)$.

For every edge $e\in G\setminus(A_0\cup B\cup X)$, at most $2\Delta(B)=o(n)$ triangles containing $e$ use an edge of $B$. Therefore the degrees of the design hypergraph are impacted by at most $o(n)$. The reserve lower degrees are unchanged, and the upper-degree and configuration-degree conditions are trivially inherited. Consequently, for $n$ sufficiently large, this treasury is $(n,2\alpha n,\beta,2\alpha)$-regular, as desired, and thus every booster produced by the construction belongs to $\cL_B$.\\

We finally deal with the spread-distribution over $\cL_B$. Let \(\cR=\{B_1,\ldots,B_m\}\) be a family of spheres rooted at distinct members of $\cF$. From~\cref{eq:many-avail-boosters},
\[\Prob{\cR\subseteq\cB} \leq \frac{1}{(c_1\log^a n)^m} \Prob{\cR\subseteq\cK\mid\cE}
\leq \frac{\rho^m}{\Prob{\cE}(c_1\log^a n)^m} \leq \parens*{\frac{C_1}{n^b}}^m,\]
for some constant $C_1$. Suppose that $t$ non-root vertices of these spheres are fixed and let $m\leq t$ be the number of spheres containing at least one of these specified vertices. By~\cref{eq:PartialSphereEmbeddingCount}, there are at most $C_0^mn^{bm-t}$ ways to complete these vertices into $m$ spheres. Consequently,
\[\Prob{\text{The specified vertices have their fixed images}}
\leq C_0^m n^{bm-t} \parens*{\frac{C_1}{n^b}}^m  
\leq \parens*{\frac{C_2}{n}}^t\]
for some constant $C_2$. Hence the resulting distribution on omni-boosters is $C_2/n$-vertex-spread. Applying \cref{lem:VertexToTriangleSpread} with $p=C_2/n$ and $K=C_2$, we obtain, for every $K_3$-packing $S$,
\[\Prob{\exists M\in\cM(\cB)\colon S\subseteq M} 
\leq\brackets*{c\cdot\parens*{\frac{C_2}{n}}^{1-1/g}}^{|S|}
\leq \parens*{n^{-1+3/(2g)}}^{|S|},\]
for $n$ sufficiently large, as desired.
\end{proof}


\section{Concluding Remarks}\label{sec:ConcludingRemarks}

It would be interesting to determine if our methods extend to $K_q^r$-decompositions. The arguments used to avoid medium and large subsquares appear relatively robust, and analogous spread estimates should be sufficient for corresponding forbidden substructures. Similarly, the quantum and booster arguments should admit extensions to more general clique decompositions. The main obstacle seems to be the regularity-boosting step: the lemma we use here is specific to triangles, and an appropriate analogue would be required in the general $K_q^r$-setting. It is also not clear what the natural versions of subsquares, or sub-hypercubes, should be in this general setting.

A second natural question is to identify other properties of Latin squares that can be enforced using spreadness. Our proof uses spreadness to avoid subsquares, but~\cref{thm:main} is formulated more generally: the family of medium subsquares could be replaced by any family of packings that is unlikely to appear under an $n^{-1+\gamma}$-spread distribution, while the large subsquares could be replaced by suitable $Y$-enclosed families that are unlikely under the stronger $(1+\varepsilon)/n$-spread condition. It would be interesting to find other natural properties of Latin squares that fit into either of these frameworks.

\paragraph{Declaration of absence of AI use:} This project was initiated in late 2023 and was mostly completed by mid-2025. No generative AI tools were used in the research or preparation of this article.

\bibliographystyle{plain}
\bibliography{Hilton}

\end{document}